\documentclass[a4paper,10pt]{article}
\usepackage{geometry}
\usepackage[utf8]{inputenc}
\usepackage[T1]{fontenc}
\usepackage{lmodern,amsmath,amssymb,amsthm,mathtools,dsfont,mleftright}
\mleftright
\usepackage[colorlinks,linkcolor=black,citecolor=black,urlcolor=black]{hyperref}
\usepackage[backend=biber]{biblatex}
\AtEveryBibitem{
	\ifentrytype{book}{\clearfield{pages}}{}
}
\DeclareDelimFormat{finalnamedelim}{\ifnumgreater{\value{liststop}}{2}{\finalandcomma}{}\addspace\&\space}
\DeclareDelimFormat{namelabeldelim}{\addnbspace}

\theoremstyle{definition}
\newtheorem{setting/}{Setting}[section]
\newenvironment{setting}{
	\pushQED{\qed}\begin{setting/}}
	{\popQED\end{setting/}
}
\theoremstyle{plain}
\newtheorem{theorem}[setting/]{Theorem}
\newtheorem{lemma}[setting/]{Lemma}

\providecommand{\smallsum}{\textstyle\sum}
\providecommand{\dd}{\mathop{}\!d}
\providecommand{\di}{\mathop{}\!}
\providecommand{\defeq}{\mathrel{\vcentcolon=}}
\providecommand{\N}{\ensuremath{\mathds{N}}}
\providecommand{\R}{\ensuremath{\mathds{R}}}
\providecommand{\1}{\mathds{1}}
\providecommand{\E}{\mathbb{E}}
\providecommand{\Prob}{\mathbb{P}}
\DeclareMathOperator{\Lip}{Lip}
\DeclareMathOperator{\sgn}{sgn}
\providecommand{\dl}{\delta}
\providecommand{\eps}{\varepsilon}
\providecommand{\cadlag}{c\`adl\`ag}

\providecommand{\MCB}{{\mathcal{B}}}
\providecommand{\MCE}{{\mathcal{E}}}
\providecommand{\MCL}{{\mathcal{L}}}
\providecommand{\MCM}{{\mathcal{M}}}
\providecommand{\MCS}{{\mathcal{S}}}
\providecommand{\MCT}{{\mathcal{T}}}

\providecommand{\deltab}{{\bar\delta}}
\providecommand{\etab}{{\bar{\eta}}}
\providecommand{\Yb}{{\bar{Y}}}

\providecommand{\at}{{\tilde{a}}}
\providecommand{\gt}{{\tilde{g}}}
\providecommand{\hht}{{\tilde{h}}}
\providecommand{\mt}{{\tilde{m}}}
\providecommand{\Xt}{{\tilde{X}}}
\providecommand{\Yt}{{\tilde{Y}}}

\providecommand\given{\nonscript\:\vert\nonscript\:\mathopen{}}
\DeclarePairedDelimiterXPP\Prop[1]{\Prob}(){}{
\renewcommand\given{\nonscript\:\delimsize\vert\nonscript\:\mathopen{}} #1}
\DeclarePairedDelimiterXPP\Exp[1]{\E}[]{}{
\renewcommand\given{\nonscript\:\delimsize\vert\nonscript\:\mathopen{}} #1}
\DeclarePairedDelimiter\abs{\lvert}{\rvert}
\DeclarePairedDelimiter\norm{\lVert}{\rVert}

\title{Propagation of chaos
and the many-demes limit for \\
weakly interacting diffusions
in the sparse regime}
\author{Martin Hutzenthaler$^{1}$ \& Daniel Pieper$^{2}$
\bigskip
\\
\small{Faculty of Mathematics, University of Duisburg-Essen, 45117 Essen,
Germany}\\
\small{$^1$e-mail: martin.hutzenthaler@uni-due.de, $^2$e-mail: daniel.pieper@uni-due.de}
}
\date{September 5, 2019}

\begin{document}

\maketitle

\begin{abstract}
	Propagation of chaos is a well-studied phenomenon and shows that weakly
	interacting diffusions may become independent as the system size converges
	to infinity.
	Most of the literature focuses on the case of exchangeable systems where all
	involved diffusions have the same distribution and are ``of the same size''.
	In this paper, we analyze the case where only a few diffusions start outside
	of an accessible trap.
	Our main result shows that in this ``sparse regime'' the system of weakly
	interacting diffusions converges in distribution to a forest of
	excursions from the trap.
	In particular, initial independence  propagates in the limit and results
	in a forest of independent trees.
\end{abstract}

\let\thefootnote\relax

\tableofcontents

\section{Introduction}\label{sec:introduction}
The notion ``propagation of chaos'' was originally termed by Mark
Kac~\cite{Kac1956} and refers to a relation between microscopic and
macroscopic models.  Microscopic descriptions, on the one hand, are based on
molecules (or particles, individuals, subpopulations, etc.)\ and model their
interactions and driving forces. Macroscopic descriptions, on the other hand,
are based on macroscopic observables such as the density and model the
dynamics of these quantities.  To connect microscopic and macroscopic
descriptions, the limit of the density in the $D$-molecule microscopic model
should converge as $D\to\infty$ to the density in the macroscopic model.
Now Kac's idea behind the terminology ``propagation of chaos'' is that if the
initial distribution is ``chaotic'' (e.g.\ positions and velocities of
molecules are purely random and independent), then the dynamics of the
microscopic model destroys this independence, but finitely many fixed molecules
should in the limit as $D\to\infty$ evolve independently (depending on all
other molecules only through deterministic macroscopic observables such as the
density).
In this sense, independence of finitely many fixed molecules ``propagates''.

Next we give a formal statement of ``propagation of chaos'' for weakly
interacting diffusions.
Let $I\subseteq\R$ be a closed interval (we focus on one-dimensional cases),
let the set $\mathcal{M}_1(I)$ of probability measures on $I$ be equipped with the
$1$-Wasserstein metric,
let $b,\tilde{\sigma}\colon I\times\mathcal{M}_1(I)\to\R$ be measurable
functions,
let $W(i)$, $i\in\N$, be independent standard Brownian motions,
for every $D\in\N$ let $X^D=\{(X_t^D(i))_{t\in[0,\infty)}\colon
i\in\{1,\dotsc,D\}\}$ have state space $I^D$ and be a solution of the
stochastic differential equation (SDE)
\begin{equation}  \begin{split}\label{eq:XDgeneral}
  dX_t^{D}(i)={}
	b\biggl(X_t^{D}(i),\tfrac{1}{D}\smallsum\limits_{j=1}^D
	\delta_{X_t^{D}(j)}\biggr)\dd t
	+\tilde{\sigma}\biggl(X_t^{D}(i),\tfrac{1}{D}\smallsum\limits_{j=1}^D
	\delta_{X_t^{D}(j)}\biggr)&\dd W_t(i),\\
	&t\in (0,\infty),\, i\in\{1,\dotsc,D\},
\end{split}     \end{equation}
and let $M(i)$, $i\in\N$, be independent and identically distributed (i.i.d.)\
and be a solution of the SDE
\begin{equation}  \begin{split}\label{eq:MFSDE}
	d M_t(i)=b\bigl(M_t(i),\Prop{M_t(i)\in\:\cdot\:}\bigr)\dd t +
	\tilde{\sigma}\bigl(M_t(i),\Prop{M_t(i)\in\:\cdot\:}\bigr)\dd W_t(i), \quad
	t\in (0,\infty),\, i\in\N.
\end{split}     \end{equation}
Then under the additional assumptions that $I=\R$, that $b,\tilde{\sigma}$ are
globally Lipschitz continuous, that $\tilde{\sigma}$ is bounded, that $b$
satisfies a linear growth condition, and that for all $D\in\N$ it holds that
$X_0^D$ and $(M_0(i))_{i\in\{1,\dotsc,D\}}$ have the same distribution (in
particular, the components of $X_0^D$ are i.i.d.), Theorem 1
in~\textcite{Oelschlaeger1984} implies for all $k\in\N$
that
\begin{equation}  \begin{split}\label{eq:components}
	\bigl(X_t^D(1),\dotsc,X_t^D(k)\bigr)_{t\in[0,\infty)}\xRightarrow{D\to\infty}
	\bigl(M_t(1),\dotsc,M_t(k)\bigr)_{t\in[0,\infty)}
\end{split}     \end{equation}
in the sense of convergence in distribution on $C([0,\infty),I^k)$ and that
\begin{equation}  \begin{split}\label{eq:empirical.distribution}
	\Biggl(\frac{1}{D}\sum_{j=1}^D\delta_{X_t^D(j)}\Biggr)_{t\in[0,\infty)}
	\xRightarrow{D\to\infty} \bigl(\Exp{\delta_{M_t(1)}}\bigr)_{t\in[0,\infty)}
\end{split}     \end{equation}
in the sense of convergence in distribution on
$C([0,\infty),\mathcal{M}_1(I))$. So although the
components of $X^D$ depend on each other through the empirical distribution
process $(\frac{1}{D}\sum_{j=1}^D\delta_{X^D_t(j)})_{t\in[0,\infty)}$ for
every finite $D\in\N$, in the limit as $D\to\infty$ a finite number of fixed
components become independent since they only ``depend'' on each other through
the deterministic process
$(\Exp{\delta_{M_t(1)}})_{t\in[0,\infty)}$.
Theorem 4.1 in~\textcite{Gaertner1988} implies~\eqref{eq:components}
and~\eqref{eq:empirical.distribution} under more general assumptions including
strict positivity of $\tilde{\sigma}$ and Proposition 4.29
in~\textcite{Hutzenthaler2012} implies~\eqref{eq:components}
and~\eqref{eq:empirical.distribution}
for certain cases where $\tilde{\sigma}$ is locally
H\"older-$\frac{1}{2}$-continuous in the first argument and does not depend on
the second argument.
For further results on propagation of chaos see, e.g.,
\textcite{McKean1967,Sznitman1989,Oelschlaeger1985,MeleardRoelly1987,
LasryLions2007, BuckdahnDjehicheLiPeng2009}.
The limit~\eqref{eq:empirical.distribution} is also referred to as mean-field
approximation.  The SDE~\eqref{eq:MFSDE} is
referred to as mean-field SDE or SDE of McKean-Vlasov type.
An essential observation for all of these results is that $X^D(i)$,
$i\in\{1,\dotsc,D\}$, are exchangeable for every $D\in\N$ so that all components have the
same distribution and are -- informally speaking -- of the ``same size''.

In this paper, we focus on the case $I=[0,1]$ and interpret elements of $[0,1]$
as frequencies (e.g.\ of a certain property within a subpopulation) and think
of a population which is spatially separated into finitely many subpopulations
(also denoted as ``demes'') which are labeled by the elements of
$\{1,\dotsc,D\}$, where $D\in\N$.  We assume that a subpopulation stays in
frequency $0$ as long as there is no immigration into this subpopulation.  Our
question is: What is the limit of $X^D$ as $D\to\infty$ if only one entry in
the vector $(X_0^D(i))_{i\in \N}$ is non-zero?  We will assume that $X_0$ is a
$[0,1]^{\N}$-valued random variable which is almost surely summable and that
for all $D\in\N$ and all $i\in\{1,\dotsc,D\}$  it holds almost surely that
$X_0^D(i)=X_0(i)$.  We will refer to this case as \emph{sparse regime}. In
particular, in the sparse regime $X_0^D$ cannot be exchangeable (and
nontrivial) for every $D\in\N$.  The puzzling question is now how does independence of the
initial frequencies propagate in the \emph{many-demes limit} (cf.,
e.g.,~\textcite{WakeleyTakahashi2004}) as $D\to\infty$?

We will study this non-trivial question under the simplifying assumption that
$b$ is affine-linear in the second argument and that $\tilde{\sigma}$ is
constant in the second argument.  More precisely, let $f\colon [0,1]^2\to \R$,
$h\colon [0,1]\to \R$, $\sigma\colon [0,1]\to[0,\infty)$, and
$h_D\colon[0,1]\to\R$, $D\in\N$, be functions which satisfy
Setting~\ref{set:coefficients} below.  In the special case  where $I=[0,1]$
and where for all $(x,\nu)\in I\times\mathcal{M}_1(I)$ it holds that
$b(x,\nu)=\int_I yf(y,x)\di \nu(dy)+h_D(x)$ and
$\tilde{\sigma}(x,\nu)=\sigma(x)$, for every $D\in\N$ the solution $X^D$
of~\eqref{eq:XDgeneral} solves the SDE
\begin{equation}
	\begin{split}
		dX_t^{D}(i) =
		\frac 1D \sum_{j=1}^D
		X_t^{D}(j)f\bigl( X_t^{D}(j), X_t^{D}(i)\bigr) \dd t + h_D\bigl(
		X_t^{D}(i) \bigr) \dd t
		&+\sqrt{\sigma^2\bigl(X_t^{D}(i)\bigr)}\dd
		W_t(i),\\
		&\quad t\in (0,\infty),\, i\in\{1,\dotsc,D\}.
	\end{split}
	\label{eq:XD}
\end{equation}
We allow the function $h_D$ to depend on $D\in\N$ in order to include weak
immigration (one could think of $h_D(x)=h(x)+\frac{\mu}{D}$ where
$\mu\in[0,\infty)$ is the immigration rate into the total population and where
$h(0)=0$).

Now we describe heuristically the propagation of initial independence in the
many-demes limit.  For this, we assume for simplicity for all $D\in\N$ that
$h_D=h$ (no immigration) and that $X_0(i)=0$ for all $i\in\N\cap[3,\infty)$.
The total mass is bounded in $D$ for every time point.  As a consequence, the
first summand on the right-hand side of~\eqref{eq:XD} converges to zero and
the first deme $X^D(1)$ converges to the solution of the SDE
\begin{equation}
	dY_t = h(Y_t)\dd t + \sqrt{\sigma^2(Y_t)}\dd W_t(1),\quad t\in (0,\infty),
	\label{eq:SDE_Y}
\end{equation}
as $D\to\infty$.
Mass emigrates from this first deme.  This mass will not migrate to deme $2$ (or
deme $1$) since the immigration rate
$\frac{1}{D}X_t^D(1)f(X_t^D(1),X_t^D(2))$ at time $t\in[0,\infty)$ from deme
$1$ to deme $2$ vanishes as $D\to\infty$.  Thus, this mass migrates to a deme
with index in $\{3,4,\dotsc,D\}$ and there will be a finite number of demes
where this mass immigrates and founds a non-vanishing subpopulation.  From
these subpopulations again mass emigrates.
This mass again will not migrate to deme $1$, $2$, $3$, or any other deme with
fixed index $i\in\N$ since the total migration rate into a deme with fixed
index vanishes in the many-demes limit. Instead, this mass migrates again to
randomly chosen demes (which are ``empty'' with asymptotic probability one) and
founds non-vanishing subpopulations. Consequently, since every migrating mass
populates ``empty'' demes (with asymptotic probability one), the
subpopulations which originate from descendants of migrants from deme $1$
constitute a tree of independent subpopulations. Analogously, the
subpopulations which originate from descendants of migrants from deme $2$
constitute a tree of independent subpopulations. In addition, these two trees
are disjoint (and thus driven by independent families of Brownian motions) and
therefore independent if $X_0(1)$ and $X_0(2)$ are independent random
variables.
In other words, independence of the family $\{X_0(i)\colon
i\in\N\}$ propagates in the many-demes limit and results in a forest of
independent trees of independent subpopulations.  A formal statement of this
``propagation of chaos'' result in the sparse regime will be proved in
Theorem~\ref{thm:convergence} below.
We note that, as opposed to the exchangeable regime,
``propagation of independence'' does not mean that fixed demes (e.g., deme $3$ and $4$)
become independent in the limit as $D\to\infty$ (which is rather trivial)
but that the full progenies of individuals starting on deme $3$ and on deme
$4$ do not interfere in the limit as $D\to\infty$.

In the literature, this type of ``propagation of chaos'' has already been
established in two special cases.  Theorem 3.3
in~\textcite{Hutzenthaler2012} proves the analog of
Theorem~\ref{thm:convergence} below in the special case where the
infinitesimal variance $\sigma^2$ is additive (and where $I=[0,\infty)$ and
for all $x,y\in[0,\infty)$ it holds that $f(y,x)=1$) and this additivity
of infinitesimal variances is a strong tool for decomposing the total
population into ``loop-free'' processes.  Moreover, Proposition~2.9
in~\textcite{DawsonGreven2014} proves an analog of
Theorem~\ref{thm:convergence} below in the special case where for all
$x,y\in[0,1]$ and all $D\in\N$ it holds that $\sigma^2(x)= d x(1-x)$,
$f(y,x)=c$, $h_D(x)=-cx+sx(1-x)+\frac{m}{D}(1-x)$ where
$c,d,m,s\in(0,\infty)$ are positive constants and where the forest
of excursions is replaced by a dynamic description hereof which is a
continuous atomic-valued Markov process and where independence of disjoint
trees is not obvious.  In this special case of Wright-Fisher diffusions with
selection and rare mutation, there exists a duality with a particle jump
process and this duality is a very strong tool.
Our more general setup allows for new applications, one of which is carried
out in Subsection~\ref{ss:altruism} below.
Moreover, there are many related results for interacting particle systems
or systems of interacting diffusions where branching processes
appear in ``sparse'' regimes. For example, it is a classical result that
the number of alleles
of one type in a Wright-Fisher model (Moran model) converges to a
branching process in discrete (continuous) time as the population size converges
to infinity if the initial numbers of alleles of this type are bounded.
For results with SuperBrownian motion appearing in suitable rescalings see, e.g.,
\textcite{DurrettPerkins1999, CoxPerkins2005,ChetwyndDiggleEtheridge2018,
ChetwyndDiggleKlimek2019}.

The structure of this paper is as follows.  In
Subsection~\ref{ss:tree_excursions} we introduce the forest of
excursions, in Subsection~\ref{ss:main_result} we state our main result
Theorem~\ref{thm:convergence}, and in Subsection~\ref{ss:altruism} we specify
an application to altruistic defense traits.  The proof of
Theorem~\ref{thm:convergence} consists essentially of two major steps.  In
Section~\ref{sec:convergence_loop_free} we prove that if ancestral lineages of
individuals never come back to a deme, then the resulting ``loop-free''
processes (see the SDE~\eqref{eq:ZD_k} below) converge in the many-demes limit
(see Lemma~\ref{l:convergence_of_the_loop_free_process} below).  Moreover, in
Section~\ref{sec:convergence_excursions} we show that the distance between the
$D$-demes process~\eqref{eq:XD} and the corresponding ``loop-free'' process
converges suitably to zero as $D\to\infty$ (see Lemma~\ref{l:replace_fidi}
below).
The principal idea of reducing the problem to loop-free processes stems
from~\textcite{Hutzenthaler2012}.
Throughout this paper, we use the notation from
Subsection~\ref{ss:notation} below without further mentioning.

\subsection{Notation}\label{ss:notation}
%We use the conventions that $0^0 = 1$, $0\cdot \infty = 0$, and that for all
%$x\in(0,\infty)$ we have that $\frac{x}{\infty} = 0$ and $\frac{x}{0} =
%\infty$.
%In this subsection, we fix some notation which we use in this paper.
%Throughout this subsection let $(E,d_E)$ and $(F,d_F)$ be metric spaces.
For all $x,y\in\R$ we define $x\wedge y \defeq \min\{x,y\}$, $x \vee y
\defeq \max\{x,y\}$, $x^+ \defeq x \vee 0$, $x^- \defeq -(x \wedge 0)$,
and $\sgn(x) \defeq \1_{(0,\infty)}(x) - \1_{(-\infty,0)}(x)$.
We define $\sup \emptyset
\defeq -\infty$ and $\inf \emptyset \defeq \infty$.
We write $\N \defeq \{1,2,3,\ldots\}$ and $\N_0 \defeq \N \cup \{0\}$.
For all $N\in\N$ we write $[N] \defeq \{1,\dotsc,N\}$ and $[N]_0 \defeq
[N] \cup \{0\}$.

For the remainder of this subsection, let $(E,d_E)$ and $(F,d_F)$ be metric
spaces and let $d\in\N$ and $m\in\N_0$.
We denote by $\MCB(E)$ the Borel
$\sigma$-algebra on $(E,d_E)$ and by $\MCM_\textup{f}(E)$ the set
of finite measures on $(E,\MCB(E))$ endowed with the weak
topology.
For every $s\in[0,\infty)$ we denote by
$D([s,\infty),E)$ the set of
all \cadlag\ functions $f\colon [s,\infty) \to E$ endowed with the
Skorokhod topology.
We denote by $C(E,F)$ the set of
all continuous functions $f\colon
E\to F$ and by $\Lip(E,F)$ the set of all Lipschitz continuous functions
$f\colon E\to F$.
We denote by $C^2_b(\R,\R)$ the set of twice continuously
differentiable bounded functions $\psi\colon \R \to \R$ with bounded first
and second derivative. For every $\psi \colon \R \to \R$ we write
$\norm{\psi}_{\infty} \defeq \sup_{x\in\R} \abs{\psi(x)} \in [0,\infty]$.
We denote by $C^m([0,1]^d,\R)$ the set
of functions $\psi\colon [0,1]^d \to \R$ whose partial derivatives of
order $0$ through $m$ exist and are continuous on $[0,1]^d$. For every
$\psi\colon [0,1]^d \to \R$ we define $\norm{\psi}_\infty \defeq
\sup_{x\in [0,1]^d}\abs{\psi(x)} \in [0,\infty]$. For every multiindex
$\alpha = (\alpha_1,\dotsc,\alpha_d) \in \N_0^d$ of length $\abs{\alpha}
\defeq \sum_{k=1}^d \alpha_k$ we write $\partial^\alpha \defeq
\frac{\partial^{\abs \alpha}}{\partial x_1^{\alpha_1}\dotsm\partial
x_d^{\alpha_d}}$.
For every $\psi \in C^m([0,1]^d,\R)$ we set
$\norm{\psi}_{C^m} \defeq \max_{\alpha\in\N_0^d, \abs{\alpha} \leq
m}\norm{\partial^\alpha\psi}_\infty$.

By a solution of an SDE driven by Brownian motions we mean a stochastic
process with continuous sample paths which is adapted to the filtration
generated by the Brownian motions and the initial value and which satisfies
the integrated SDE for every time point almost surely.

\subsection{Setting and forest of excursions}\label{ss:tree_excursions}
In this subsection, we gather the assumptions that we impose in our main
result, Theorem~\ref{thm:convergence} below, and we introduce the forest of
excursions which plays the role of a limiting object in our main result.

In the following Setting~\ref{set:coefficients}, we collect our assumptions on
the coefficients of the SDE~\eqref{eq:XD}.
Under these assumptions, for every $D\in\N$ the SDE~\eqref{eq:XD} has a unique
strong solution with continuous sample paths in $[0,1]^D$; see Theorem~3.2
in~\textcite{ShigaShimizu1980}.
Moreover, under these assumptions, the SDE~\eqref{eq:SDE_Y} has a unique
strong solution $(Y_t)_{t\in[0,\infty)}$ with continuous sample paths in
$[0,1]$ for which $0$ is a trap, that is, for all $t,s\in[0,\infty)$ it holds
that ($Y_t=0$ implies $Y_{t+s}=0$).
\begin{setting}[Coefficient functions] \label{set:coefficients}
	Let $\mu \in [0,\infty)$, $f \in C^3([0,1]^2,\R)$, $h \in
	C^3([0,1],\R)$, and $h_D \in C^3([0,1],\R)$, $D\in\N$, satisfy that
	$\sup_{D\in\N} \norm{h_D}_{C^2} < \infty$, that
	$\lim_{D\to\infty}Dh_D(0) = \mu$, for all $x \in [0,1]$ that
	$\lim_{D\to\infty} h_D(x) = h(x)$, and for all $D\in\N$ and all $y\in
	(0,1]$ that $yf(y,1) + h(1) \leq 0$, that $yf(y,1) + h_D(1) \leq 0$,
	that $f(y,0) > 0$, that $2\mu \geq Dh_D(0) \geq 0 = h(0)$, and that
	$h(1) < 0$. 

	Let $L_f, L_h \in [0,\infty)$ be such that for all $D\in\N$ and all
	$x,y,u,v\in[0,1]$ it holds that $\abs{f(y,x) - f(v,u)} \leq L_f
	\abs{y-v} + L_f \abs{x-u}$, that $\abs{f(y,x)} \leq L_f$, that
	$\abs{h(x) - h(y)} \leq L_h \abs{x-y}$, and that $\abs{h_D(x) - h_D(y)}
	\leq L_h \abs{x-y}$.
		
	Let $\sigma^2 \in C^3([0,1],\R)$ satisfy that $\sigma^2(0) = 0 =
	\sigma^2(1)$ and for all $x\in(0,1)$ that $\sigma^2(x) > 0$. Let
	$L_\sigma \in [0,\infty)$ be such that for all $x,y\in[0,1]$ it holds
	that $\abs{\sigma^2(x) - \sigma^2(y)} \leq L_\sigma \abs{x-y}$.

	For every $D\in\N$ we define the function $\hht_D\colon [0,1]\to\R$
	by $[0,1]\ni x \mapsto \hht_D(x) \defeq h_D(x) - h_D(0)$.
\end{setting}

In addition to Setting~\ref{set:coefficients}, we impose the following
assumptions involving the scale function $S$ of $Y$, which imply that $Y$ hits
zero in finite time almost surely (a straightforward adaptation of Lemma~9.5
and Lemma~9.6 in~\textcite{Hutzenthaler2009} to the state space $[0,1]$
shows that this is ensured by~\eqref{eq:Y_hits_zero_wpp} below) and that there
exists an excursion measure for $Y$.
\begin{setting}[Scale function] \label{set:excursion_measure}
	Assume that Setting~\ref{set:coefficients} holds and that there exists a $y\in
	(0,1)$ such that
	\begin{equation}
		\lim_{(0, y)\ni \eps\to 0}\int_\eps^y
		\frac{h(x)}{\sigma^2(x)}\dd x \in \R.
	\end{equation}
	We define the functions $s,S\colon
	[0,1)\to[0,\infty)$ and $\at\colon[0,1]\to[0,\infty)$ by
  \begin{align}
		[0,1) \ni z \mapsto s(z) 
    &\defeq \exp\mleft( -\int_0^z
		\frac{h(x)}{\frac{1}{2}\sigma^2(x)} \dd x \mright),
		\label{eq:s}
    \\
		[0,1) \ni y \mapsto S(y)
		&\defeq\int_0^y s(z)\dd z,
		\label{eq:S}
    \\
		[0,1] \ni y \mapsto \at(y) 
    &\defeq yf(y,0).
		\label{eq:atilde}
  \end{align}
	We further assume that there exists an $x\in (0,1)$ with the property that
	\begin{equation}
		\int_0^x \frac{S(y)}{\sigma^2(y)s(y)}\dd y
    +
		\int_x^{1}\frac{\at(y)}{\sigma^2(y)s(y)}\dd y < \infty.
    \qedhere
		\label{eq:Y_hits_zero_wpp}
	\end{equation}
\end{setting}
For all $\eta \in C(\R,[0,1])$ we denote by $T_0(\eta) \defeq \inf\{t \in
(0,\infty) \colon \eta_t = 0\}\in[0,\infty]$ the first (nonnegative) time of
hitting zero and we define the \emph{set of excursions from zero} by
\begin{equation}
	U\defeq\left\{\eta\in C(\R,[0,1])\colon
		T_0(\eta) \in (0,\infty] \text{ and } \eta_t=0\text{ for all } t\in(-\infty,0] \cup
		[T_0,\infty)\right\}.
	\label{eq:excursion_space}
\end{equation}
Moreover, we denote by $D(\R,[0,1])$ the set of all \cadlag\ functions $f\colon \R
\to [0,1]$ and we define
\begin{equation}
	V \defeq \left\{ \eta \in D(\R,[0,1])\colon \eta_t = 0 \text{ for all } t
	\in (-\infty,0)\right\}.
	\label{eq:path_space}
\end{equation}
In the situation of Setting~\ref{set:excursion_measure}, Theorem~1 
in~\textcite{Hutzenthaler2009} adapted to the state space $[0,1]$ shows that
there exists a unique $\sigma$-finite measure $Q$ on $U$ satisfying the
following property: For every bounded and continuous function $F\colon
C([0,\infty),[0,1])\to\R$ with the property that there exists a $\dl>0$ such
that for all $\chi\in C([0,\infty),[0,1])$ with
$\sup_{t\in[0,\infty)}\chi_t<\dl$ it holds that $F(\chi) = 0$, it holds that
\begin{equation} \label{eq:Q}
	\lim_{(0,1) \ni \eps\to0}\frac{1}{S(\eps)}\Exp{F(Y) \given Y_0 = \eps}=\int F(\eta)
	\di Q(d\eta).
\end{equation}
The measure $Q$ is called the \emph{excursion measure} associated with $Y$;
see also~\textcite{PitmanYor1982}.
A straightforward adaptation of Lemma~9.8 in~\textcite{Hutzenthaler2009} to
the state space $[0,1]$ and assumption~\eqref{eq:Y_hits_zero_wpp} imply that
\begin{equation}
	\int \int_0^\infty \at(\chi_t) \dd t \di Q(d\chi) = \int_0^{1}
	\frac{\at(y)}{\frac{1}{2}\sigma^2(y)s(y)}\dd y < \infty.
	\label{eq:excursion_area}
\end{equation}

For the convergence result, we further assume the following
Setting~\ref{set:initial_moment} for the initial distributions which
establishes the sparse regime.
\begin{setting}[Sparse initial condition] \label{set:initial_moment}
	Assume that Setting~\ref{set:excursion_measure} holds. For every $i\in\N$ let $X_0(i)$
	be a $[0,1]$-valued random variable and for every $D\in\N$ let
	$\{(X^D_t(i))_{t\in[0,\infty)}
	\colon i\in[D]\}$ be a solution of~\eqref{eq:XD}
	such that a.s.\ $\sum_{i=1}^\infty
	X_0(i) < \infty$
	and such that for all $D\in\N$ and all $i\in[D]$ it holds
	a.s.~that $X_0^D(i) = X_0(i)$.
\end{setting}
Under the assumption of Setting~\ref{set:initial_moment}, we now construct the
associated forest of excursions. For that, let
$Y(i) = (Y_t(i))_{t\in[0,\infty)}$, $i\in\N$, be independent solutions
of~\eqref{eq:SDE_Y} coupled only through their initial states such that for
all $i\in\N$ it holds a.s.~that $Y_0(i) = X_0(i)$. These trajectories describe
the demes of the initial population.
Let $\{\Pi^\emptyset\} \cup \{ \Pi^{(n,s,\chi)}\colon
(n,s,\chi)\in\N_0\times[0,\infty)\times V\}$ be an
independent family of Poisson point processes on $[0,\infty) \times U$ with
intensity measures
\begin{equation}
	\Exp*{\Pi^\emptyset(dt\otimes d\eta)} = \mu\dd t\otimes Q(d\eta)
	\label{eq:intPi0}
\end{equation}
and
\begin{equation}
	\Exp*{\Pi^{(n,s,\chi)}(dt\otimes d\eta)} =
	\at(\chi_{t-s})\dd t\otimes Q(d\eta),\quad
	(n,s,\chi)\in\N_0\times[0,\infty)\times V.
	\label{eq:intPin}
\end{equation}
The points of $\Pi^\emptyset$ and $\Pi^{(n,s,\chi)}$ are interpreted as
tuples of times and paths providing the population times of new demes and
the evolution of the population inside these demes. Here, $\Pi^\emptyset$
describes the demes whose founders immigrated into the system, while
$\Pi^{(n,s,\chi)}$ describes the demes which descend from a deme with
population size trajectory $(\chi_{t-s})_{t\in[s,\infty)}$ and where the
ancestral lineages of individuals living on these demes have exactly $n\in\N$
migration events (only counting migration events within the system).
The $0$-th generation is the random $\sigma$-finite measure on
$[0,\infty)\times V$ defined through $\MCT^{(0)}\defeq
\sum_{i=1}^\infty\delta_{(0,Y(i))} + \Pi^\emptyset$.  For every $n\in\N_0$ the
$(n+1)$-th generation is the random $\sigma$-finite measure on $[0,\infty)\times U$
representing all
the demes which have been colonized from demes of the $n$-th generation, that
is, $\MCT^{(n+1)}\defeq \int\Pi^{(n,s,\chi)}\di \MCT^{(n)}(ds \otimes d\chi)$.
The \emph{forest of excursions} $\MCT$ is then the sum of all of these
measures $\MCT\defeq\sum_{n\in\N_0}\MCT^{(n)}$.

A straightforward adaptation of Lemma~5.2, Lemma~9.9, and Lemma~9.10
in~\textcite{Hutzenthaler2009} to the state space $[0,1]$ shows for every
$t\in[0,\infty)$ that the total mass $\int \chi_{t-s}\di \MCT(ds \otimes
d\chi)$ has finite expectation and is thus finite almost surely.
Moreover, in the case where $\mu=0$ (no immigration) and where there exists an
$x\in(0,1]$ such that for all $i\in\N$ it holds that $X_0(i) = x\1_{i=1}$,
Theorem~5 in \textcite{Hutzenthaler2009} yields that the total mass process
dies out (that is, $\int \chi_{t-s} \di \MCT(ds \otimes d\chi)$ converges to
zero in probability as $t \to\infty$) if and only if
\begin{equation}
	\int \int_0^\infty \at(\chi_t)\dd t \di Q(d\chi) \leq 1.
	\label{eq:tree_dies_out}
\end{equation}

\subsection{Main result: Propagation of chaos in the sparse regime}\label{ss:main_result}
In this subsection, we state our main theorem.
\begin{theorem}[Convergence to a forest of excursions]\label{thm:convergence}
	Assume that \textup{Setting~\ref{set:initial_moment}} holds and let $\MCT$
	be the forest of excursions constructed in
	\textup{Subsection~\ref{ss:tree_excursions}}.
	Then it holds that
	\begin{equation}
		\left(\sum_{i=1}^D X_t^{D}(i) \delta_{X_t^{D}(i)} \right)_{t\in[0,\infty)}
		\xRightarrow{D\to\infty} \left( \int \eta_{t-s} \delta_{\eta_{t-s}} \di\MCT(ds \otimes
		d\eta)\right)_{t\in[0,\infty)}
		\label{eq:convergence_thm}
	\end{equation}
	in the sense of convergence in distribution on
	$D([0,\infty),\MCM_\textup{f}([0,1]))$.
\end{theorem}
The form of the left-hand side in~\eqref{eq:convergence_thm} might
look unfamiliar at first glance. We note that the sequence $\{(\sum_{i=1}^D
\delta_{X^D_t(i)})_{t\in[0,\infty)}\colon D\in\N\}$ has no chance of being
relatively compact as stochastic processes with values in
$\MCM_\textup{f}([0,1])$, as an infinite amount of mass piles up near zero
in the limit as $D\to\infty$. By weighting the point masses
as in~\eqref{eq:convergence_thm}, we avoid this problem and retain
a well-understood state space. Alternatively, one
can change the (topology of the) state space. This is done
in~\textcite{Hutzenthaler2012}, where $\sigma$-finite measures on $(0,1]$
with the vague topology are used to prove convergence of $\{(\sum_{i=1}^D
\delta_{X^D_t(i)})_{t\in[0,\infty)}\colon D\in\N\}$ instead.

We emphasize that the limiting object $\MCT$ is easier to analyze than the
solution of the SDE~\eqref{eq:XD} due to its tree structure and since general
branching processes are very well understood, resulting, for example, in the
criterion~\eqref{eq:tree_dies_out}. Further properties of the limiting process
in the case $\mu = 0$ without immigration are investigated
in~\textcite{Hutzenthaler2009}.

\subsection{Main ideas and structure of the proof of Theorem~\ref{thm:convergence}}
In Subsection~\ref{ss:migration_levels} we decompose~\eqref{eq:XD} into
processes with migration levels in the sense that the next higher migration
level is driven by successful migrations essentially only from the migration
level directly below, see~\eqref{eq:XD_k} below.
We then couple these migration level processes to ``loop-free''
processes which we obtain by pretending that on a fixed deme all individuals have the same
migration level, which turns~\eqref{eq:XD_k} into~\eqref{eq:ZD_k} below. We
show in Subsection~\ref{ss:convergence_loop_free} that these loop-free
processes converge in the limit as $D\to\infty$ to the forest of excursions.
To prove this result, we apply induction on the number of migration steps, which
is useful since, conditionally on the lower migration levels, the processes in the next higher
migration level in the loop-free processes evolve as independent diffusions.
The convergence of independent diffusions of this kind is obtained in
Subsection~\ref{ss:poisson_limit} in
Lemma~\ref{l:vanishing_immigration_weak_process} below.
This result can be seen as a functional Poisson limit theorem and follows
essentially from Lemma~\ref{l:vanishing_immigration} below which is proved
in~\textcite{Hutzenthaler2012} by reversing time. 

The principal idea of decomposing into loop-free processes is taken
from~\textcite{Hutzenthaler2012}. There, however, the considerations were
restricted to the ``classical'' migration term of the form
$\frac{1}{D}\sum_{j=1}^D(X^D_t(j) - X^D_t(i))$, while here we allow for
``nonlinear'' migrations of the form $\frac{1}{D}\sum_{j=1}^D
X^D_t(j)f(X^D_t(j),X^D_t(i))$. By considering the case where for all $x,y \in[0,1]$ it
holds that $f(y,x) = 1$ and where for all $x \in [0,1]$ and all $D\in\N$ we
replace $h_D(x)$ by $h_D(x) - x$ in~\eqref{eq:XD}, we recover the classical
migration term in our framework.

To complete the proof of Theorem~\ref{thm:convergence}, it remains to show
that the migration level processes and the loop-free processes have the same
limit as $D\to\infty$. This is carried out in
Subsection~\ref{ss:reduction_loop_free}. For this,~\textcite{Hutzenthaler2012}
is restricted to the case where $\sigma^2$ is of the form $\sigma^2(x) = \beta
x$ for a constant $\beta \in (0,\infty)$.
This assumption crucially simplified the situation, as this completely decouples
the infinitesimal variances of the migration level processes and allows to
prove that the $L^1$-distance between the migration level processes and the
loop-free processes converges to zero as $D\to\infty$ using Gronwall's
inequality.  In the more general situation of this paper, we instead show that
a certain form of a weak distance between the migration level processes and
the loop-free processes converges to zero as $D\to\infty$, see
Lemma~\ref{l:replace_fidi} below. The proof relies on applying It{\^o}'s
formula to suitable evaluations of the semigroup of the loop-free processes,
see~\eqref{eq:weak_error} below.

%The idea of the proof is to decompose~\eqref{eq:XD} into processes with
%migration levels and then derive loop-free processes from these. Conditioned
%on the migration levels before, the convergence of these loop-free processes
%can be proved using a Poisson limit lemma for independent diffusions. It then
%remains to show that in the limit as $D\to\infty$, the loop-free processes are
%``close'' to the migration level processes.
%In case of additive $\sigma^2$, it could be
%shown that this is the case in a strong $L^1$ sense, but we are not able to
%generalize this to nonlinear $\sigma^2$. Instead, we apply here a weak
%distance to show the asymptotic equality of distributions. To do so, we bound
%the $C^2$-norm of the semigroup of the loop-free processes uniformly in
%$D\in\N$ by applying a result obtained in~\textcite{HutzenthalerPieper2018}.

%
%
\subsection{Application: Altruistic defense traits} \label{ss:altruism}
Let $\alpha, \beta, \kappa \in (0,\infty)$, let $\mu_\infty \in [0,\infty)$,
let $a\in (1,\infty)$, let
$(\mu_D)_{D\in\N} \subseteq [0,1]$ be such that $\lim_{D\to\infty}
D\mu_D = \mu_\infty$ and such that for all $D\in\N$ it holds that $D\mu_D \leq
2 \mu_\infty$, and let $b
\in C^3([0,1],\R)$ be such that $b(1) = 0 \leq b(0)$. Let $f\colon [0,1]^2 \to \R$, $h\colon [0,1]\to \R$, $h_D\colon [0,1]\to\R$,
$D\in\N$, and
$\sigma^2\colon [0,1] \to [0,\infty)$ be the functions satisfying for all
$D\in\N$ and all $x,y
\in [0,1]$ that $f(y,x) =
\kappa(a-x)\frac{a-x}{a}\frac{1}{a-y}$, that $h(x) = -\kappa
x(a-x)\frac{1}{a} - \alpha x(1-x)$, that $h_D(x) = h(x) + \mu_D b(x)$, and
that $\sigma^2(x) = \beta (a-x)x(1-x)$.
Then it holds for all $x,y \in [0,1]$ that $yf(y,x) + h(x) =
\kappa\frac{a-x}{a-y}(y-x) - \alpha x(1-x)$.
Thus, for every $D\in\N$,~\eqref{eq:XD} specializes to the SDE
\begin{equation} \label{eq:SDE_HJM}
	\begin{split}
		dX_t^D(i) ={}
		&\frac{\kappa}{D}\sum_{j=1}^{D}\tfrac{a-X^D_t(i)}{a-X^D_t(j)}\bigl(
			X^D_t(j)-X^D_t(i) \bigr)\dd t - \alpha
			X^D_t(i)\bigl(1-X^D_t(i)\bigr)\dd t + \mu_D b\bigl(X^D_t(i)\bigr) \dd t \\
		&+ \sqrt{\beta\bigl(a-X^D_t(i)\bigr)X^D_t(i)\bigl(1-X^D_t(i)\bigr)}\dd W_t(i),\quad t\in
		(0,\infty),\,i\in[D].
	\end{split}
\end{equation}
Theorem 1.3 in~\textcite{HutzenthalerJordanMetzler2015} shows (in the case
where for all $D\in\N$ it holds that
$\mu_D = 0$, that is, the case of no mutation) that
the solution process of the SDE~\eqref{eq:SDE_HJM} arises as a diffusion limit
of the relative frequency of altruistic individuals in the host population in
a Lotka-Volterra type host-parasite model.
The reason for the functional form of the coefficient functions is as follows.
Since individuals migrate at fixed rate, relative frequencies (write it as $\tfrac{A_t^{D,N}(i)}{H_t^{D,N}(i)}$)
of altruistic individuals in the host population (let $H_t^{D,N}(i)$ denote the total number of hosts on deme $i\in[D]$
at time $t\in[0,\infty)$) migrate at rate $\tfrac{\kappa}{N}m(i,j)\tfrac{H_t^{D,N}(j)}{H_t^{D,N}(i)}$ from deme $j$ to deme $i$.
The total host population evolves much faster (separation of time scales) than
relative frequencies and stabilizes locally in time
at $\tfrac{1}{\beta(a-x)}$ if $x$ is the current frequency of altruists and where $a,\beta$ are suitable parameters.
Thus, if time is measured in multiples of $N$, migration rates of frequencies converge to $w\!-\!\lim_{N\to\infty}N\cdot\tfrac{\kappa}{N}m(i,j)
\tfrac{H_{tN}^{D,N}(j)}{H_{tN}^{D,N}(i)}=\kappa m(i,j)\tfrac{\beta(a-X_t^{D}(i))}{\beta(a-X_t^{D}(j))}$.
The resampling rate in the Wright-Fisher diffusion is inverse-proportional to the total mass and this explains why the squared diffusion term is
%\begin{equation}  \begin{split}
  $w\!-\!\lim_{N\to\infty} \tfrac{1}{H_{tN}^{D,N}(i)}\tfrac{A_{tN}^{D,N}(i)}{H_{tN}^{D,N}(i)}(1-\tfrac{A_{tN}^{D,N}(i)}{H_{tN}^{D,N}(i)})
  =\beta(a-X_t^{D}(i))X_t^{D}(i)(1-X_t^{D}(i))$.
%\end{split}     \end{equation}

Define $\mu \defeq \mu_\infty b(0)$.
We check the assumptions of Theorem~\ref{thm:convergence}:
Indeed, $\mu \in [0,\infty)$, $f\in C^3([0,1]^2,\R)$, $h,\sigma^2\in C^3([0,1],\R)$, for all
$D\in\N$ it holds that $h_D \in C^3([0,1],\R)$, it holds that
$\sup_{D\in\N}\norm{h_D}_{C^2} \leq \norm{h}_{C^2} + \norm{b}_{C^2} < \infty$
and
that $\lim_{D\to\infty}
Dh_D(0) = \lim_{D\to\infty} D \mu_D b(0) = \mu_\infty b(0) = \mu$, for all $x\in[0,1]$ it holds that $\lim_{D\to\infty} h_D(x) =
h(x)$,
for all $D\in\N$ and all $y\in(0,1]$ it holds that $yf(y,1)+h_D(1) = yf(y,1)
+ h(1) = \frac{\kappa y(a-1)^2}{a(a-y)}-\frac{\kappa (a-1)}{a} \leq 0$,
that $f(y,0)=\frac{\kappa a}{a-y}>0$, that $Dh_D(0) = D \mu_D b(0) \leq 2
\mu_\infty b(0) = 2\mu$, that $Dh_D(0) = D \mu_D b(0) \geq 0$, that
$h(0) = 0$, and that $h(1)=-\frac{\kappa(a-1)}{a}<0$.
Moreover, it holds that $\sigma^2(0)=0=\sigma^2(1)$ and for all
$x\in(0,1)$ that $\sigma^2(x)>0$.
Thus, Setting~\ref{set:coefficients} is satisfied. Furthermore, it holds that
\begin{equation}
	\begin{split}
		\lim_{(0,\frac{1}{2}) \ni \eps \to 0} \int_\eps^{\frac{1}{2}}
		\tfrac{h(x)}{\sigma^2(x)} \dd x
		&= \lim_{(0,\frac{1}{2}) \ni \eps \to 0} \int_\eps^{\frac{1}{2}} \tfrac{-\kappa
		x(a-x)\frac{1}{a} - \alpha x(1-x)}{\beta (a-x)x(1-x)} \dd x\\
		&= \lim_{(0,\frac{1}{2}) \ni \eps \to 0} \int_\eps^{\frac{1}{2}}
		\tfrac{-\kappa}{a\beta(1-x)} -
		\tfrac{\alpha}{\beta(a-x)} \dd x\\
		&= \lim_{(0,\frac{1}{2}) \ni \eps \to 0} \left( \tfrac{\kappa}{a
		\beta}\bigl(\ln(1-\tfrac{1}{2}) - \ln(1-\eps)\bigr) +
		\tfrac{\alpha}{\beta}\bigl( \ln(a-\tfrac{1}{2}) - \ln(a-\eps) \bigr)
		\right)\\
		&= \tfrac{\kappa}{a\beta} \ln(1-\tfrac{1}{2}) +
		\tfrac{\alpha}{\beta}\bigl(\ln(a-\tfrac{1}{2}) - \ln(a)\bigr) \in \R.
	\end{split}
	\label{eq:altruism_set1}
\end{equation}
Let $s,S\colon [0,1)\to[0,\infty)$ and $\at \colon [0,1]\to[0,\infty)$ be
given by~\eqref{eq:s}, \eqref{eq:S}, and~\eqref{eq:atilde}, respectively. Then
it holds for all
$z \in [0,1)$ that
\begin{equation}
	s(z) = \exp\left( \int_0^z \tfrac{2\kappa}{a\beta(1-x)} +
	\tfrac{2\alpha}{\beta(a-x)} \dd x\right) =
	(1-z)^{-\frac{2\kappa}{a\beta}}\left( \tfrac{a-z}{a}
	\right)^{-\frac{2\alpha}{\beta}}
	\label{eq:s_altruism}
\end{equation}
and
\begin{equation}
	S(z) = \int_0^z s(x) \dd x \leq z s(z).
	\label{eq:S_estimate}
\end{equation}
We obtain from~\eqref{eq:S_estimate} that
\begin{equation}
	\int_0^{\frac{1}{2}} \tfrac{S(y)}{\sigma^2(y) s(y)} \dd y \leq
	\int_0^{\frac{1}{2}} \tfrac{1}{\beta(a-y)(1-y)} \dd y \leq
	\tfrac{1}{2\beta(a-\frac{1}{2})(1-\frac{1}{2})} < \infty
	\label{eq:altruism_set2}
\end{equation}
and it follows from~\eqref{eq:s_altruism} and from the fact that
$\frac{2\kappa}{a\beta} - 1 \in (-1,\infty)$ that
\begin{equation}
	\begin{split}
		\int_{\frac{1}{2}}^1 \tfrac{\at(y)}{\sigma^2(y)s(y)} \dd y
		&= \int_{\frac{1}{2}}^1 \tfrac{\frac{\kappa a y}{a-y}}{\beta (a-y)y(1-y)}
		(1-y)^{\frac{2\kappa}{a\beta}}
		\left(\tfrac{a-y}{a}\right)^{\frac{2\alpha}{\beta}} \dd y \\
		&= \tfrac{\kappa a}{\beta} a^{-\frac{2\alpha}{\beta}} \int_{\frac{1}{2}}^1
		(1-y)^{\frac{2\kappa}{a\beta} - 1}
		(a-y)^{\frac{2\alpha}{\beta} - 2} \dd y \\
		&\leq \tfrac{\kappa a}{\beta} a^{-\frac{2\alpha}{\beta}} \left(
		(a-\tfrac{1}{2})^{\frac{2\alpha}{\beta} - 2} + (a-1)^{\frac{2\alpha}{\beta}
		- 2}\right) \int_{\frac{1}{2}}^1 (1-y)^{\frac{2\kappa}{a\beta} - 1} \dd y
		< \infty.
	\end{split}
	\label{eq:altruism_set3}
\end{equation}
Hence, Setting~\ref{set:excursion_measure} is satisfied. Therefore,
Theorem~\ref{thm:convergence} is applicable to the SDE~\eqref{eq:SDE_HJM} for
any initial configuration satisfying Setting~\ref{set:initial_moment}.

For the remainder of this subsection we consider the case where $\mu_\infty =
0$ and where there exists an $x\in(0,1]$ such that for all
$D\in\N$ and all $i\in[D]$ it holds that $X^D_0(i) = x\1_{i=1}$.
We obtain from \eqref{eq:altruism_set1},
\eqref{eq:altruism_set2}, and \eqref{eq:altruism_set3} together with a
straightforward adaptation of Lemma~9.6, Lemma~9.9, and Lemma~9.10
in~\textcite{Hutzenthaler2009} to the state space $[0,1]$ that the assumptions
of Theorem~5 in~\textcite{Hutzenthaler2009} are satisfied. An application of
the latter theorem shows that the total
mass process $(\int \chi_{t-s} \di \MCT(ds \otimes d\chi))_{t\in[0,\infty)}$
dies out (that is, it converges to zero in probability as $t\to\infty$) if and
only if~\eqref{eq:tree_dies_out} holds.
Equations~\eqref{eq:excursion_area} and~\eqref{eq:s_altruism} yield that
\begin{equation}
	\int \int_0^\infty \at(\chi_t)\dd t \di Q(d\chi) = \int_0^1
	\tfrac{\at(y)}{\frac{1}{2}\sigma^2(y)s(y)} \dd y
	= \int_0^1 \tfrac{\frac{\kappa
	ay}{a-y}}{\frac{1}{2}\beta(a-y)y(1-y)}(1-y)^{\frac{2\kappa}{a\beta}}\left(
	\tfrac{a-y}{a} \right)^{\frac{2\alpha}{\beta}} \dd y.
	\label{eq:altruism_excursion_area}
\end{equation}
This together with~\eqref{eq:tree_dies_out} and the fact that $\frac{2\kappa}{a\beta}\int_0^1
(1-y)^{\frac{2\kappa}{a\beta} - 1} \dd y = 1$ proves that the total mass
process dies out if and only if
\begin{equation}
	\begin{split}
		0 \geq \int_0^1 \tfrac{\frac{\kappa
		ay}{a-y}}{\frac{1}{2}\beta(a-y)y(1-y)}(1-y)^{\frac{2\kappa}{a\beta}}\left(
		\tfrac{a-y}{a} \right)^{\frac{2\alpha}{\beta}} \dd y - 1
		&= \tfrac{2\kappa}{a\beta} \int_0^1 (1-y)^{\frac{2\kappa}{a\beta} - 1}
		\left( \tfrac{a-y}{a} \right)^{\frac{2\alpha}{\beta} - 2} \dd y - 1\\
		&= \tfrac{2\kappa}{a\beta} \int_0^1 (1-y)^{\frac{2\kappa}{a\beta} - 1}
		\left( \left( \tfrac{a-y}{a} \right)^{\frac{2\alpha}{\beta} - 2} - 1
		\right) \dd y.
	\end{split}
\end{equation}
Consequently, the total mass process dies out if and only if $\alpha \geq
\beta$.
%
%Moreover,
%in the mean-field regime (that is, if $(X_0(i))_{i\in\N}$ are exchangeable with positive mean frequency)
%the many-demes
%limit of~\eqref{eq:SDE_HJM} goes to fixation if $\alpha \geq \beta$; see
%Theorem 1.4 in~\textcite{HutzenthalerJordanMetzler2015}.
%Additionally, if
%initially there is a positive frequency of altruists on one deme, then the
%total mass process of the tree of excursions associated
%with~\eqref{eq:SDE_HJM} dies out, that is, it converges to zero in probability, if and
%only if $\alpha \geq \beta$; see~\textcite[Proposition~5.1]{HutzenthalerJordanMetzler2015}. Informally
%speaking, the combination of these results shows that an
%altruistic defense allele has a positive invasion probability and persists in
%an infinite-dimensional space if $\alpha < \beta$.

\section{Convergence of the loop-free
processes}\label{sec:convergence_loop_free}

\subsection{Migration level processes and loop-free processes}
\label{ss:migration_levels}
Throughout this subsection, assume that Setting~\ref{set:coefficients} holds.
To prove Theorem~\ref{thm:convergence}, we use a decomposition into migration
levels.
We say that an individual has \emph{migration level} $k\in\N_0$ at time
$t\in[0,\infty)$ if its ancestral lineage up to time $t$ contains exactly $k$ migration
steps (within the system).
To formalize this, we define for all $D\in\N$ that $X^{D,-1} \defeq 0$ and
consider for every $D\in\N$ the SDE
\begin{equation}
	\begin{split} \label{eq:XD_k}
		dX_t^{D,k}(i) ={}
		&\frac{1}{D} \sum_{j=1}^D X_t^{D,k-1}(j)
			f\left( \sum_{m\in\N_0}X_t^{D,m}(j),
			\sum_{m\in\N_0}X_t^{D,m}(i)\right)\dd t \\
		&+\frac{X_t^{D,k}(i)}{\sum_{m\in\N_0}X_t^{D,m}(i)}\hht_D\left(
			\sum_{m\in\N_0}X_t^{D,m}(i)
			\right)\dd t + \1_{k=0}h_D(0)\dd t \\
		&+\sqrt{\frac{X_t^{D,k}(i)}{\sum_{m\in\N_0}X_t^{D,m}(i)}\sigma^2\left(
			\sum_{m\in\N_0}X_t^{D,m}(i) \right)}\dd W_t^k(i),
		\quad t\in(0,\infty),\, (i,k)\in[D]\times\N_0,
	\end{split}
\end{equation}
where
$\{W^k(i)\colon (i,k)\in\N\times\N_0\}$ is a set of
independent standard Brownian motions.
Throughout this paper, we consider weak solutions of~\eqref{eq:XD_k} with
initial distribution and values in
$\{(x_{i,k})_{(i,k)\in[D]\times\N_0} \in
[0,1]^{[D]\times\N_0} \colon \sum_{k\in\N_0}x_{i,k} \in [0,1]
\text{ for all } i\in[D]\}$. Existence of such solutions can be shown as in~\textcite[Lemma~4.3]{Hutzenthaler2012}.
These processes will be referred to as
\emph{migration level processes}.

The following lemma shows that~\eqref{eq:XD} can be recovered
from~\eqref{eq:XD_k} by summing over all migration levels.
\begin{lemma}[Decomposition into migration levels]
\label{l:decomposition}
	Assume that \textup{Setting~\ref{set:coefficients}} holds, let $D\in\N$, let
	$\{(X_t^{D,k}(i), W_t^{k}(i))_{t\in[0,\infty)} \colon
	(i,k)\in[D]\times\N_0\}$ be a weak solution
	of~\eqref{eq:XD_k},
	and let $\{(W_t(i))_{t\in[0,\infty)} \colon i\in[D]\}$ be continuous
	adapted processes
	satisfying for all $i\in[D]$ and all $t\in[0,\infty)$ that
	a.s.
	\begin{equation}
		\begin{split}
			W_t(i)={} &\int_0^t \1_{\bigl\{\sum_{m\in\N_0}X_s^{D,m}(i)>0\bigr\}}
			\sum_{k\in\N_0}
			\sqrt{\frac{X_s^{D,k}(i)}{\sum_{m\in\N_0}X_s^{D,m}(i)}}\dd W_s^{k}(i)\\
			&+\int_0^t\1_{\bigl\{\sum_{m\in\N_0}X_s^{D,m}(i)=0\bigr\}}\dd
			W_s^{0}(i).
		\end{split}
		\label{eq:corresponding_BM}
	\end{equation}
	Then $\{(W(i))_{t\in[0,\infty)}\colon i\in[D]\}$ defines a
	$D$-dimensional standard Brownian motion and
	the process $\{(\Xt_t^D(i))_{t\in[0,\infty)}\colon
	i\in[D]\}$ defined for all $i\in[D]$ and all $t\in[0,\infty)$ by
	\begin{equation}
		\Xt_t^{D}(i) \defeq \sum_{k\in\N_0}X_{t}^{D,k}(i)
		\label{eq:all_levels}
	\end{equation}
	is the unique solution of~\eqref{eq:XD} with Brownian motion given
	by~\eqref{eq:corresponding_BM}.
\end{lemma}
\begin{proof}
	The processes $W(i)$, $i \in [D]$, defined
	by~\eqref{eq:corresponding_BM} are continuous local martingales whose
	cross-variation processes satisfy for all $i,j\in[D]$ and all
	$t\in[0,\infty)$ that $\langle W(i), W(j)\rangle_t = \delta_{ij}t$.
	L{\'e}vy's characterization of Brownian motion implies that $W$ is a
	$D$-dimensional standard Brownian motion. Moreover, it follows from
	summing~\eqref{eq:XD_k} over $k\in\N_0$ that
	$\Xt^{D}$ satisfies~\eqref{eq:XD} with Brownian motion given
	by~\eqref{eq:corresponding_BM}. Pathwise
	uniqueness of the SDE~\eqref{eq:XD} in the situation of
	Setting~\ref{set:coefficients} follows from Theorem~3.2
	in~\textcite{ShigaShimizu1980}. This
	finishes the proof of Lemma~\ref{l:decomposition}.
\end{proof}
In the limit as $D\to\infty$, the migration level processes are
essentially loop-free in the following sense. We define for all $D\in\N$ that
$Z^{D,-1} \defeq 0$ and consider for every $D\in\N$ the SDE
\begin{equation}
	\begin{split} \label{eq:ZD_k}
		dZ_t^{D,k}(i) ={}
		&\frac{1}{D} \sum_{j=1}^D Z_t^{D,k-1}(j)
			f\bigl(Z_t^{D,k-1}(j),
			Z_t^{D,k}(i)\bigr)\dd t
		+\hht_D\bigl(Z_t^{D,k}(i)\bigr) \dd t  + \1_{k=0}h_D(0)\dd t\\
			&+\sqrt{\sigma^2\bigl(Z_t^{D,k}(i)\bigr)}\dd W_t^k(i),\quad
			t\in(0,\infty),\, (i,k)\in[D]\times\N_0,
	\end{split}
\end{equation}
where $\{W^k(i) \colon (i,k)\in\N\times\N_0\}$ is a set of
independent standard Brownian motions.
Existence and uniqueness of strong solutions of~\eqref{eq:ZD_k}
follow from Theorem~3.2 in~\textcite{ShigaShimizu1980}. These
processes will be referred to as \emph{loop-free processes}.

\begin{setting}[Coupling of migration level and loop-free processes]
	\label{set:migration_levels}
	Assume that Setting~\ref{set:coefficients} holds. For every $D\in\N$ let
	$\{(X_t^{D,k}(i), W_t^{k}(i))_{t\in[0,\infty)} \colon
	(i,k)\in[D]\times\N_0\}$ be a weak solution of~\eqref{eq:XD_k}
	with initial distribution and
	values in
	$\{(x_{i,k})_{(i,k)\in[D]\times\N_0} \in
	[0,1]^{[D]\times\N_0} \colon \sum_{k\in\N_0}x_{i,k} \in [0,1]
	\text{ for all } i\in[D]\}$.
	For every $D\in\N$ and $x \in [0,1]^{[D]\times \N_0}$ we denote by
	$\{ (Z_t^{D,k,x}(i))_{t\in[0,\infty)} \colon (i,k)\in[D]\times\N_0 \}$
	continuous adapted processes that are defined on
	the stochastic basis given by the weak solution of~\eqref{eq:XD_k},
	satisfy~\eqref{eq:ZD_k} with Brownian motion given by the Brownian motion
	of the weak solution of~\eqref{eq:XD_k}, and further satisfy for all
	$(i,k)\in[D]\times\N_0$ that a.s.~$Z_0^{D,k,x}(i) =
	x_{i,k}$.  Whenever we omit the index $x$, we consider the solution
	of~\eqref{eq:ZD_k} satisfying for all $(i,k)\in[D]\times\N_0$
	that a.s.~$Z_0^{D,k}(i) = X_0^{D,k}(i)$. For notational simplicity,
	we do not distinguish notationally between the possibly different stochastic bases and
	Brownian motions for different $D\in\N$.
\end{setting}

\subsection{Moment and regularity estimates} \label{ss:preliminary_results}
In this subsection, we collect some preparatory results. We start with the
following lemma which provides an estimate for the first moment of the total
mass process.
\begin{lemma}[First moment]
	\label{l:X_Z_first_moment}
	Assume that \textup{Setting~\ref{set:migration_levels}} holds and let
	$T\in[0,\infty)$.
	Then we have for all $D\in\N$ that
	\begin{equation}
		\sup_{t\in[0,T]} \Exp*{\sum_{i=1}^D \sum_{k\in\N_0} X_t^{D,k}(i)} \leq
		e^{(L_f+L_h)T}\left( \Exp*{\sum_{i=1}^D\sum_{k\in\N_0} X_0^{D,k}(i)} + 2\mu T\right)
		\label{eq:X_Z_first_moment}
	\end{equation}
	and~\eqref{eq:X_Z_first_moment} holds with $X_t^{D,k}(i)$ replaced by $Z_t^{D,k}(i)$.
\end{lemma}
\begin{proof}
	For all $D\in\N$ let $\{(W_t(i))_{t\in[0,\infty)} \colon
	i\in[D]\}$ and $\{(\Xt_t^D(i))_{t\in[0,\infty)} \colon
	i\in[D]\}$ be as in Lemma~\ref{l:decomposition}.
	Setting~\ref{set:coefficients} implies for all $D\in\N$ and all $x,y\in
	[0,1]$
	that $f(y,x) \leq L_f$ and $h_D(x) \leq L_h x + 2\mu/D$. Together
	with Lemma~\ref{l:decomposition}, this shows for all
	$D\in\N$ and all $t\in[0,\infty)$ that a.s.
	\begin{equation}
		\begin{split}
			\sum_{i=1}^D \Xt_t^D(i) &= \sum_{i=1}^D \Xt_0^D(i)
			+ \sum_{i=1}^D\int_0^t \frac 1D \sum_{j=1}^D
			\Xt_s^{D}(j)f\bigl( \Xt_s^{D}(j), \Xt_s^{D}(i)\bigr) + h_D\bigl(
			\Xt_s^{D}(i) \bigr)\dd s\\
			&\quad+\sum_{i=1}^D\int_0^t\sqrt{\sigma^2\bigl(\Xt_s^{D}(i)\bigr)}\dd
			W_s(i)\\
			&\leq \sum_{i=1}^D \Xt_0^D(i) + 2\mu t
			+ (L_f+L_h)\int_0^t \sum_{i=1}^D \Xt_s^{D}(i) \dd s
			+\sum_{i=1}^D\int_0^t\sqrt{\sigma^2\bigl(\Xt_s^{D}(i)\bigr)}\dd
			W_s(i).
		\end{split}
		\label{eq:estimate_XD}
	\end{equation}
	The stochastic integrals on the right-hand side of~\eqref{eq:estimate_XD}
	are martingales since the integrands are globally bounded.
	Hence,~\eqref{eq:estimate_XD} and Tonelli's theorem imply for all
	$D\in\N$ and all $t\in[0,\infty)$ that
	\begin{equation}
		\Exp*{\sum_{i=1}^D \Xt_t^{D}(i)} \leq \Exp*{\sum_{i=1}^D \Xt_0^D(i)} + 2\mu t
		+ (L_f+L_h)\int_0^t \Exp*{\sum_{i=1}^D \Xt_s^{D}(i)} \dd s.
	\end{equation}
	Gronwall's inequality then yields for all
	$D\in\N$ and all $t\in[0,\infty)$ that
	\begin{equation}
		\Exp*{\sum_{i=1}^D
		\Xt_t^{D}(i)} \leq e^{(L_f+L_h)t}\left(\Exp*{\sum_{i=1}^D \Xt_0^D(i)} + 2\mu
		t\right).
	\end{equation}
	Taking the supremum over $t\in [0,T]$ and using~\eqref{eq:all_levels}
	proves~\eqref{eq:X_Z_first_moment}.  The proof for the loop-free processes
	is similar. This completes the proof of Lemma~\ref{l:X_Z_first_moment}.
\end{proof}
The following lemma is a variant of Lemma~\ref{l:X_Z_first_moment} and,
heuristically speaking, shows that uniformly in the number of demes
essentially only finitely many migration levels contribute to the total mass.
\begin{lemma}[Essentially only finitely many levels]
	\label{l:essentially_finitely_many_generations}
	Assume that \textup{Setting~\ref{set:migration_levels}} holds and that
	\begin{equation}
		\sum_{k\in\N_0} \sup_{D\in\N} \Exp*{\sum_{i=1}^D X_0^{D,k}(i)} < \infty.
		\label{eq:essentially_finitely_many_generations_ass}
	\end{equation}
	Then we have for all $T\in[0,\infty)$ that
	\begin{equation} \label{eq:essentially_finitely_many_generations}
		\sum_{k\in\N_0}\sup_{D\in\N}\sup_{t\in[0,T]}
		\Exp*{\sum_{i=1}^D X_t^{D,k}(i)}
		\leq e^{(L_f+L_h)T}\left(\sum_{k\in\N_0}\sup_{D\in\N}\Exp*{\sum_{i=1}^D
		X^{D,k}_0(i)}+2\mu T\right)
	\end{equation}
	and~\eqref{eq:essentially_finitely_many_generations} holds with $X_t^{D,k}(i)$
	replaced by $Z_t^{D,k}(i)$.
\end{lemma}
\begin{proof}
	In the situation of Setting~\ref{set:coefficients}, it holds for all
	$D\in\N$ and all $x,y\in [0,1]$ that $f(y,x) \leq L_f$, that $\hht_D(x) \leq L_h
	x$, and that $Dh_D(0) \leq 2\mu$. Moreover, the stochastic integral part
	of~\eqref{eq:XD_k} yields a martingale. These facts,~\eqref{eq:XD_k}, and Tonelli's theorem
	show for all $k\in\N_0$, all $D\in\N$, and all $t\in[0,\infty)$ that
	\begin{equation}
		\begin{split}
			\Exp*{\sum_{i=1}^D X_{t}^{D,k}(i)}
			\leq \Exp*{\sum_{i=1}^D X_0^{D,k}(i)}+\1_{k=0}2\mu t
			+\int_0^t
			L_f \Exp*{\sum_{i=1}^D X_{s}^{D,k-1}(i)}
			+L_h \Exp*{\sum_{i=1}^D
			X_{s}^{D,k}(i)}\dd s.
		\end{split}
	\end{equation}
	This implies for all $T\in[0,\infty)$ and all $k\in\N_0$ that
	\begin{equation}
		\begin{split}
			\sup_{D\in\N}\sup_{t\in[0,T]}\Exp*{\sum_{i=1}^D X_{t}^{D,k}(i)}
			\leq {} &\sup_{D\in\N}\Exp*{\sum_{i=1}^D
			X^{D,k}_0(i)}+\1_{k=0}2\mu T \\
			&+ L_f  \int_0^T \sup_{D\in\N}\sup_{u\in[0,s]} \Exp*{\sum_{i=1}^D
			X_{u}^{D,k-1}(i)} \dd s\\
			&+L_h \int_0^T \sup_{D\in\N}\sup_{u\in[0,s]} \Exp*{\sum_{i=1}^D
			X_{u}^{D,k}(i)}\dd s.
		\end{split}
		\label{eq:sum_over_k}
	\end{equation}
	Lemma~\ref{l:X_Z_first_moment}
	and~\eqref{eq:essentially_finitely_many_generations_ass} show that the
	right-hand side of~\eqref{eq:sum_over_k} is finite. For every $K\in\N$ a
	summation of~\eqref{eq:sum_over_k} over $k\in[K]_0$ and Gronwall's
	inequality yield for all $T\in[0,\infty)$ that
	\begin{equation}
		\sum_{k=0}^K\sup_{D\in\N}\sup_{t\in[0,T]} \Exp*{\sum_{i=1}^D
		X_{t}^{D,k}(i)}
		\leq e^{(L_f+L_h)T}\left(\sum_{k=0}^K\sup_{D\in\N}\Exp*{\sum_{i=1}^D
		X^{D,k}_0(i)}+2\mu T\right).
	\end{equation}
	Taking the limit as $K\to\infty$
	proves~\eqref{eq:essentially_finitely_many_generations}.  The proof for the
	loop-free processes is similar. This completes the proof of
	Lemma~\ref{l:essentially_finitely_many_generations}
\end{proof}
The following lemma gives an estimate for the second moment of the total mass
process.
\begin{lemma}[Second moment]
	\label{l:X_Z_second_moment}
	Assume that \textup{Setting~\ref{set:migration_levels}} holds.
	Then we have for all $D\in\N$ and all $T\in[0,\infty)$ that
	\begin{equation}
		\begin{split}
			\MoveEqLeft\Exp*{\sup_{t\in[0,T]} \left(\sum_{i=1}^D \sum_{k\in\N_0}
			X_t^{D,k}(i) \right)^2}\\
			&\leq e^{(8L_\sigma +
			4(L_f+L_h)^2T)T}\left( 4\Exp*{\left(\sum_{i=1}^D\sum_{k\in\N_0}
			X_0^{D,k}(i)\right)^2} + 8T(L_\sigma+2\mu^2 T) \right)
		\end{split}
		\label{eq:X_Z_second_moment}
	\end{equation}
	and~\eqref{eq:X_Z_second_moment} holds with $X_t^{D,k}(i)$ replaced by $Z_t^{D,k}(i)$.
\end{lemma}
\begin{proof}
	For all $D\in\N$ let
	$\{(W_t(i))_{t\in[0,\infty)} \colon i\in[D]\}$ and
	$\{(\Xt^D_t(i))_{t\in[0,\infty)} \colon i\in[D]\}$ be as in
	Lemma~\ref{l:decomposition}.
	Setting~\ref{set:coefficients} implies for all $D\in\N$
	and all $x,y \in [0,1]$ that $\abs{f(y,x)}\leq L_f$ and $\abs{h_D(x)}\leq L_h x
	+ 2\mu/D$. 
	This and Lemma~\ref{l:decomposition} imply for all $D\in\N$ and all
	$t\in [0,\infty)$ that a.s.
	\begin{equation}
		\abs*{\sum_{i=1}^D \Xt_t^D(i)} \leq \abs*{\sum_{i=1}^D \Xt_0^D(i)}
		+ (L_f+L_h)\int_0^t \abs*{\sum_{i=1}^D \Xt_s^{D}(i)} \dd s + 2\mu t
		+\abs*{\sum_{i=1}^D\int_0^t\sqrt{\sigma^2\bigl(\Xt_s^{D}(i)\bigr)}\dd
		W_s(i)}.
	\end{equation}
	The Minkowski inequality then implies for all $D\in\N$ and all $T\in[0,\infty)$
	that
	\begin{equation}
		\begin{split}
			\Exp*{\sup_{t\in[0,T]} \left(\sum_{i=1}^D
			\Xt_t^{D}(i)\right)^2}^{\frac{1}{2}}
			&\leq \Exp*{\left(\sum_{i=1}^D \Xt_0^D(i)\right)^2}^{\frac{1}{2}} + (L_f+L_h)\int_0^T
			\Exp*{\left(\sum_{i=1}^D \Xt_s^{D}(i)\right)^2}^{\frac 12}\dd s\\
			&\quad + 2\mu T + \Exp*{\sup_{t\in[0,T]}\abs*{\sum_{i=1}^D \int_0^t
			\sqrt{\sigma^2\bigl(\Xt_s^{D}(i)\bigr)}\dd W_s(i)}^2}^{\frac 12}.
		\end{split}
		\label{eq:second_moment_firststep}
	\end{equation}
	Using Doob's $L^2$-inequality (e.g.\ Corollary~2.2.17 in~\textcite{EthierKurtz1986}), the It{\^o} isometry,
	Setting~\ref{set:coefficients}, and the fact that for all
	$x\in\R$ it holds that $2x\leq
	1+x^2$, we obtain for all $D\in\N$ and all $T\in[0,\infty)$ that
	\begin{equation}
		\begin{split}
			\Exp*{\sup_{t\in[0,T]}\abs*{\sum_{i=1}^D \int_0^t
			\sqrt{\sigma^2\bigl(\Xt_s^{D}(i)\bigr)}\dd W_s(i)}^2}
			&\leq 4\Exp*{\sum_{i=1}^D \int_0^T
			\sigma^2\bigl(\Xt_s^{D}(i)\bigr)\dd s}\\
			&\leq 4\int_0^T
			L_\sigma\Exp*{\sum_{i=1}^D \Xt_s^{D}(i)}\dd s\\
			&\leq 2L_\sigma T + 2L_\sigma\int_0^T\Exp*{
			\left(\sum_{i=1}^D \Xt_s^{D}(i)\right)^2}\dd s.
		\end{split}
		\label{eq:second_moment_doob}
	\end{equation}
	Equations~\eqref{eq:second_moment_firststep}
	and~\eqref{eq:second_moment_doob}, the fact that it holds for all
	$x_1,\dotsc,x_4 \in \R$ that $(\sum_{i=1}^4
	x_i)^2 \leq 4 \sum_{i=1}^4 x_i^2$, and Hölder's inequality
	yield for all $D\in\N$ and all
	$T\in[0,\infty)$ that
	\begin{equation}
		\begin{split}
			\Exp*{\sup_{t\in[0,T]} \left(\sum_{i=1}^D \Xt_t^{D}(i)\right)^2}
			&\leq 4\Exp*{\left(\sum_{i=1}^D \Xt_0^D(i)\right)^2} + 16\mu^2 T^2 \\
			&\quad +
			4(L_f+L_h)^2\left(\int_0^T
			\Exp*{\left(\sum_{i=1}^D \Xt_s^{D}(i)\right)^2}^{\frac 12}\dd s\right)^2\\
			&\quad + 8L_\sigma T
			+ 8L_\sigma\int_0^T\Exp*{
			\left(\sum_{i=1}^D \Xt_s^{D}(i)\right)^2}\dd s\\
			&\leq 4\Exp*{\left(\sum_{i=1}^D \Xt_0^D(i)\right)^2} + 8T(L_\sigma+2\mu^2
			T)\\
			&\quad+ \bigl(8L_\sigma+ 4(L_f+L_h)^2T\bigr) \int_0^T\Exp*{
			\sup_{u\in[0,s]}\left(\sum_{i=1}^D \Xt_u^{D}(i)\right)^2}\dd s.
		\end{split}
	\end{equation}
	Gronwall's inequality then yields
	for all $D\in\N$ and all $T\in[0,\infty)$ that
	\begin{equation}
		\Exp*{\sup_{t\in[0,T]} \left(\sum_{i=1}^D \Xt_t^{D}(i) \right)^2} \leq
		e^{(8L_\sigma + 4(L_f+L_h)^2 T)T}\left( 4\Exp*{\left(\sum_{i=1}^D
		\Xt_0^{D}(i)\right)^2} + 8T(L_\sigma+2\mu^2 T) \right).
	\end{equation}
	Together with~\eqref{eq:all_levels}, this proves~\eqref{eq:X_Z_second_moment}.
	The proof for the loop-free processes is similar. This completes the proof
	of Lemma~\ref{l:X_Z_second_moment}.
\end{proof}
The following lemma is a consequence of Lemma~\ref{l:X_Z_second_moment} and
allows for a localization procedure in the total mass uniformly in the number
of demes.
\begin{lemma}[Localization argument]
	\label{l:tau_theta}
	Assume that \textup{Setting~\ref{set:migration_levels}} and
	\begin{equation}
		\sup_{D\in\N}\Exp*{\left(\sum_{i=1}^D\sum_{k\in\N_0}
		X_0^{D,k}(i)\right)^2} < \infty
		\label{eq:tau_theta_ass}
	\end{equation}
	hold and for every $D,M\in\N$ define the stopping time
	\begin{equation}
		\tau_M^D \defeq \inf\left\{ t\in[0,\infty) \colon
		\sum_{i=1}^D\sum_{m\in\N_0} X_t^{D,m}(i) \geq
		M \right\}.
		\label{eq:tau_MD}
	\end{equation}
	Then it holds for all $T\in[0,\infty)$ that
	\begin{equation}
		\lim_{M\to\infty}\sup_{D\in\N}
		\Exp*{\sup_{t\in[0,T]}\sum_{i=1}^D\sum_{m\in\N_0}X_t^{D,m}(i)\1_{\{\tau_M^D\leq
		T\}}} = 0.
		\label{eq:tau_theta}
	\end{equation}
\end{lemma}
\begin{proof}
	For all $D,M\in\N$ and $T\in[0,\infty)$ we have that $\{\tau_M^D\leq T\}
	=\{\sup_{t\in[0,T]}\sum_{i=1}^D
	\sum_{m\in\N_0}X_t^{D,m}(i)\geq M\}$. This implies for all $D,M\in\N$ and
	all $T\in[0,\infty)$ that
	\begin{equation}
		\Exp*{\sup_{t\in[0,T]}\sum_{i=1}^D\sum_{m\in\N_0}X_t^{D,m}(i)\1_{\{\tau_M^D\leq
		T\}}}
		\leq
		\frac{1}{M}\Exp*{\sup_{t\in[0,T]}\left(\sum_{i=1}^D\sum_{m\in\N_0}X_t^{D,m}(i)\right)^2}.
	\end{equation}
	This,
	Lemma~\ref{l:X_Z_second_moment}, and~\eqref{eq:tau_theta_ass}
	show~\eqref{eq:tau_theta}. This completes the proof of
	Lemma~\ref{l:tau_theta}.
\end{proof}
Throughout the rest of this subsection and in
Subsection~\ref{ss:poisson_limit} below, the following Setting~\ref{set:g_D} will frequently
be referred to. In the situation of Setting~\ref{set:g_D}, for every $D\in\N$ the
SDE~\eqref{eq:SDE_YD} below with $g = g_D$ has a unique strong solution with continuous
sample paths in $[0,1]$; see, e.g., Theorem~5.4.22, Proposition~5.2.13, and
Corollary~5.3.23 in \textcite{KaratzasShreve1991}. This fact will be used
tacitly in the remainder of this paper.
\begin{setting}[Time-dependent immigration]\label{set:g_D}
	Assume that Setting~\ref{set:coefficients} holds and that $g_D\colon
	[0,\infty)\times [0,1] \to \R$, $D\in\N$, are measurable functions that satisfy for all
	$D\in\N$ and all $t\in[0,\infty)$ that $g_D(t,0)\geq 0$, that
	$\frac{1}{D}g_D(t,1) + \hht_D(1) \leq 0$, that
	\begin{equation}
		\sup_{u\in[0,\infty)}\sup_{\substack{x,y\in [0,1] \\ x\neq y}}\frac{\abs{g_D(u,x) -
		g_D(u,y)}}{\abs{x-y}} < \infty,
		\label{eq:g_D_Lipschitz}
	\end{equation}
	and that
	\begin{equation}
		\sup_{M\in\N}\int_0^t \sup_{x\in [0,1]}\abs{g_M(u,x)}^2\dd u <
		\infty.\qedhere
		\label{eq:g_D_squareint}
	\end{equation}
\end{setting}
For all $s\in[0,\infty)$, all $D\in\N$, and all measurable functions $g\colon [0,\infty) \times
[0,1] \to \R$ we consider the one-dimensional SDE
\begin{equation}
	dY_{t,s}^{D,g} = \tfrac{1}{D}g\bigl(t,Y_{t,s}^{D,g}\bigr) \dd t
	+\hht_D\bigl(Y_{t,s}^{D,g}\bigr)\dd t
	+\sqrt{\sigma^2\bigl(Y_{t,s}^{D,g}\bigr)}\dd W_t,\quad t \in [s,\infty),
	\label{eq:SDE_YD}
\end{equation}
where $W$ is a standard Brownian motion. We adopt the same notation and write,
for example, $Y_{t,s}^{D,\zeta}$ (or $Y_{t,s}^{D,c}$) when the function
$g\colon [0,\infty)\times[0,1]\to \R$ in~\eqref{eq:SDE_YD} is replaced by a
function $\zeta\colon [0,\infty) \to \R$ (or by a constant $c \in \R$).

The following lemma estimates the $L^1$-distance between certain solutions
of~\eqref{eq:SDE_YD}.
\begin{lemma}[$L^1$-regularity]
	\label{l:YD_L1_distance}
	Assume that \textup{Setting~\ref{set:coefficients}} holds, let
	$g_D\colon[0,\infty)\times [0,1]\to \R$, $D\in\N$, and
	$\gt_D\colon[0,\infty)\times [0,1]\to \R$, $D\in\N$, be two sequences of
	functions satisfying \textup{Setting~\ref{set:g_D}}, let $s\in[0,\infty)$,
	and for every $D\in\N$ let $(Y^{D,g_D}_{t,s})_{t\in[s,\infty)}$ and
	$(Y^{D,\gt_D}_{t,s})_{t\in[s,\infty)}$ be solutions of~\eqref{eq:SDE_YD}
	with respect to the same Brownian motion.
	Then it holds for all $D\in\N$ and all $t\in[s,\infty)$ that
	\begin{equation*}
		\begin{split}
			\Exp[\big]{\abs[\big]{Y_{t,s}^{D,g_D} - Y_{t,s}^{D,\gt_D}}}
			\leq e^{L_h(t-s)}\biggl(\Exp[\big]{\abs[\big]{Y^{D,g_D}_{s,s} -
			Y^{D,\gt_D}_{s,s}}} 
			+ \tfrac{1}{D}
			\int_s^t
			\Exp[\big]{\abs[\big]{g_D\bigl(u,Y^{D,g_D}_{u,s}\bigr) -
			\gt_D(u,Y^{D,\gt_D}_{u,s})}}\dd u\biggr).
		\end{split}
		%\label{eq:YD_L1_distance}
	\end{equation*}
\end{lemma}
\begin{proof}
	As in Theorem~1 in \textcite{YamadaWatanabe1971} (see also, e.g., the proof
	of Lemma~3.3 in \textcite{HutzenthalerWakolbinger2007})
	an approximation of $\R\ni x\mapsto \abs{x}\in\R$ with $C^2$-functions and
	exploiting that $\sup_{x,y\in[0,1],x\neq
	y}\frac{\abs{\sqrt{\sigma^2(x)}-\sqrt{\sigma^2(y)}}^2}{\abs{x-y}}<\infty$
  shows for
	all $D\in\N$ and all $t\in[s,\infty)$ that a.s.
	\begin{equation}
		\abs[\big]{Y_{t,s}^{D,g_D} - Y_{t,s}^{D,\gt_D}} =
		\abs[\big]{Y^{D,g_D}_{s,s} - Y^{D,\gt_D}_{s,s}} +
		\int_s^t\sgn\bigl(Y_{u,s}^{D,g_D}-Y_{u,s}^{D,\gt_D}\bigr)\dd\bigl(Y_{u,s}^{D,g_D}
		- Y_{u,s}^{D,\gt_D}\bigr).
		\label{eq:l1_yamada}
	\end{equation}
	For every $D\in\N$ and every $t\in[s,\infty)$ let $M_t^{D}$ be a real-valued
	random variable satisfying a.s.~that
	\begin{equation}
		M^{D}_t = \int_s^t \sgn\bigl(Y_{u,s}^{D,g_D} -
		Y_{u,s}^{D,\gt_D}\bigr)
		\left(\sqrt{\sigma^2\bigl(Y^{D,g_D}_{u,s}\bigr)} -
		\sqrt{\sigma^2\bigl(Y^{D,\gt_D}_{u,s}\bigr)}\right)\dd W_u.
		\label{eq:l1_MDn}
	\end{equation}
	Then~\eqref{eq:l1_yamada} and Setting~\ref{set:coefficients}
	imply for all $D\in\N$
	and all $t\in[s,\infty)$ that a.s.
	\begin{equation}
		\begin{split}
			\abs[\big]{Y_{t,s}^{D,g_D} - Y_{t,s}^{D,\gt_D}}
			&\leq
			\abs[\big]{Y^{D,g_D}_{s,s} - Y^{D,\gt_D}_{s,s}} +
			\tfrac{1}{D}\int_s^t \abs[\big]{g_D\bigl(u,Y^{D,g_D}_{u,s}\bigr) -
			\gt_D\bigl(u,Y^{D,\gt_D}_{u,s}\bigr)}\dd u \\
			&\quad+ L_h\int_s^t
			\abs[\big]{Y^{D,g_D}_{u,s} -
			Y^{D,\gt_D}_{u,s}} \dd u  + M^{D}_t.
		\end{split}
		\label{eq:l1_yamada2}
	\end{equation}
	Since the integrand of the stochastic integral in~\eqref{eq:l1_MDn} is globally
	bounded, it holds for all $D\in\N$ and all
	$t\in[s,\infty)$ that $\Exp{M^{D}_t} = 0$.
	Therefore,~\eqref{eq:l1_yamada2} and Tonelli's theorem imply for
	all $D\in\N$ and all $t\in[s,\infty)$ that
	\begin{equation}
		\begin{split}
			\Exp[\big]{\abs[\big]{Y_{t,s}^{D,g_D} - Y_{t,s}^{D,\gt_D}}}
			\leq{} &
			\Exp[\big]{\abs[\big]{Y^{D,g_D}_{s,s} - Y^{D,\gt_D}_{s,s}}} +
			\tfrac{1}{D}\int_s^t
			\Exp[\big]{\abs[\big]{g_D\bigl(u,Y^{D,g_D}_{u,s}\bigr) -
			\gt_D(u,Y^{D,\gt_D}_{u,s})}}\dd u \\
			&+ L_h \int_s^t
			\Exp[\big]{\abs[\big]{Y_{u,s}^{D,g_D} -
			Y_{u,s}^{D,\gt_D}}}\dd u.
		\end{split}
		\label{eq:l1_pre_gronwall}
	\end{equation}
	Gronwall's inequality then finishes the proof of Lemma~\ref{l:YD_L1_distance}.
\end{proof}
The following lemma provides us with a second moment estimate for solutions
of~\eqref{eq:SDE_YD} and is similar to Lemma~\ref{l:X_Z_second_moment}. Its
proof is completely analogous to that of Lemma~\ref{l:X_Z_second_moment} and
thus omitted here.
\begin{lemma}[Second moment]
	\label{l:YD_second_moment}
	Assume that \textup{Setting~\ref{set:g_D}} holds, let
	$s\in[0,\infty)$, 
	and for every $D\in\N$ let $(Y^{D,g_D}_{t,s}(i))_{t\in[s,\infty)}$,
	$i\in[D]$, be independent solutions of~\eqref{eq:SDE_YD}
	satisfying for all $i\in[D]$ a.s.~that
	$Y^{D,g_D}_{s,s}(i) = 0$.
	Then it holds for all $D\in\N$ and all $T\in[s,\infty)$ that
	\begin{equation}
		\Exp*{\sup_{t\in[s,T]}\left(\sum_{i=1}^D
		Y_{t,s}^{D,g_D}(i)\right)^2}
		\leq 3(T-s)e^{(6L_\sigma+3L_h^2(T-s))(T-s)}\left(\int_s^T \sup_{x\in [0,1]}\abs{g_D(u,x)}^2 \dd
		u + 2L_\sigma\right).
		\label{eq:YD_sum_second_moment}
	\end{equation}
\end{lemma}

The following lemma follows from a straightforward adaptation of Lemma~9.8
in~\textcite{Hutzenthaler2009} to the state space $[0,1]$.
\begin{lemma}[Finite excursion area] \label{l:finite_excursion_area}
	Assume that \textup{Setting~\ref{set:excursion_measure}} holds.
	Then it holds that
	\begin{equation}
		\int \int_0^\infty \eta_t\dd t \di Q(d\eta)
		=\int_0^{1}\frac{y}{\frac{1}{2}\sigma^2(y)s(y)} \dd y
		<\infty.
	\end{equation}
\end{lemma}

\subsection{Poisson limit of independent diffusions with vanishing
immigration} \label{ss:poisson_limit}
To show the convergence of the loop-free processes in
Subsection~\ref{ss:convergence_loop_free} below, we first prove a Poisson
limit for independent diffusions with vanishing immigration, see
Lemma~\ref{l:vanishing_immigration_weak_process} below, based on the
following lemma which is essentially Lemma~4.19 in~\textcite{Hutzenthaler2012}
and proved there utilizing a time reversion argument.
\begin{lemma}[Poisson limit, constant case] \label{l:vanishing_immigration}
	Assume that \textup{Setting~\ref{set:excursion_measure}} holds,
	let $c,s\in[0,\infty)$, let $D_0 \in \N$ be such that it holds for all
	$D\in\N \cap [D_0,\infty)$ that $c/D + \hht_D(1)
	\leq 0$,
	for every $D\in\N\cap[D_0,\infty)$ let $(Y^{D,c}_{t,s})_{t\in[s,\infty)}$ be
	a solution of~\eqref{eq:SDE_YD} satisfying a.s.~that $Y^{D,c}_{s,s} = 0$,
	and let $\phi_D\colon [0,1]\to\R$, $D\in\N_0$,
	be functions with the property that
	\begin{equation}
		\sup_{\substack{x,y\in
		[0,1] \\ x\neq y}}\sup_{D\in\N}\frac{\abs{\phi_D(x)-\phi_D(y)}}{\abs{x-y}}<\infty,
	\end{equation}
	that ${\lim_{D\to\infty}D|\phi_D(0)|=0}$,
	and for all $y\in [0,1]$ that
	$\lim_{D\to\infty}\phi_D(y)=\phi_0(y)$. Then it holds for all
	$t\in[s,\infty)$ that
	\begin{equation} \label{eq:vanishing_immigration}
		\lim_{D\to\infty} D\Exp*{\phi_D\bigl(Y_{t,s}^{D,c}\bigr)}
		= c\int_s^t\int \phi_0(\eta_{t-u})\di Q(d\eta)\dd u.
	\end{equation}
\end{lemma}
For every $T\in(0,\infty)$ and every $s\in[0,T)$ we define
\begin{equation}
	\MCE_{s,T} \defeq
	\begin{Bmatrix}
	C([s,T],[0,1]) \ni \eta \mapsto
	\prod_{i=1}^n \psi_i(\eta_{t_i}) \in\R \colon
	\text{$n\in\N$, $\psi_1,\dotsc,\psi_n \in \Lip([0,1],\R)$,}\\
	\text{and $t_1,\dotsc,t_n\in[s,T]$
	with $s \leq t_1 \leq \dotsb \leq t_n \leq T$}
	\end{Bmatrix}.
	\label{eq:MCE_sT}
\end{equation}
From the Lipschitz continuity and boundedness of the involved functions
in~\eqref{eq:MCE_sT}, it follows for all $T\in(0,\infty)$ and all $s\in[0,T)$ that
the elements of $\MCE_{s,T}$ are globally Lipschitz continuous in the sense of
Lemma~\ref{l:F_is_globally_Lipschitz}. The proof of
Lemma~\ref{l:F_is_globally_Lipschitz} is straightforward and therefore omitted.
\begin{lemma}[Lipschitz continuity] \label{l:F_is_globally_Lipschitz}
	Let $n\in\N$ and let $\psi_1,\dotsc,\psi_n \in
	\Lip([0,1],\R)$.
	Then there exists a constant $L \in [0,\infty)$ such that it holds for all
	$x_1,\dotsc,x_n \in [0,1]$ and all $y_1,\dotsc,y_n \in [0,1]$ that
	\begin{equation}\label{eq:F_is_globally_Lipschitz}
		\abs*{\prod_{i=1}^n \psi_i(x_i) - \prod_{i=1}^n
		\psi_i(y_i)} \leq L\sum_{i=1}^n \abs{x_i -
		y_i}.
	\end{equation}
\end{lemma}
The following two lemmas generalize Lemma~\ref{l:vanishing_immigration} in a
suitable way. The proof of Lemma~\ref{l:vanishing_immigration_piecewise}
below is analogous to the first part of the proof of Lemma~4.20
in~\textcite{Hutzenthaler2012} and therefore omitted here.
\begin{lemma}[Poisson limit, piecewise constant case]
	\label{l:vanishing_immigration_piecewise}
	Assume that \textup{Setting~\ref{set:excursion_measure}} holds and
	let $T\in (0,\infty)$.
	Then for all $n\in\N$, all $s\in[0,T)$,
	all $c_1,\dotsc,c_n \in [0,\infty)$, all
		$\psi_1,\dotsc,\psi_n \in \Lip([0,1],\R)$, all $t_0,\dotsc,t_n \in [s,T]$
	with $s = t_0 \leq t_1 \leq \dotsb \leq t_n \leq T$,
	all $\zeta\colon [0,\infty) \to [0,\infty)$ satisfying for all
	$t\in[0,\infty)$ that
	$\zeta(t) = \sum_{i=1}^n c_i \1_{[t_{i-1},t_i)}(t)$,
	all $D_0 \in \N$ such that it holds for all $D\in\N \cap [D_0,\infty)$
	that $\max\{c_1,\dotsc,c_n\}/D + \hht_D(1) \leq 0$,
	all solutions $(Y^{D,\zeta}_{t,s})_{t\in[s,\infty)}$, $D\in\N \cap
	[D_0,\infty)$, of~\eqref{eq:SDE_YD} satisfying for all
	$D\in\N\cap[D_0,\infty)$ a.s.~that $Y^{D,\zeta}_{s,s} = 0$,
	and all $F \in \MCE_{s,T}$ satisfying for all $\eta \in C([s,T],[0,1])$ that
	$F(\eta) = \prod_{i=1}^n \psi_i(\eta_{t_i})$ and $F(0) = 0$ it holds that
	\begin{equation}
		\lim_{D\to\infty}D\Exp*{ F\left( \bigl(Y^{D,\zeta}_{t,s}\bigr)_{t\in[s,T]}
		\right) } = \int_s^T
		\zeta(u)\int F\left( (\eta_{t-u})_{t\in[s,T]} \right)
		\di Q(d\eta)\dd u.
		\label{eq:vanishing_immigration_piecewise}
	\end{equation}
\end{lemma}
The following lemma generalizes Lemma~\ref{l:vanishing_immigration_piecewise}
and is a crucial ingredient in the proof of
Lemma~\ref{l:vanishing_immigration_weak_process} below.
\begin{lemma}[Poisson limit, general case]\label{l:vanishing_immigration_general}
	Assume that \textup{Setting~\ref{set:excursion_measure}} and
	\textup{Setting~\ref{set:g_D}} hold,
	let $T\in(0,\infty)$, let $s\in[0,T)$,
	for every $D\in\N$ let $(Y^{D,g_D}_{t,s})_{t\in[s,\infty)}$ be a solution
	of~\eqref{eq:SDE_YD} satisfying a.s.~that
	$Y^{D,g_D}_{s,s} = 0$,
	and let $g\colon [0,\infty)\times [0,1] \to \R$ be a measurable function satisfying
	for all $t\in[0,\infty)$ that $g(t,0) \geq 0$, that $[0,1] \ni x \mapsto
	g(t,x)$ is continuous, and that
	\begin{equation}
		\lim_{D\to\infty}\int_s^T \sup_{x\in [0,1]}\abs{g_D(u,x)-g(u,x)}\dd
		u = 0.
		\label{eq:g_D_convergence}
	\end{equation}
	Then it holds for all $F\in\MCE_{s,T}$ with $F(0) = 0$ that
	\begin{equation}
		\lim_{D\to\infty}D\E\left[ F\left( \bigl(Y^{D,g_D}_{t,s}\bigr)_{t\in[s,T]} \right)
		\right] = \int_s^\infty g(u,0)\int F\left( (\eta_{t-u})_{t\in[s,T]}
		\right) \di Q(d\eta)\dd u \in \R.
		\label{eq:vanishing_immigration_general}
	\end{equation}
\end{lemma}
\begin{proof}
	In a first step, we assume that $(g_D)_{D\in\N}$ and $g$ are uniformly
	bounded by $K \in \N$.
	Fix $F\in\MCE_{s,T}$ with $F(0) = 0$ for the rest of the proof and let $m\in\N$,
	$\psi_1,\dotsc,\psi_m\in\Lip([0,1],\R)$, and $t_1,\dotsc,t_m\in[s,T]$ with $s\leq t_1
	\leq \dotsb \leq t_m \leq T$ be such that it holds for all $\eta \in
	C([s,T],[0,1])$ that $F(\eta) = \prod_{i=1}^m \psi_i(\eta_{t_i})$. 
	We choose step functions
	$\zeta^{(n)}\colon [s,T] \to [0,\infty)$, $n\in\N$, with the property that
	$\zeta^{(n)}(\cdot) \to g(\cdot,0)$ almost everywhere as $n\to\infty$ and
	such that it holds for all $n\in\N$ that
	$\zeta^{(n)} \leq K$. For every $n\in\N$ we extend $\zeta^{(n)}$ to
	$[0,\infty)$ by setting it to zero outside of $[s,T]$.
	Setting~\ref{set:coefficients} implies the existence of $D_0\in\N$
	such that
	we have for all $D\in\N\cap [D_0,\infty)$ that $K/D + \hht_D(1) \leq 0$. For
	every $n\in\N$ and every $D\in\N\cap [D_0,\infty)$ let
	$(Y^{D,\zeta^{(n)}}_{t,s})_{t\in[s,\infty)}$ be
	a solution of~\eqref{eq:SDE_YD} satisfying a.s.~that
	$Y^{D,\zeta^{(n)}}_{s,s} = 0$.
	Since we
	may let $F$ depend trivially on further time points,
	Lemma~\ref{l:vanishing_immigration_piecewise} yields for every $n\in\N$ that
	\begin{equation}
		\lim_{D\to\infty}D\E\left[ F\left(
			\bigl(Y^{D,\zeta^{(n)}}_{t,s}\bigr)_{t\in[s,T]} \right) \right] =
			\int_s^T \zeta^{(n)}(u)\int F\left( (\eta_{t-u})_{t\in[s,T]}
		\right) \di Q(d\eta)\dd u.
		\label{eq:vanishing_immigration_partitions}
	\end{equation}
	We are going to show that~\eqref{eq:vanishing_immigration_partitions} converges
	to~\eqref{eq:vanishing_immigration_general} as $n\to\infty$.
	Let
	$L_F \in [0,\infty)$ be a Lipschitz constant of $F$ in the sense of
	Lemma~\ref{l:F_is_globally_Lipschitz}.
	Then Lemma~\ref{l:F_is_globally_Lipschitz} and
	Lemma~\ref{l:YD_L1_distance} applied with $\gt_D = \zeta^{(n)}$ imply for all
	$n\in\N$ and all $D\in\N\cap [D_0,\infty)$ that
	\begin{equation}
		\begin{split}
			\MoveEqLeft \abs*{D\E\left[ F\left(
			\bigl(Y_{t,s}^{D,g_D}\bigr)_{t\in[s,T]}
			\right)\right] - D\E\left[ F\left( \bigl(Y_{t,s}^{D,\zeta^{(n)}}\bigr)_{t\in[s,T]}
			\right)\right]}\\
			&\leq DL_F \sum_{i=1}^m
			\E\bigl[\abs[\big]{Y_{t_i,s}^{D,g_D} - Y_{t_i,s}^{D,\zeta^{(n)}}}\bigr]\\
			&\leq mL_F e^{L_h(T-s)}\int_s^T \Exp[\big]{\abs[\big]{g_D\bigl(u,Y^{D,g_D}_{u,s}\bigr) -
			\zeta^{(n)}(u)}} \dd u\\
			&\leq mL_F e^{L_h(T-s)}
			\biggl( \int_s^T \sup_{x\in [0,1]}\abs{g_D(u,x)-g(u,x)} \dd u + \int_s^T
			\E\bigl[\abs[\big]{g\bigl(u,Y^{D,g_D}_{u,s}\bigr) - g(u,0)}\bigr]\dd u\\
			&\quad+\int_s^T
			\abs[\big]{g(u,0)-\zeta^{(n)}(u)} \dd u
			\biggr) .
		\end{split}
		\label{eq:immigration_distance_approx}
	\end{equation}
	The first summand on the right-hand side
	of~\eqref{eq:immigration_distance_approx} converges to zero as $D\to\infty$ by~\eqref{eq:g_D_convergence}. The dominated convergence theorem
	and the fact that $Y^{D,g_D}_{\cdot,s}$ converges to zero in distribution as
	$D\to\infty$
	yield that the second summand on the right-hand side
	of~\eqref{eq:immigration_distance_approx} converges to zero as $D\to\infty$. Finally, the
	dominated convergence theorem ensures that
	the third summand on the right-hand side
	of~\eqref{eq:immigration_distance_approx} converges to zero as $n\to\infty$. Altogether, it follows that
	\begin{equation}
		\begin{split}
			\lim_{n\to\infty}\lim_{D\to\infty}{\abs*{D\E\left[ F\left(
			\bigl(Y_{t,s}^{D,g_D}\bigr)_{t\in[s,T]}
			\right)\right] - D\E\left[ F\left( \bigl(Y_{t,s}^{D,\zeta^{(n)}}\bigr)_{t\in[s,T]}
			\right)\right]}} = 0.
		\end{split}
	\end{equation}
	This proves convergence of the left-hand side
	of~\eqref{eq:vanishing_immigration_partitions} to the left-hand side
	of~\eqref{eq:vanishing_immigration_general} as $n\to\infty$.
	Lemma~\ref{l:F_is_globally_Lipschitz}, $F(0)=0$, and
	Lemma~\ref{l:finite_excursion_area} ensure that $\int_s^T \int \abs{F(
	(\eta_{t-u})_{t\in[s,T]} )}\di Q(d\eta) \dd u < \infty$. This, the fact that
	we have for all $n\in\N$ that $\zeta^{(n)} \leq K$, and the
	dominated convergence theorem show that
	\begin{equation}
		\begin{split}
			\lim_{n\to\infty} \int_s^T \zeta^{(n)}(u) \int F\left( (\eta_{t-u})_{t\in[s,T]} \right)
			\di Q(d\eta) \dd u = \int_s^T g(u,0)\int F\left( (\eta_{t-u})_{t\in[s,T]} \right)
			\di Q(d\eta) \dd u.
		\end{split}
		\label{eq:vig_almostRHS}
	\end{equation}
	It remains to note that $F(0) = 0$ ensures
	\begin{equation}
		\int_s^T g(u,0)\int F\left( (\eta_{t-u})_{t\in[s,T]} \right)
		\di Q(d\eta)\dd u = \int_s^\infty g(u,0) \int F\left(
		(\eta_{t-u})_{t\in[s,T]} \right) \di Q(d\eta)\dd u.
		\label{eq:vig_expand_integral}
	\end{equation}
	Hence, \eqref{eq:vig_almostRHS} and~\eqref{eq:vig_expand_integral} show that
	the right-hand sides of~\eqref{eq:vanishing_immigration_partitions}
	and~\eqref{eq:vanishing_immigration_general} are equal in the limit
	$n\to\infty$.

	For the rest of the proof, we return to the case of general
	$(g_D)_{D\in\N}$ and $g$.
	For all $D,K\in\N$ let $(Y^{D,g_D\wedge K}_{t,s})_{t\in[s,\infty)}$
	be a solution of~\eqref{eq:SDE_YD} satisfying a.s.~that $Y^{D,g_D\wedge
	K}_{s,s} = 0$.
	%(for existence and uniqueness of strong solutions of~\eqref{eq:SDE_YD} see,
	%e.g., Theorem~5.4.22, Proposition~5.2.13, and Corollary~5.3.23 in
	%\textcite{KaratzasShreve1991}).
	It holds for all $K\in\N$ that
	\begin{equation}
		\lim_{D\to\infty}\int_s^T \sup_{x\in [0,1]}\abs{g_D(u,x)\wedge K - g(u,x)\wedge K}\dd u \leq
		\lim_{D\to\infty}\int_s^T \sup_{x\in [0,1]}\abs{g_D(u,x)- g(u,x)}\dd u = 0.
		\label{eq:vig_assK}
	\end{equation}
	We note that~\eqref{eq:g_D_squareint} and~\eqref{eq:g_D_convergence}
	imply that
	\begin{equation}
		\int_s^T \sup_{x\in[0,1]}\abs{g(u,x)} \dd u < \infty.
		\label{eq:int_g_finite}
	\end{equation}
	Lemma~\ref{l:YD_L1_distance} applied with $(\gt_D)_{D\in\N} = (g_D\wedge
	K)_{D\in\N}$, arguments as in~\eqref{eq:immigration_distance_approx},
	the dominated convergence theorem, and~\eqref{eq:int_g_finite} then show that
	\begin{equation}
		\begin{split}
			\MoveEqLeft \varlimsup_{K\to\infty}\varlimsup_{D\to \infty}\abs*{D\E\left[ F\left(
			\bigl(Y_{t,s}^{D,g_D}\bigr)_{t\in[s,T]}
			\right)\right] - D\E\left[ F\left( \bigl(Y_{t,s}^{D,g_D\wedge K}\bigr)_{t\in[s,T]}
			\right)\right]}\\
			&\leq mL_F e^{L_h(T-s)}\varlimsup_{K\to\infty}\int_s^T \abs{g(u,0) - g(u,0)\wedge
			K} \dd u = 0.
		\end{split}
		\label{eq:vig_LHSK}
	\end{equation}
	For all $i\in[m]$ and $x\in [0,1]$ we write $\psi_i(x) =
	\psi_i(x)^+ - \psi_i(x)^-$ to obtain a decomposition of $F$ of the form $F =
	F^+ - F^-$, where $F^+$ and $F^-$ are finite sums of nonnegative functions
	in $\MCE_{s,T}$ and satisfy $F^+(0) = 0 = F^-(0)$. Due to this
	and~\eqref{eq:vig_assK}, the first part of the
	proof yields for all $K\in\N$ that
	\begin{equation}
		\lim_{D\to\infty}D\E\left[ F^+\left( \bigl(Y^{D,g_D\wedge K}_{t,s}\bigr)_{t\in[s,T]} \right)
		\right] = \int_s^\infty \bigl(g(u,0) \wedge K\bigr)\int F^+\left( (\eta_{t-u})_{t\in[s,T]}
		\right) \di Q(d\eta)\dd u.
		\label{eq:vig_FKconv}
	\end{equation}
	The monotone convergence theorem ensures that
	\begin{equation}
		\lim_{K\to\infty}\int_s^\infty \bigl(g(u,0) \wedge K\bigr)\int F^+\left(
		(\eta_{t-u})_{t\in[s,T]} \right) \di Q(d\eta)\dd u = 
		\int_s^\infty g(u,0)\int F^+\left(
		(\eta_{t-u})_{t\in[s,T]} \right) \di Q(d\eta)\dd u.
		\label{eq:vig_RHSK}
	\end{equation}
	Moreover, Lemma~\ref{l:F_is_globally_Lipschitz}, $F(0) = 0$,
	Lemma~\ref{l:YD_L1_distance}, \eqref{eq:vig_assK},
	and~\eqref{eq:int_g_finite} yield
	for all $K\in\N$ that
	\begin{equation}
		\begin{split}
			\lim_{D\to\infty}D\E\left[ F^+\left( \bigl(Y^{D,g_D\wedge
			K}_{t,s}\bigr)_{t\in[s,T]} \right) \right] &\leq L_F m e^{L_h(T-s)}\int_s^T
			\sup_{x\in [0,1]}\abs{g(u,x)\wedge K} \dd u \\
			&\leq L_F m e^{L_h(T-s)}\int_s^T \sup_{x\in [0,1]}\abs{g(u,x)}\dd u < \infty,
		\end{split}
	\end{equation}
	which together with~\eqref{eq:vig_FKconv} yields that
	\begin{equation}
		\sup_{K\in\N}\int_s^\infty \bigl(g(u,0) \wedge K\bigr)\int F^+\left(
		(\eta_{t-u})_{t\in[s,T]} \right) \di Q(d\eta)\dd u < \infty.
	\end{equation}
	The same is true when we replace $F^+$ by $F^-$. This implies that
		$\int_s^\infty g(u,0)\int F(
		(\eta_{t-u})_{t\in[s,T]} ) \di Q(d\eta)\dd u$
	is well-defined as a real number.
	Hence, combining~\eqref{eq:vig_LHSK},~\eqref{eq:vig_FKconv},
	and~\eqref{eq:vig_RHSK} for $F^+$ and $F^-$ yields that
	\begin{equation}
		\begin{split}
			\lim_{D\to\infty}D\E\left[ F\left( \bigl(Y^{D,g_D}_{t,s}\bigr)_{t\in[s,T]} \right)
			\right] &= \lim_{K\to\infty}\lim_{D\to\infty}D\E\left[ F\left(
			\bigl(Y^{D,g_D\wedge K}_{t,s}\bigr)_{t\in[s,T]} \right) \right] \\
			& = \int_s^\infty g(u,0)\int F\left( (\eta_{t-u})_{t\in[s,T]} \right)
			\di Q(d\eta)\dd u \in \R.
		\end{split}
	\end{equation}
	This finishes the proof of Lemma~\ref{l:vanishing_immigration_general}.
\end{proof}
With Lemma~\ref{l:vanishing_immigration_general} in hand, we are ready to
prove the following Poisson limit lemma for independent diffusions with vanishing
immigration, which generalizes Lemma~4.21 in~\textcite{Hutzenthaler2012} to
state-dependent~$g$.
\begin{lemma}[Poisson limit for independent diffusions with vanishing
	immigration]
	\label{l:vanishing_immigration_weak_process}
	Assume that \textup{Setting~\ref{set:excursion_measure}} and
	\textup{Setting~\ref{set:g_D}} hold,
	let $s\in[0,\infty)$, 
	for every $D\in\N$ let $(Y^{D,g_D}_{t,s}(i))_{t\in[s,\infty)}$,
	$i\in[D]$, be independent solutions of~\eqref{eq:SDE_YD}
	satisfying for all $i\in[D]$ a.s.~that $Y^{D,g_D}_{s,s}(i) = 0$,
	let $g\colon [0,\infty)\times [0,1] \to \R$ be a measurable function satisfying
	for all $t\in[0,\infty)$ that $g(t,0) \geq 0$, that $[0,1] \ni x \mapsto
	g(t,x)$ is continuous, and that
	\begin{equation}
		\lim_{D\to\infty}\int_s^{t \vee s} \sup_{x\in [0,1]}\abs{g_D(u,x)-g(u,x)}\dd
		u = 0,
		\label{eq:convergence_gD}
	\end{equation}
	and let $\Pi$ be a Poisson point
	process on $[s,\infty) \times U$ with intensity measure $\Exp{\Pi(du \otimes
	d\eta)} = g(u,0) \dd u \otimes Q(d\eta)$.
	Then it holds that
	\begin{equation}
		\left(\sum_{i=1}^D Y_{t,s}^{D,g_D}(i) \delta_{Y_{t,s}^{D,g_D}(i)}
		\right)_{t\in[s,\infty)} \xRightarrow{D\to\infty} \left(\int
		\eta_{t-u}\delta_{\eta_{t-u}} \di \Pi(du \otimes
		d\eta)\right)_{t\in[s,\infty)}
		\label{eq:vanishing_immigration_weak_process}
	\end{equation}
	in the sense of convergence in distribution on
	$D([s,\infty),\MCM_\textup{f}([0,1]))$.
\end{lemma}
\begin{proof}
	Fix $\varphi\in C^2([0,1],\R)$ for the rest of this paragraph. We define the
	function $\phi \colon [0,1] \to \R$ by $[0,1] \ni x \mapsto \phi(x) \defeq
	x\varphi(x)$ and for all $D\in\N$
	and all $t\in[s,\infty)$ we define $S^D_{t,s} \defeq \sum_{i=1}^D
	\phi(Y_{t,s}^{D,g_D}(i))$.
	The fact that there exists a constant $L_\phi \in
	[0,\infty)$ such that for
	all $x\in[0,1]$ it holds that $\abs{\phi(x)} \leq L_\phi x$ and Markov's
	inequality yield
	for all $D,K\in\N$ and all $t\in[s,\infty)$ that
	\begin{equation}
		\Prop*{\abs{S^D_{t,s}} \geq K}
		\leq \frac{1}{K}\Exp*{\sum_{i=1}^D
		\abs[\big]{\phi\bigl(Y_{t,s}^{D,g_D}(i)\bigr)}}
		\leq \frac{L_\phi}{K}\Exp*{\sum_{i=1}^D
		Y_{t,s}^{D,g_D}(i)}
		= \frac{L_\phi}{K} D
		\Exp*{Y_{t,s}^{D,g_D}(1)}.
		\label{eq:vig_marginals_estimate}
	\end{equation}
	This, Lemma~\ref{l:YD_L1_distance}, and Setting~\ref{set:g_D} imply for all
	$t\in[s,\infty)$ that
	\begin{equation}
		\lim_{K\to\infty}\sup_{D\in\N}\Prop*{\abs{S^D_{t,s}} \geq K} \leq
		\lim_{K\to\infty}\frac{L_\phi}{K}e^{L_h(t-s)}\sup_{D\in\N} \int_s^t
		\sup_{x\in [0,1]}\abs{g_D(u,x)}\dd u =
		0.
		\label{eq:vig_marginals}
	\end{equation}
	For every $T\in(s,\infty)$ and every $D\in\N$ let $\MCS^D_T$ be the set of all stopping
	times with respect to the natural filtration of $S^D_{\cdot, s}$ that are
	bounded by $T$. For all $y \in [0,1]$,
	all $D\in \N$, and
	all $t\in[s,\infty)$ we define
	\begin{equation}
		(G^D_t \phi)(y) \defeq \left( \tfrac{1}{D}g_D(t,y) + \hht_D(y)
		\right)\phi'(y) + \tfrac{1}{2}\sigma^2(y)\phi''(y).
	\end{equation}
	Setting~\ref{set:coefficients} and
	$\phi\in C^2([0,1],\R)$ imply that there
	exists a constant $C_\phi \in
	[1,\infty)$ such that it holds for all $y\in [0,1]$, all $D\in\N$, and all $t\in[s,\infty)$
	that $\abs{(G^D_t\phi)(y)} \leq
	C_\phi(\frac{1}{D}\sup_{x\in [0,1]}\abs{g_D(t,x)}+y)$ and $\sigma^2(y)\phi'^2(y) \leq
	C_\phi^2 y$,
	Jensen's inequality ensures for all $x_1,x_2,x_3 \in
	\R$ that $(\sum_{i=1}^3 x_i)^2 \leq 3\sum_{i=1}^3 x_i^2$,
	Hölder's inequality yields for every $\delta\in[0,\infty)$ and every integrable
	function $\alpha\colon
	[0,\delta] \to \R$ that $(\int_0^\delta \alpha(u)\dd u)^2
	\leq \delta \int_0^\delta (\alpha(u))^2 \dd u$, and it holds for all $x\in\R$
	that $x \leq 1 +
	x^2$. It\^{o}'s formula, the It\^{o} isometry, and the preceding estimates
	show for all $T \in (s,\infty)$, all $\deltab \in [0,1]$, all $D\in\N$, all
	$\tau\in\MCS^D_T$, and all $\delta\in[0,\deltab]$ that
	\begin{equation}
		\begin{split}
			\Exp*{\left(S^D_{\tau+\delta,s} - S^D_{\tau,s}\right)^2} &= \Exp*{\left(
			\mathop{\textstyle\sum}\limits_{i=1}^D\smallint_{\tau}^{\tau+\delta} (G^D_u
			\phi)\bigl(Y^{D,g_D}_{u,s}(i)\bigr)\dd u + \mathop{\textstyle\sum}\limits_{i=1}^D
			\smallint_{\tau}^{\tau+\delta}(\sqrt{\sigma^2}\phi')\bigl(Y^{D,g_D}_{u,s}(i)\bigr)\dd
			W_u(i)
			\right)^2}\\
			&\leq 3C_\phi^2\Exp*{\biggl( \smallint_{\tau}^{\tau+\delta} \sup_{x\in
			[0,1]}\abs{g_D(u,x)} \dd u
			\biggr)^2} + 3C_\phi^2\Exp*{\left(
				\mathop{\textstyle\sum}\limits_{i=1}^D \smallint_{0}^\delta
			Y^{D,g_D}_{\tau+u,s}(i)\dd u\right)^2}\\
			&\quad+3\Exp*{\mathop{\textstyle\sum}\limits_{i=1}^D\smallint_0^\delta
			\sigma^2\bigl(Y^{D,g_D}_{\tau+u,s}(i)\bigr)\phi'^2\bigl(Y^{D,g_D}_{\tau+u,s}(i)\bigr)\dd u}\\
			&\leq3C_\phi^2\delta\Exp*{\smallint_{\tau}^{\tau+\delta} \sup_{x\in
			[0,1]}\abs{g_D(u,x)}^2 \dd
			u}
			+3C_\phi^2\delta\Exp*{\smallint_{0}^\delta \left(
			\mathop{\textstyle\sum}\limits_{i=1}^D
			Y^{D,g_D}_{\tau+u,s}(i) \right)^2 \dd u}\\
			&\quad +3C_\phi^2\Exp*{\smallint_0^\delta \mathop{\textstyle\sum}\limits_{i=1}^D
			Y^{D,g_D}_{\tau+u,s}(i) \dd u}\\
			&\leq3C_\phi^2\deltab\smallint_{s}^{T+1} \sup_{x\in
			[0,1]}\abs{g_D(u,x)}^2 \dd
			u + 6C_\phi^2\deltab\Exp*{\sup_{u\in[s,T+1]} \biggl(
			\mathop{\textstyle\sum}\limits_{i=1}^D
			Y^{D,g_D}_{u,s}(i) \biggr)^2}
			+3C_\phi^2\deltab.
		\end{split}
	\end{equation}
	This, Lemma~\ref{l:YD_second_moment}, and Setting~\ref{set:g_D} imply for
	all $T\in(s,\infty)$ that
	\begin{equation}
			\lim_{\deltab\to 0}\sup_{D\in\N}\sup_{\tau \in
			\MCS^D_T}\sup_{\delta\in[0,\deltab]}\Exp*{\left(S^D_{\tau+\delta,s} -
			S^D_{\tau,s}\right)^2} = 0.
		\label{eq:vig_aldous}
	\end{equation}
	By Aldous' tightness criterion (see, e.g., Theorem~3.8.6
	in~\textcite{EthierKurtz1986}),
	\eqref{eq:vig_marginals} and~\eqref{eq:vig_aldous} ensure that
	\begin{equation}
		\left\{ \left( \sum_{i=1}^D Y_{t,s}^{D,g_D}(i)\varphi\bigl(Y_{t,s}^{D,g_D}(i)\bigr)
		\right)_{t\in[s,\infty)} \colon D\in\N\right\}
		\label{eq:vig_roelly}
	\end{equation}
	is relatively compact.
	Since $\varphi\in
	C^2([0,1],\R)$ was
	arbitrary, it follows from~\eqref{eq:vig_roelly},
	from Theorem~2.1 in~\textcite{RoellyCoppoletta1986}, and from
	Prohorov's theorem (e.g.\ Theorem~3.2.2 in~\textcite{EthierKurtz1986}) that
	\begin{equation}
		\left\{ \left( \sum_{i=1}^D Y_{t,s}^{D,g_D}(i) \delta_{Y_{t,s}^{D,g_D}(i)}
		\right)_{t\in[s,\infty)} \colon D\in\N \right\}
		\label{eq:vig_relcomptoshow}
	\end{equation}
	is relatively compact.

	In the following, we identify the limit points
	of~\eqref{eq:vig_relcomptoshow} by showing convergence of
	finite-dimensional distributions. For that, fix $n\in\N$, fix
	$\varphi_1,\dotsc,\varphi_n\in \Lip([0,1],[0,\infty))$, and fix
	$t_1,\dotsc,t_n\in[s,\infty)$ with $t_1 \leq \dotsb \leq t_n$. For every
	$j\in[n]$ we define the function $\phi_j\colon [0,1]\to[0,\infty)$ by
	$[0,1]\ni x \mapsto \phi_j(x) \defeq x\varphi_j(x)$.
	The fact that $Y^{D,g_D}_{\cdot, s}(i)$,
	$i\in[D]$, are i.i.d.~for all $D\in\N$ yields
	for all $D\in\N$ that
	\begin{equation}
		\begin{split}
			\Exp[\bigg]{\exp\biggl( -\sum_{j=1}^n\sum_{i=1}^D
			\phi_j\bigl(Y^{D,g_D}_{t_j,s}(i)\bigr) \biggr)}
			&= \prod_{i=1}^D \Exp[\bigg]{\exp\biggl( -\sum_{j=1}^n
			\phi_j\bigl(Y^{D,g_D}_{t_j,s}(i)\bigr) \biggr)}\\
			&= \left( 1-\frac{D\Exp*{1-\exp\left( -\sum_{j=1}^n
			\phi_j\bigl(Y^{D,g_D}_{t_j,s}(1)\bigr) \right)}}{D}
			\right)^D.
		\end{split}
		\label{eq:expS}
	\end{equation}
	For all $x_1,\dotsc,x_n\in[0,1]$ it holds that
	\begin{equation}
		1-\exp\biggl( -\sum_{j=1}^n \phi_j(x_j) \biggr) =
		\sum_{j=1}^n\bigl( 1-\exp( -\phi_j(x_j) )
		\bigr)\exp\biggl( -\sum_{i=1}^{j-1}\phi_i(x_i) \biggr).
		\label{eq:vig_telescopic}
	\end{equation}
	This shows that~\eqref{eq:expS}
	involves the expectation of a sum. Each summand has the form of a
	functional $F\in \MCE_{s,t_n}$ with $F(0) = 0$. On compact
	subintervals of $[0,\infty)$, the sequence of functions
	$x \mapsto (1-\frac{x}{D})^D$, $D\in\N$, converges uniformly to the function
	$x \mapsto e^{-x}$ as $D\to\infty$.
	This and Lemma~\ref{l:vanishing_immigration_general} applied to each summand of the
	sum obtained from~\eqref{eq:expS} and~\eqref{eq:vig_telescopic} show that
	\begin{equation*}
		\begin{split}
			\lim_{D\to\infty}\Exp[\bigg]{\exp\biggl( -\sum_{j=1}^n\sum_{i=1}^D
			\phi_j\bigl(Y^{D,g_D}_{t_j,s}(i)\bigr) \biggr)}
			&= \exp\biggl(
			-\lim_{D\to\infty}D\Exp[\bigg]{1-\exp\biggl(
			-\sum_{j=1}^n\phi_j\bigl(Y^{D,g_D}_{t_j,s}(1)\bigr)
			\biggr)} \biggr)\\
			&=\exp\biggl( -\int\int_{s}^\infty \biggl( 1-\exp\biggl(
			-\sum_{j=1}^n \phi_j(\eta_{t_j-u})
			\biggr) \biggr) g(u,0) \dd u\di Q(d\eta) \biggr)\\
			&= \Exp[\bigg]{\exp\biggl( -\sum_{j=1}^n \int \phi_j(\eta_{t_j-u})
			\di \Pi(du \otimes d\eta) \biggr)}.
		\end{split}
	\end{equation*}
	This implies the convergence of finite-dimensional distributions
	of~\eqref{eq:vig_relcomptoshow} and completes the proof of
	Lemma~\ref{l:vanishing_immigration_weak_process}.
\end{proof}

\subsection{Convergence of the loop-free processes}
\label{ss:convergence_loop_free}
In this subsection, we show convergence of the loop-free processes using
Lemma~\ref{l:vanishing_immigration_weak_process}. For
that, we make the following assumption, which implies that the initial
population has migration level zero.
\begin{setting}[Initial distribution]\label{set:initial_dist_level0}
	Assume that
	Setting~\ref{set:initial_moment}
	and Setting~\ref{set:migration_levels} hold, that
	\begin{equation}
		\Exp*{\left(\sum_{i=1}^\infty X_0(i)\right)^2}<\infty,
		\label{eq:initial_second_moment}
	\end{equation}
	and that it holds for all $D\in\N$ and all $i\in[D]$ that $\MCL(X_0^{D,0}(i)) =
	\MCL(X_0(i))$ and for all $D\in\N$ and all $(i,k)\in[D]\times \N$ that
	$\MCL(X_0^{D,k}(i)) = \delta_0$.
\end{setting}
The following lemma establishes the convergence of the loop-free processes and
is analogous to Lemma~4.22 in~\textcite{Hutzenthaler2012}. Recall that by
Setting~\ref{set:migration_levels}, which is satisfied by assumption in
the following lemma, the loop-free processes
%$\{(Z_t^{D,k}(i))_{t\in[0,\infty)}\colon (i,k) \in [D] \times \N_0\}$ 
fulfill for all $D\in\N$ and all $(i,k) \in [D]\times\N_0$ that a.s.~$Z_0^{D,k}(i) =
X_0^{D,k}(i)$.
\begin{lemma}[Convergence of the loop-free processes]
	\label{l:convergence_of_the_loop_free_process}
	Assume that \textup{Setting~\ref{set:initial_dist_level0}} holds and let
	$\MCT$ be the forest of excursions constructed in
	\textup{Subsection~\ref{ss:tree_excursions}}.
	Then it holds that
	\begin{equation}
		\left(\sum_{i=1}^D\sum_{k\in\N_0} Z_t^{D,k}(i) \delta_{Z_t^{D,k}(i)}
		\right)_{t\in[0,\infty)} \xRightarrow{D\to\infty} \left(\int
		\eta_{t-s}\delta_{\eta_{t-s}}
		\di\MCT(ds \otimes d\eta)\right)_{t\in[0,\infty)}
		\label{eq:convergence_of_the_loop_free_process}
	\end{equation}
	in the sense of convergence in distribution on
	$D([0,\infty),\MCM_\textup{f}([0,1]))$.
\end{lemma}
\begin{proof}
	Fix $\varphi\in C^2([0,1],\R)$ for the rest of this
	paragraph. We define the function $\phi\colon[0,1]\to\R$ by $[0,1] \ni x
	\mapsto \phi(x) \defeq x\varphi(x)$ and for all $D\in\N$ and all
	$t\in[0,\infty)$ we define $S^D_{t} \defeq
	\sum_{i=1}^D\sum_{k\in\N_0} \phi(Z_{t}^{D,k}(i))$.
	The fact that there exists a constant $L_\phi \in
	[0,\infty)$ such that for all $x\in[0,1]$ it holds that $\abs{\phi(x)} \leq L_\phi
	x$ and Markov's inequality yield for all $D,K\in\N$ and all $t\in[0,\infty)$
	that
	\begin{equation}
		\Prop*{\abs{S^D_{t}} \geq K} \leq
		\frac{1}{K}\Exp*{\sum_{i=1}^D\sum_{k\in\N_0}
		\abs[\big]{\phi\bigl(Z_t^{D,k}(i)\bigr)}} \leq
		\frac{L_\phi}{K}\Exp*{\sum_{i=1}^D\sum_{k\in\N_0}
		Z_{t}^{D,k}(i)}.
		\label{eq:clf_marginals_estimate}
	\end{equation}
	This, Lemma~\ref{l:X_Z_first_moment}, and
	Setting~\ref{set:initial_dist_level0} imply for all $t\in[0,\infty)$ that
	\begin{equation}
		\lim_{K\to\infty} \sup_{D\in\N} \Prop*{\abs{S^D_{t}} \geq K} \leq
		\lim_{K\to\infty} \sup_{D\in\N}\frac{L_\phi}{K}\Exp*{\sum_{i=1}^D\sum_{k\in\N_0}
		Z_{t}^{D,k}(i)} = 0.
		\label{eq:clf_marginals}
	\end{equation}
	For every $T\in(0,\infty)$ and every $D\in\N$ let $\MCS^D_T$ be the set of stopping times
	with respect to the natural filtration of $S^D$ that are bounded by $T$.
	For all $D\in\N$ and all $x = (x_{i,k})_{(i,k)\in[D]\times\N_0}\in
	[0,1]^{[D]\times\N_0}$ we define $\psi^D(x) \defeq
	\sum_{i=1}^D\sum_{k\in\N_0} \phi(x_{i,k})$ and
	\begin{equation}
		\begin{split}
			(G^D \psi^D)(x) \defeq{} & \sum_{i=1}^D\sum_{k\in\N_0}\left(
			\frac{\1_{k > 0}}{D}\sum_{j=1}^D x_{j,\abs{k-1}}f(x_{j,\abs{k-1}},x_{i,k}) +
			\hht_D(x_{i,k}) + \1_{k=0}
			h_D(0) 
			\right)\phi'(x_{i,k})\\
			&+\frac{1}{2}\sum_{i=1}^D\sum_{k\in\N_0}\sigma^2(x_{i,k})\phi''(x_{i,k}).
		\end{split}
	\end{equation}
	Setting~\ref{set:coefficients} and $\phi\in C^2([0,1],\R)$
	imply that there exists a constant $C_\psi\in[1,\infty)$ such
	that it holds for all $x = (x_{i,k})_{(i,k)\in[D]\times\N_0} \in
	[0,1]^{[D]\times \N_0}$, all $y \in [0,1]$, and all $D\in\N$ that
	$\abs{(G^D\psi^D)(x)} \leq C_\psi(2\mu + \sum_{i=1}^D\sum_{k\in\N_0}
	x_{i,k})$ and $\sigma^2(y)\phi'^2(y) \leq
	C_\psi^2 y$,
	Jensen's inequality ensures for all $x_1,x_2 \in
	\R$ that $(x_1 + x_2)^2 \leq 2 (x_1^2 + x_2^2)$, Hölder's inequality yields
	for every $\delta\in[0,\infty)$ and every integrable function $\alpha\colon
	[0,\delta] \to \R$ that $(\int_0^\delta \alpha(u)\dd u)^2
	\leq \delta \int_0^\delta (\alpha(u))^2 \dd u$, and it holds for
	all $x\in\R$ that $2x \leq 1 + x^2$.
	It\^{o}'s formula, the It\^{o} isometry, and the preceding estimates show
	for all $T\in(0,\infty)$, all $\deltab
	\in [0,1]$, all $D\in\N$, all $\tau\in\MCS^D_T$, and all $\delta \in
	[0,\deltab]$ that
	\begin{equation*}
		\begin{split}
			\Exp*{\left(S^D_{\tau+\delta} - S^D_{\tau}\right)^2} &= \Exp*{\biggl(
			\smallint_{\tau}^{\tau+\delta} (G^D
			\psi)\bigl(Z^{D,\cdot}_{u}(\cdot)\bigr)\dd u +
			\mathop{\textstyle\sum}\limits_{i=1}^D\mathop{\textstyle\sum}\limits_{k\in\N_0}
			\smallint_{\tau}^{\tau+\delta}(\sqrt{\sigma^2}\phi')\bigl(Z^{D,k}_{u}(i)\bigr)\dd
			W^k_u(i)
			\biggr)^2}\\
			&\leq 2C_\psi^2\Exp*{\biggl( \smallint_{0}^\delta
			2\mu +
	\mathop{\textstyle\sum}\limits_{i=1}^D\mathop{\textstyle\sum}\limits_{k\in\N_0}
	Z^{D,k}_{\tau+u}(i)\dd u\biggr)^2}
	+2\Exp*{\mathop{\textstyle\sum}\limits_{i=1}^D\mathop{\textstyle\sum}\limits_{k\in\N_0}\smallint_0^\delta
	(\sqrt{\sigma^2}\phi')^2\bigl(Z^{D,k}_{\tau+u}(i)\bigr)\dd u}\\
			&\leq 2C_\psi^2\delta\Exp*{ \smallint_{0}^\delta
			\biggl(2\mu +
		\mathop{\textstyle\sum}\limits_{i=1}^D\mathop{\textstyle\sum}\limits_{k\in\N_0}
Z^{D,k}_{\tau+u}(i)\biggr)^2\dd u}
		+2C_\psi^2\Exp*{\smallint_0^\delta\mathop{\textstyle\sum}\limits_{i=1}^D\mathop{\textstyle\sum}\limits_{k\in\N_0}
			Z^{D,k}_{\tau+u}(i)\dd u}\\
			&\leq 3C_\psi^2\deltab\Exp*{\sup_{t\in[0,T+1]} \biggl( 2\mu +
					\mathop{\textstyle\sum}\limits_{i=1}^D\mathop{\textstyle\sum}\limits_{k\in\N_0}
		Z^{D,k}_{t}(i) \biggr)^2} + C_\psi^2\deltab.
		\end{split}
	\end{equation*}
	This, Lemma~\ref{l:X_Z_second_moment}, and
	Setting~\ref{set:initial_dist_level0} imply for all
	$T\in(0,\infty)$ that
	\begin{equation}
		\lim_{\deltab\to
		0}\sup_{D\in\N}\sup_{\tau\in\MCS^D_T}\sup_{\delta\in[0,\deltab]}\Exp*{\left(S^D_{\tau+\delta}
		- S^D_{\tau}\right)^2} = 0.
		\label{eq:clf_aldous}
	\end{equation}
	By Aldous' tightness criterion (see, e.g., Theorem~3.8.6
	in~\textcite{EthierKurtz1986}),
	\eqref{eq:clf_marginals} and~\eqref{eq:clf_aldous} ensure that
	\begin{equation}
		\left\{ \left( \sum_{i=1}^D\sum_{k\in\N_0} Z_{t}^{D,k}(i)\varphi\bigl(Z_{t}^{D,k}(i)\bigr)
		\right)_{t\in[0,\infty)} \colon D\in\N\right\}
		\label{eq:clf_roelly}
	\end{equation}
	is relatively compact.
	Since $\varphi\in
	C^2([0,1],\R)$ was
	arbitrary, it follows
	from~\eqref{eq:clf_roelly}, from Theorem~2.1
	in~\textcite{RoellyCoppoletta1986} and from Prohorov's theorem
	(e.g.\ Theorem~3.2.2 in~\textcite{EthierKurtz1986}) that
	\begin{equation}
		\left\{ \left( \sum_{i=1}^D\sum_{k\in\N_0} Z_{t}^{D,k}(i)\delta_{Z_{t}^{D,k}(i)}
		\right)_{t\in[0,\infty)} \colon D\in\N \right\}
		\label{eq:clf_relcomptoshow}
	\end{equation}
	is relatively compact.

	In the following, we identify the limit points
	of~\eqref{eq:clf_relcomptoshow} by showing convergence of finite-dimensional
	distributions.
	For that, fix $n\in\N$, fix
	$\varphi_1,\dotsc,\varphi_n\in \Lip([0,1],[0,\infty))$, and fix
	$t_1,\dotsc,t_n\in[0,\infty)$ with $t_1 \leq \dotsc \leq t_n$. For every
	$j\in[n]$ we define the function $\phi_j\colon[0,1]\to [0,\infty)$ by
	$[0,1]\ni x \mapsto \phi_j(x) \defeq x\varphi_j(x)$.
	Next we show that it holds for all $m\in\N_0$ that
	\begin{equation}
		\lim_{D\to\infty}\Exp[\bigg]{\exp\biggl( -\sum_{j=1}^n
		\sum_{i=1}^D\sum_{k=0}^m \phi_j\bigl(Z^{D,k}_{t_j}(i)\bigr) \biggr)} =
		\Exp[\bigg]{\exp\biggl( -\sum_{j=1}^n \sum_{k=0}^m \int
		\phi_j(\eta_{t_j-s})\di \MCT^{(k)}(ds \otimes d\eta) \biggr)}.
		\label{eq:clf_fdds}
	\end{equation}
	We prove~\eqref{eq:clf_fdds} by induction on $m\in\N_0$. For all $\eta \in
	C([0,\infty),[0,1])$ we define $F(\eta) \defeq \sum_{j=1}^n
	\phi_j(\eta_{t_j})$. 
	For every $D\in\N$ let
	$(Y^{D,Dh_D(0)}_{t,0}(i))_{t\in[0,\infty)}$, $i\in[D]$, be
	independent solutions of~\eqref{eq:SDE_YD} coupled only through their
	initial states satisfying for all
	$i\in[D]$ a.s.~that
	$Y^{D,Dh_D(0)}_{0,0}(i) = X_0(i)$ and
	let $(\Yb^{D,Dh_D(0)}_{t,0}(i))_{t\in[0,\infty)}$, $i\in[D]$, be
	independent solutions
	of~\eqref{eq:SDE_YD} satisfying for all $i\in[D]$ a.s.~that
	$\Yb^{D,Dh_D(0)}_{0,0}(i) = 0$.
	%(for existence and uniqueness of strong solutions of~\eqref{eq:SDE_YD} see,
	%e.g., Theorem~5.4.22, Proposition~5.2.13, and Corollary~5.3.23 in
	%\textcite{KaratzasShreve1991}).
	Then $Y^{D,Dh_D(0)}_{\cdot,0}$ is equal in distribution to $Z^{D,0}$. 
	Note that it holds for all $x,y,z \in [0,\infty)$ that $\abs{e^{-x} -
	e^{-(y+z)}} \leq \abs{e^{-x}(1-e^{-y})} + \abs*{e^{-y}(e^{-x} - e^{-z})}
	\leq 1-e^{-y} + \abs{x-z}$.  Moreover,
	there exists a constant $L_F\in[0,\infty)$ such that it holds for all
	$\eta,\etab \in C([0,\infty),[0,1])$ that $\abs{F(\eta) - F(\etab)} \leq
	L_F \sum_{j=1}^n \abs{\eta_{t_j} - \etab_{t_j}}$. These facts and
	Lemma~\ref{l:YD_L1_distance} imply that
	\begin{equation}
		\begin{split}
			\MoveEqLeft \lim_{K\to\infty}\lim_{D\to\infty}\abs[\bigg]{\E\biggl[
					\exp\biggl(
				-\sum_{i=K+1}^D F\bigl(Y_{\cdot,0}^{D,D
				h_D(0)}(i)\bigr)
			\biggr)\biggr] - \E\biggl[ \exp\biggl( -\sum_{i=1}^D
				F\bigl(\Yb_{\cdot,0}^{D,D h_D(0)}(i)\bigr)\biggr)\biggr]}\\
			&\leq
			\lim_{K\to\infty}\lim_{D\to\infty} \E\biggl[ 1-\exp\biggl( -\sum_{i=1}^K
				F\bigl(\Yb_{\cdot,0}^{D,D h_D(0)}(i)
			\bigr)\biggr)\biggr] + \lim_{K\to\infty}L_F ne^{L_h t_n}\sum_{i=K+1}^\infty \E[X_0(i)].
		\end{split}
		\label{eq:clf_initial}
	\end{equation}
	The second summand on the right-hand side of~\eqref{eq:clf_initial} is zero due to
	Setting~\ref{set:initial_dist_level0}. For every $i\in\N$
	the process $\Yb^{D,Dh_D(0)}_{\cdot, 0}(i)$ converges weakly to zero as
	$D\to\infty$, so the first summand on the right-hand side
	of~\eqref{eq:clf_initial} is also zero. On the other hand, for
	every $i\in\N$ the process $Y^{D, D
	h_D(0)}_{\cdot,0}(i)$ converges weakly to $Y(i)$ as
	$D\to\infty$ (see, e.g., Theorem~4.8.10 in~\textcite{EthierKurtz1986}). These observations and
	Lemma~\ref{l:vanishing_immigration_weak_process} with $s=0$,
	$(g_D)_{D\in\N} = (D
	h_D(0))_{D\in\N}$, and $g = \mu$ imply that
	\begin{equation}
		\begin{split}
			\MoveEqLeft\lim_{D\to\infty}\E\biggl[ \exp\biggl(
				-\sum_{i=1}^D
				F\bigl(Z^{D,0}(i)\bigr)\biggr) \biggr]\\
			&= \lim_{K\to\infty}\lim_{D\to\infty}\E\biggl[ \exp\biggl(
				-\sum_{i=1}^K F\bigl(Y_{\cdot,0}^{D,D
				h_D(0)}(i)\bigr)
			\biggr)\biggr]\E\biggl[ \exp\biggl( -\sum_{i=K+1}^D
				F\bigl(Y_{\cdot,0}^{D,D h_D(0)}(i)\bigr)
			\biggr)\biggr]\\
			&= \lim_{K\to\infty}\E\biggl[ \exp\biggl(
				-\sum_{i=1}^K F\bigl(Y(i)\bigr)
			\biggr)\biggr]\lim_{D\to\infty}\E\biggl[ \exp\biggl( -\sum_{i=1}^D
				F\bigl(\Yb_{\cdot,0}^{D,D h_D(0)}(i)\bigr)
			\biggr)\biggr]\\
			&= \E\biggl[ \exp\biggl(
				-\sum_{i=1}^\infty F\bigl(Y(i)\bigr)
			\biggr)\biggr]\E\biggl[ \exp\biggl( -\int
				F(\eta_{\cdot-s})
			 \di \Pi^\emptyset(ds \otimes d\eta)\biggr)\biggr]\\
			&= \E\biggl[ \exp\biggl( -\int
				F(\eta_{\cdot-s})
			\di \MCT^{(0)}(ds \otimes d\eta)\biggr)\biggr].
		\end{split}
	\end{equation}
	This establishes~\eqref{eq:clf_fdds} in the base case $m=0$.
	For the induction step $\N_0 \ni m \to m+1$ the induction hypothesis and
	relative compactness for all $\mt \in [m]_0$ of $\{(\sum_{i=1}^D\sum_{k=0}^\mt
	Z_t^{D,k}(i)\delta_{Z_t^{D,k}(i)})_{t\in[0,\infty)} \colon D\in\N\}$
	imply for all $\mt \in [m]_0$ and all $\varphi \in C([0,1],\R)$
	that
	\begin{equation}
		\left( \sum_{i=1}^D\sum_{k=0}^\mt Z_t^{D,k}(i)\varphi\bigl(Z_t^{D,k}(i)\bigr)
		\right)_{t\in [0,\infty)} \xRightarrow{D\to\infty} \left( \sum_{k=0}^\mt\int
		\eta_{t-s}\varphi(\eta_{t-s})\di \MCT^{(k)}(ds \otimes d\eta) \right)_{t\in
		[0,\infty)}.
		\label{eq:clf_weakconvergence}
	\end{equation}
	By the Skorokhod representation theorem (e.g.\ Theorem~3.1.8 in~\textcite{EthierKurtz1986}), we may
	assume almost sure convergence in~\eqref{eq:clf_weakconvergence}.
	Consequently, we may assume for all $\varphi \in C([0,1],\R)$ that
	\begin{equation}
		\left( \sum_{i=1}^D Z_t^{D,m}(i)\varphi\bigl(Z_t^{D,m}(i)\bigr)
		\right)_{t\in [0,\infty)} \xrightarrow[\textup{a.s.}]{D\to\infty} \left( \int
		\eta_{t-s}\varphi(\eta_{t-s})\di \MCT^{(m)}(ds \otimes d\eta) \right)_{t\in
		[0,\infty)}.
		\label{eq:clf_asconvergence}
	\end{equation}
	For every $D\in\N$ we define $g_D$ by
	$[0,\infty) \times [0,1] \ni (t,x) \mapsto g_D(t,x)\defeq \sum_{j=1}^D
	Z_t^{D,m}(j)f(Z_t^{D,m}(j),x)$ and we define $g$ by $[0,\infty) \times [0,1]
	\ni (t,x) \mapsto g(t,x) \defeq \int \eta_{t-s}f(\eta_{t-s},x)\di
	\MCT^{(m)}(ds \otimes d\eta)$.
	The sequence of functions $(g_D)_{D\in\N}$
	satisfies Setting~\ref{set:g_D} almost surely. Moreover, the function $g$
	satisfies almost surely for all $t\in[0,\infty)$ that $g(t,0) \geq 0$ and that $[0,1]\ni x
	\mapsto g(t,x)$ is continuous.
	Equation~\eqref{eq:clf_asconvergence} and the assumptions on $f$ imply almost surely for all
	$t\in[0,\infty)$ that $(g_D(t,\cdot))_{D\in\N}$ is an equicontinuous sequence and
	this together with~\eqref{eq:clf_asconvergence} yields almost surely for all
	$t\in[0,\infty)$ that
	\begin{equation}
		\lim_{D\to\infty} \sup_{x\in [0,1]}\abs{g_D(t,x) - g(t,x)} = 0.
		\label{eq:clf_gDtog}
	\end{equation}
	It follows almost surely for all $t\in[0,\infty)$
	from~\eqref{eq:clf_asconvergence} that
	$[0,t] \ni u \mapsto \int \eta_{u-s}\di \MCT^{(m)}(ds \otimes d\eta)$
	is \cadlag\ and therefore square-integrable and thus that $\sup_{D\in\N}
	\int_0^t ( \sum_{i=1}^D Z^{D,m}_u(i) )^2 \dd u < \infty$.
	This, the fact that for all $D\in\N$ and all $u\in[0,\infty)$ it holds that
	$\sup_{x\in[0,1]}\abs{g_D(u,x) - g(u,x)} \leq L_f( \sum_{i=1}^D
	Z^{D,m}_u(i) + \int \eta_{u-s}\di \MCT^{(m)}(ds\otimes d\eta))$, and
	Theorem~6.18 and Corollary~6.21 in~\textcite{Klenke2014} imply almost
	surely for all
	$t\in[0,\infty)$ that the family
	\begin{equation}
		\left\{[0,t] \ni u \mapsto \sup_{x\in[0,1]} \abs{g_D(u,x) - g(u,x)} \colon
		D\in\N \right\}
	\end{equation}
	is uniformly integrable. This, Theorem~6.25 in~\textcite{Klenke2014},
	and~\eqref{eq:clf_gDtog} show almost surely for all $t\in[0,\infty)$
	that
	\begin{equation}
		\lim_{D\to\infty} \int_0^t \sup_{x\in [0,1]}\abs{g_D(u,x) - g(u,x)} \dd
		u = 0.
	\end{equation}
	Conditionally on $(Z^{M,m})_{M\in\N}$, for every $D\in\N$ a version of $Z^{D,m+1}$
	is given by $Y^{D,g_D}_{\cdot,0}$
	satisfying for all $i\in[D]$ and all
	$t\in[0,\infty)$ that a.s.
	\begin{equation}
		Y_{t,0}^{D,g_D}(i) =
		\int_0^t\tfrac{1}{D}g_D\bigl(u,Y^{D,g_D}_{u,0}(i)\bigr)
		+\hht_D\bigl(Y_{u,0}^{D,g_D}(i)\bigr)\dd u
		+\int_0^t\sqrt{\sigma^2\bigl(Y_{u,0}^{D,g_D}(i)\bigr)}\dd
		W^{m+1}_u(i).
	\end{equation}
	Therefore, Lemma~\ref{l:vanishing_immigration_weak_process} yields that a.s.
	\begin{equation}
		\begin{split}
			\MoveEqLeft\lim_{D\to\infty}\Exp[\bigg]{ \exp\biggl(
				-\sum_{i=1}^D
				F\bigl(Z^{D,m+1}(i)\bigr)\biggr) \given
				(Z^{M,m})_{M\in\N},\MCT^{(m)}}
				\\
				&=\lim_{D\to\infty}\Exp[\bigg]{ \exp\biggl(
				-\sum_{i=1}^D
				F\bigl(Y_{\cdot,0}^{D,g_D}(i)\bigr)\biggr) \given
				(Z^{M,m})_{M\in\N},\MCT^{(m)}}\\
				&= \Exp[\bigg]{ \exp\biggl(
				-\int
				F(\eta_{\cdot-s})\di \Pi^{(m+1)}(ds \otimes d\eta)\biggr) \given
				\MCT^{(m)}},
		\end{split}
		\label{eq:clf_induction_step}
	\end{equation}
	where $\Pi^{(m+1)}$ conditioned on $\MCT^{(m)}$ is a Poisson
	point process on $[0,\infty)\times U$ with the property that for all bounded
	measurable $\Phi\colon
	[0,\infty)\times U \to [0,\infty)$ it holds almost surely that
	\begin{equation}
		\begin{split}
			\int \Phi(s,\eta)\di \Exp*{ \Pi^{(m+1)}(ds\otimes d\eta) \given \MCT^{(m)}}
			& =\int \Phi(s,\eta)\left(\int
			\at(\chi_{s-r})\di \MCT^{(m)}(dr \otimes d\chi)\right) \dd
			s\otimes Q(d\eta)\\
			&= \int \int \Phi(s,\eta) \at(\chi_{s-r}) \dd s \otimes
			Q(d\eta) \di  \MCT^{(m)}(dr \otimes d\chi)\\
			&= \int\int \Phi(s,\eta) \di \Exp*{\Pi^{(m,r,\chi)}(ds\otimes
		d\eta)}\di \MCT^{(m)}(dr \otimes d\chi).
		\end{split}
	\end{equation}
	This shows that $\Pi^{(m+1)}$
	conditioned on $\MCT^{(m)}$ is equal in distribution to $\int
	\Pi^{(m,r,\chi)} \di \MCT^{(m)}(dr \otimes d\chi) = \MCT^{(m+1)}$.
	This,~\eqref{eq:clf_induction_step}, and the induction hypothesis show that
	\begin{equation}
		\begin{split}
			\MoveEqLeft\lim_{D\to\infty}\E\Biggl[ \exp\biggl(
				-\sum_{k=0}^{m+1}\sum_{i=1}^D
				F\bigr(Z^{D,k}(i)\bigr)\biggr) \Biggr]\\
			&=\E\Biggl[ \lim_{D\to\infty}\exp\biggl(
				-\sum_{k=0}^m\sum_{i=1}^D
				F\bigl(Z^{D,k}(i)\bigr)\biggr)
				\Exp[\bigg]{\exp\biggl( -\sum_{i=1}^D
			F\bigl(Z^{D,m+1}(i)\bigr) \biggl)\given
		(Z^{M,m})_{M\in\N},\MCT^{(m)}}\Biggr]\\
			&=\E\Biggl[ \exp\biggl(
				-\sum_{k=0}^m\int
				F(\eta_{\cdot-s})\di \MCT^{(k)}(ds \otimes d\eta)\biggr)\Exp[\bigg]{
				\exp\biggl( -\int
				F(\eta_{\cdot-s}) \di \MCT^{(m+1)}(ds \otimes d\eta)\biggr)\given
				\MCT^{(m)}}\Biggr]\\
			&=\E\biggl[ \exp\biggl(
				-\sum_{k=0}^{m+1}\int
				F(\eta_{\cdot-s})\di \MCT^{(k)}(ds \otimes d\eta)\biggr)
				\biggr].
		\end{split}
	\end{equation}
	This finishes the induction step $\N_0 \ni m \to m+1$ and hence
	proves~\eqref{eq:clf_fdds}.
	Due to Lemma~\ref{l:essentially_finitely_many_generations}
  and the fact that
  \begin{equation*}
		\lim_{m\to\infty}\Exp[\bigg]{\exp\biggl( - \sum_{k=0}^m \sum_{j=1}^n\int
		\phi_j(\eta_{t_j-s})\di \MCT^{(k)}(ds \otimes d\eta) \biggr)}
		=\Exp[\bigg]{\exp\biggl( -\sum_{k=0}^\infty \sum_{j=1}^n \int
		\phi_j(\eta_{t_j-s})\di \MCT^{(k)}(ds \otimes d\eta) \biggr)},
  \end{equation*}
  it suffices to
	consider finite sums over $k$ in~\eqref{eq:clf_fdds} to prove the convergence of
	finite-dimensional distributions of~\eqref{eq:clf_relcomptoshow}. Therefore,
	this finishes the proof of
	Lemma~\ref{l:convergence_of_the_loop_free_process}.
\end{proof}

\section{Convergence to a forest of excursions}
\label{sec:convergence_excursions}
To prove Theorem~\ref{thm:convergence} in
Subsection~\ref{ss:convergence_proof} below, we first show that the migration
level processes
and the loop-free processes
have the same limit as $D\to\infty$; see Lemma~\ref{l:replace_fidi} below. Our
method of proof is the integration by parts formula for semigroups;
see~\eqref{eq:weak_error_T}, \eqref{eq:weak_error_0},
and~\eqref{eq:weak_error} below. For this, we first derive moment estimates in
Subsection~\ref{ss:results_for_migration_levels} and uniform bounds on the
derivatives of the semigroups of the loop-free processes in
Lemma~\ref{l:regularity_semigroup}.
\subsection{Results for the migration level processes}
\label{ss:results_for_migration_levels}
The following lemma implies that individuals on the
same deme have essentially the same migration level in the limit
$D\to\infty$ and is analogous to Lemma~4.24 in~\textcite{Hutzenthaler2012}.
\begin{lemma}[Essentially one level per deme]
	\label{l:one_generation_per_island}
	Assume that \textup{Setting~\ref{set:migration_levels}} holds, that
	\begin{equation}
		\sup_{D\in\N}\Exp*{\left(\sum_{i=1}^D\sum_{k\in\N_0}
		X_0^{D,k}(i)\right)^2} < \infty,
		\label{eq:one_gen_ass1}
	\end{equation}
	and that
	\begin{equation}
		\lim_{D\to\infty}\Exp*{\sum_{i=1}^D\sum_{k\in\N_0}
		X_0^{D,k}(i)\sum_{m\in\N_0\setminus\{k\}} X_0^{D,m}(i)} = 0.
		\label{eq:one_gen_ass2}
	\end{equation}
	Then it holds for all $T \in (0,\infty)$ that
	\begin{equation} \label{eq:one_generation_per_island}
		\lim_{D\to\infty}\sup_{t\in[0,T]}\Exp*{\sum_{i=1}^D\sum_{k\in\N_0}
		X_t^{D,k}(i)
		\sum_{m\in\N_0\setminus\left\{ k
		\right\}}X_t^{D,m}(i)}
		=0.
	\end{equation}
\end{lemma}
\begin{proof}
	Fix $T \in (0,\infty)$ for the rest of the proof.
	For every $D,M\in\N$ we consider the stopping time $\tau_M^D$ defined
	in~\eqref{eq:tau_MD}.
	Since it holds for all $D\in\N$, all $i\in[D]$, and all
	$t\in[0,T]$ that $\sum_{m\in\N_0} X^{D,m}_t(i) \in [0,1]$, we obtain for all
	$D,M\in\N$ that
	\begin{equation}
		\begin{split}
			&\sup_{t\in[0,T]}\Exp*{\sum_{i=1}^D\sum_{k\in\N_0} X_{t}^{D,k}(i)
			\sum_{m\in\N_0\setminus\{k\}}X_{t}^{D,m}(i)}\\
			&\leq \sup_{t\in[0,T]}\Exp*{\sum_{i=1}^D\sum_{k\in\N_0} X_{t}^{D,k}(i)
			\sum_{m\in\N_0\setminus\{k\}}X_{t}^{D,m}(i)\1_{\{\tau^D_M > T\}}}
			+ \Exp*{\sup_{t\in[0,T]}\sum_{i=1}^D\sum_{k\in\N_0} X_{t}^{D,k}(i)
			\1_{\{\tau^D_M \leq T\}}}.
			\end{split}
			\label{eq:one_gen_red}
	\end{equation}
	Lemma~\ref{l:tau_theta} ensures that the
	second summand on the right-hand side of~\eqref{eq:one_gen_red} converges to
	zero uniformly in $D\in\N$ as $M\to\infty$. To
	prove~\eqref{eq:one_generation_per_island} it therefore suffices to show
	that the first summand on the right-hand side of~\eqref{eq:one_gen_red}
	converges to zero as $D\to\infty$ for all $M\in\N$.
	We fix $M\in\N$ for the rest of the proof.
	For all $D,K\in\N$ and all
	$t\in [0,\infty)$ let $M^{D,K}_t$ be real-valued
	random variables satisfying that a.s.
	\begin{equation}
		M^{D,K}_t = \sum_{i=1}^D\sum_{\substack{k,m=0 \\ m\neq k}}^K\int_0^t X^{D,k}_s(i)
		\sqrt{\frac{X_s^{D,m}(i)}{\sum_{l\in\N_0}X_s^{D,l}(i)}\sigma^2\left(
		\sum_{l\in\N_0}X_s^{D,l}(i) \right)}\dd W_s^m(i).
		\label{eq:one_gen_mg}
	\end{equation}
	It{\^o}'s formula,~\eqref{eq:XD_k},~\eqref{eq:one_gen_mg}, and
	Setting~\ref{set:coefficients} yield for all $D,K\in\N$ and all
	$t\in[0,\infty)$ that a.s.
	\begin{equation}
		\begin{split}
		\MoveEqLeft[3]\sum_{i=1}^D\sum_{k=0}^K X_{t}^{D,k}(i)
		\sum_{\substack{m=0 \\ m\neq k}}^K X_{t}^{D,m}(i)\\
		\leq{} & \sum_{i=1}^D\sum_{k=0}^K X_{0}^{D,k}(i)
			\sum_{\substack{m=0 \\ m\neq k}}^K X_{0}^{D,m}(i) + 2\int_0^t
			\sum_{i=1}^D\sum_{k=0}^K
			X_{s}^{D,k}(i)h_D(0)\dd s\\
			&+2\int_0^t \sum_{i=1}^D\sum_{\substack{k,m=0 \\ m \neq k}}^K X_{s}^{D,k}(i)
			\biggl(\frac{1}{D}\sum_{j=1}^D L_f X_{s}^{D,m-1}(j)
			+L_h X_{s}^{D,m}(i)\biggr) \dd s
			+ 2M^{D,K}_t.
		\end{split}
		\label{eq:one_gen_ito}
	\end{equation}
	For every $D,K\in\N$ the fact that
	\begin{equation}
		\sum_{i=1}^D \sum_{\substack{k,m=0 \\ m \neq k}}^K \int_0^T
		\bigl(X^{D,k}_s(i)\bigr)^2 \frac{X^{D,m}_s(i)}{\sum_{l\in\N_0}
		X^{D,l}_s(i)}\sigma^2\left( \sum_{l\in\N_0}X^{D,l}_s(i) \right) \dd s \leq
		DT L_\sigma
	\end{equation}
	implies that $(M^{D,K}_t)_{t\in[0,T]}$ is
	a martingale.
	Using this, using for all $D\in\N$ and all $s\in[0,T]$
	that $\sum_{i=1}^D\sum_{k\in\N_0} X^{D,k}_{s\wedge\tau_M^D}(i) \leq M$,
	and
	applying the optional sampling theorem (e.g.\ Theorem~2.2.13 in~\textcite{EthierKurtz1986}) and Tonelli's theorem, we obtain
	from~\eqref{eq:one_gen_ito} for
	all $D,K\in\N$ and all $t\in[0,T]$ that
	\begin{equation*}
		\begin{split}
			\E\left[\sum_{i=1}^D \sum_{k=0}^K X_{t\wedge \tau_M^D}^{D,k}(i)
			\sum_{\substack{m=0 \\ m\neq k}}^K X_{t\wedge\tau_M^D}^{D,m}(i)\right]
			\leq {}& \Exp*{\sum_{i=1}^D\sum_{k\in\N_0}
			X_0^{D,k}(i)\sum_{m\in\N_0\setminus\{k\}} X_0^{D,m}(i)}
			+ 2Th_D(0)M\\ 
			&+ \frac{1}{D}2TL_f M^2
			+2L_h\int_0^t\Exp*{\sum_{i=1}^D\sum_{k=0}^K
			X_{s\wedge\tau_M^D}^{D,k}(i)
			\sum_{\substack{m=0 \\ m\neq k}}^K X_{s\wedge\tau_M^D}^{D,m}(i)}\dd s.
		\end{split}
	\end{equation*}
	This, the fact that we have for all $D\in\N$ that $h_D(0)\leq 2\mu/D$, Gronwall's
	inequality, and the monotone convergence theorem ensure for all $D\in\N$
	that
	\begin{equation}
		\begin{split}
			\MoveEqLeft\sup_{t\in[0,T]}\Exp*{\sum_{i=1}^D\sum_{k\in\N_0} X_{t}^{D,k}(i)
			\sum_{m\in\N_0\setminus\{k\}}X_{t}^{D,m}(i)\1_{\{\tau_M^D >
			T\}}}\\
			&\leq\sup_{t\in[0,T]}\Exp*{\sum_{i=1}^D\sum_{k\in\N_0} X_{t\wedge\tau_M^D}^{D,k}(i)
			\sum_{m\in\N_0\setminus\{k\}}X_{t\wedge\tau_M^D}^{D,m}(i)}\\
			&\leq e^{2L_h T}\left(\Exp*{\sum_{i=1}^D\sum_{k\in\N_0}
			X_0^{D,k}(i)\sum_{m\in\N_0\setminus\{k\}} X_0^{D,m}(i)} +
			\frac{1}{D}2T(2\mu M+L_f M^2)\right).
		\end{split}
		\label{eq:one_gen_stopped}
	\end{equation}
	Letting $D\to\infty$ and applying~\eqref{eq:one_gen_ass2} finishes the proof
	of Lemma~\ref{l:one_generation_per_island}.
\end{proof}
The following lemma implies that the total mass is not evenly distributed over
all demes and is analogous to Lemma~4.23 in~\textcite{Hutzenthaler2012}.
\begin{lemma}[Concentration of mass]
	\label{l:concentration}
	Assume that \textup{Setting~\ref{set:excursion_measure}}
	and \textup{Setting~\ref{set:migration_levels}} hold, that
	\begin{equation}
		\sum_{k\in\N_0}\sup_{D\in\N}\Exp*{\sum_{i=1}^D X_0^{D,k}(i)} < \infty,
		\label{eq:concentration_ass1}
	\end{equation}
	that
	\begin{equation}
		\sup_{D\in\N}\Exp*{\left(\sum_{i=1}^D\sum_{k\in\N_0}
		X_0^{D,k}(i)\right)^2} < \infty,
		\label{eq:concentration_ass2}
	\end{equation}
	and that
	\begin{equation}
		\lim_{\delta\to 0}\sup_{D\in\N}\Exp*{\sum_{i=1}^D\sum_{k\in\N_0}
		\bigl(X_{0}^{D,k}(i)\wedge\delta\bigr)} = 0.
		\label{eq:concentration_ass3}
	\end{equation}
	Then it holds for all $T\in(0,\infty)$ that
	\begin{equation}
		\lim_{\delta\to 0}\sum_{k\in\N_0}\sup_{D\in\N}\sup_{t\in[0,T]}\Exp*{\sum_{i=1}^D
		\bigl(X_{t}^{D,k}(i)\wedge\delta\bigr)} = 0.
		\label{eq:concentration}
	\end{equation}
\end{lemma}
\begin{proof}
	Fix $T\in (0,\infty)$ for the rest of the proof. For every $D,M\in\N$ we
	consider the stopping time $\tau_M^D$ defined in~\eqref{eq:tau_MD}. Then it
	holds for all $\delta\in(0,\infty)$, all $K\in\N_0$, and all $M\in\N$ that
	\begin{equation}
		\begin{split}
			\MoveEqLeft[3]\sum_{k\in\N_0}\sup_{D\in\N}\sup_{t\in[0,T]}
			\Exp*{\sum_{i=1}^D \bigl(X_t^{D,k}(i)\wedge\delta\bigr)}\\
			\leq {}& \sum_{k=0}^K\sup_{D\in\N}\sup_{t\in[0,T]} \Exp*{\sum_{i=1}^D
			\bigl(X_t^{D,k}(i)\wedge\delta\bigr)\1_{\{\tau_M^D > T\}}}
			+ \sum_{k=K}^{\infty}\sup_{D\in\N}\sup_{t\in[0,T]} \Exp*{\sum_{i=1}^D
			X_t^{D,k}(i)}\\
			&+ \sum_{k=0}^K\sup_{D\in\N} \Exp*{\sup_{t\in[0,T]}\sum_{i=1}^D
			\sum_{m\in\N_0}X_t^{D,m}(i)\1_{\{\tau_M^D \leq T\}}}.
		\end{split}
		\label{eq:conc_reduction}
	\end{equation}
	Lemma~\ref{l:essentially_finitely_many_generations}
	and~\eqref{eq:concentration_ass1} imply that the second
	summand on the right-hand side of~\eqref{eq:conc_reduction} converges to
	zero as $K\to\infty$,
	while Lemma~\ref{l:tau_theta} and~\eqref{eq:concentration_ass2} ensure for
	all $K\in\N_0$ that the third summand on the right-hand side
	of~\eqref{eq:conc_reduction} converges to zero as $M\to\infty$.
	To prove~\eqref{eq:concentration} it therefore suffices to show for all
	$K\in\N_0$ and all $M\in\N$ that the first
	summand on the right-hand side of~\eqref{eq:conc_reduction} converges to
	zero as $\delta \to 0$. We fix $k\in\N_0$ and $M\in\N$ for the rest
	of the proof. 
	Setting~\ref{set:coefficients} implies the existence of $D_0\in\N$ such that
	for all $D\in\N\cap [D_0,\infty)$ we have $L_f M /D + h_D(1) \leq
	0$.
	For every $D\in\N\cap [D_0,\infty)$ let $\{(\Xt^D_t(i))_{t\in [0,\infty)}
	\colon i\in[D]\}$ be as in Lemma~\ref{l:decomposition}. Moreover, for every
	$D\in\N\cap[D_0,\infty)$ and every $i\in [D]$ let
	$(\Yt^{D,L_f M + Dh_D(0)}_{t,0}(i))_{t\in [0,\infty)}$ and
	$(Y^{D,L_f M + Dh_D(0)}_{t,0}(i))_{t\in[0,\infty)}$ be two solutions
	of~\eqref{eq:SDE_YD} with respect to the same Brownian motion satisfying
	a.s.~that
	$\Yt^{D,L_f M + Dh_D(0)}_{0,0}(i) = \Xt_0^D(i)$ and
	$Y^{D,L_f M + Dh_D(0)}_{0,0}(i) = 0$.
	%(for existence and uniqueness of strong solutions of~\eqref{eq:SDE_YD} see,
	%e.g., Theorem~5.4.22, Proposition~5.2.13, and Corollary~5.3.23 in
	%\textcite{KaratzasShreve1991}).
	Lemma~\ref{l:decomposition}, Setting~\ref{set:coefficients}, and Lemma~3.3
	in~\textcite{HutzenthalerWakolbinger2007} show for all
	$D\in\N\cap [D_0,\infty)$, all $i\in[D]$, and all $t\in[0,T]$ that
	$\Xt_t^{D}(i)$ is stochastically bounded from above by
	$\Yt_{t,0}^{D,L_fM+Dh_D(0)}(i)$ on the event $\{\tau^D_M > T\}$.
	This, \eqref{eq:all_levels}, the fact that for all $a,b,\delta \in
	[0,\infty)$ it holds that $\abs{a\wedge \dl - b\wedge \dl} =
	\abs{a\wedge \dl - b\wedge \dl} \wedge \dl \leq \abs{a-b}
	\wedge \delta$, and Jensen's inequality for the conditional expectation
	applied for all $\delta \in (0,\infty)$ with the concave function $[0,1] \ni
	x \mapsto x \wedge \dl$ yield for all $\delta \in (0,\infty)$, all
	$t\in[0,T]$, and all $D\in\N\cap [D_0,\infty)$ that
	\begin{equation}
		\begin{split}
			\MoveEqLeft[3]\Exp*{\sum_{i=1}^D\bigl(X_t^{D,k}(i)\wedge\dl\bigr)\1_{\{\tau_M^D
				> T\}}}\\
			\leq{} &
				\sum_{i=1}^D\Exp*{
				\bigl(\Xt_t^{D}(i)\wedge\dl\bigr)\1_{\{\tau_M^D> T\}}}
			\leq{}
				\sum_{i=1}^D\Exp[\Big]{\Yt_{t,0}^{D,L_f M+Dh_D(0)}(i) \wedge \dl}\\
			\leq{} &
				\sum_{i=1}^D\Exp[\Big]{Y_{t,0}^{D,L_f M +Dh_D(0)}(i) \wedge \dl}\\
				&+ \sum_{i=1}^D\Exp[\bigg]{\Exp[\Big]{\abs[\big]{\Yt_{t,0}^{D,L_f
				M+Dh_D(0)}(i) - Y_{t,0}^{D,L_f M+Dh_D(0)}(i)} \given \Xt_0^D(i)} \wedge
				\dl}.
		\end{split}
		\label{eq:tau_not_yet}
	\end{equation}
	This, Lemma~\ref{l:YD_L1_distance}, and the fact that it holds for all
	$t\in[0,T]$ that $e^{L_h t} \geq 1$ ensure for all $\delta \in (0,\infty)$,
	all $t\in[0,T]$, and all $D\in\N\cap [D_0,\infty)$ that
	\begin{equation}
		\begin{split}
			\Exp*{\sum_{i=1}^D\bigl(X_t^{D,k}(i)\wedge\dl\bigr)\1_{\{\tau_M^D >
				T\}}}
			%\leq{}&
			%	\sum_{i=1}^D\Exp[\Big]{Y_{t,0}^{D,L_f M +Dh_D(0)}(i) \wedge \dl}\\
			%	&+ \sum_{i=1}^D\Exp[\bigg]{\Exp[\Big]{\abs[\big]{\Yt_{t,0}^{D,L_f
			%	M+Dh_D(0)}(i)
			%	- Y_{t,0}^{D,L_f M+Dh_D(0)}(i)} \given \Xt_0^D(i)} \wedge
			%	\dl}\\
			\leq{} &
			D\Exp[\Big]{Y_{t,0}^{D,L_f M+Dh_D(0)}(1)\wedge\dl}
				+e^{L_h t}\sum_{i=1}^D\Exp*{\Xt^D_0(i)\wedge\dl}.
		\end{split}
		\label{eq:tau_not_yet2}
	\end{equation}
	This,
	Lemma~\ref{l:vanishing_immigration_general} with $(g_D)_{D\geq D_0} = (L_f
	M+Dh_D(0))_{D\geq D_0}$ and $g = L_f M + \mu$,
	Equation~\eqref{eq:all_levels}, and subadditivity for all $\delta \in
	(0,\infty)$ of $[0,1] \ni x \mapsto x \wedge \dl$ imply for all
	$\delta \in (0,\infty)$ that
	\begin{equation}
		\begin{split}
			\MoveEqLeft\sup_{t\in[0,T]} \varlimsup_{D\to\infty}
			\Exp*{\sum_{i=1}^D\bigl(X_t^{D,k}(i)\wedge\dl\bigr)\1_{\{\tau_M^D>
			T\}}}\\
			&\leq (L_f M +
			\mu)\int\int_0^T\bigl(\chi_{T-r}\wedge\dl\bigr)\dd r \di Q(d\chi)
			+e^{L_h
			T}\sup_{D\in\N}\Exp*{\sum_{i=1}^D\sum_{m\in\N_0}\bigl(X^{D,m}_0(i)\wedge\dl\bigr)}.
		\end{split}
		\label{eq:tau_not_yet_upper_bound}
	\end{equation}
	The first summand on the right-hand side
	of~\eqref{eq:tau_not_yet_upper_bound} converges to zero
	as $\delta \to 0$ by the
	dominated convergence theorem and Lemma~\ref{l:finite_excursion_area}.
	The second summand on the
	right-hand side of~\eqref{eq:tau_not_yet_upper_bound} converges to zero as
	$\delta \to 0$ due to~\eqref{eq:concentration_ass3}.
	This completes the proof of Lemma~\ref{l:concentration}.
\end{proof}

\subsection{The migration level processes and the loop-free processes have the
same limit} \label{ss:reduction_loop_free}
For the rest of this paragraph, we fix $K\in\N_0$ and assume that
Setting~\ref{set:migration_levels} holds. For
all $D\in\N$ we denote by $\{S^D_t \colon t\in[0,\infty)\}$
the strongly continuous contraction semigroup on
$C([0,1]^{[D]\times[K]_0},\R)$ associated with
$\{(Z^{D,k}_t(i))_{t\in[0,\infty)} \colon (i,k) \in [D] \times
[K]_0 \}$; see Remark~3.2 in~\textcite{ShigaShimizu1980}. Then for all $D\in\N$, all
$t\in[0,\infty)$, all $\psi \in
C([0,1]^{[D]\times[K]_0},\R)$, and all $x \in
[0,1]^{[D]\times[K]_0}$ it holds that
\begin{equation}
	(S_t^D\psi)(x) =
	\Exp*{\psi\left(\bigl(Z^{D,k,x}_t(i)
	\bigr)_{(i,k)\in[D]\times[K]_0}\right)}.
	\label{eq:semigroup}
\end{equation}
For every $D\in\N$ the semigroup $\{S_t^D \colon t\in[0,\infty)\}$ has as its generator the closure
of the operator $G^D$ acting on
$C^2([0,1]^{[D]\times[K]_0},\R)$, given
for all $\psi\in C^2([0,1]^{[D]\times[K]_0},\R)$ and all $x =
(x_{i,k})_{(i,k)\in[D]\times[K]_0} \in
[0,1]^{[D]\times[K]_0}$ by
\begin{equation}
	\begin{split}
		(G^{D}\psi)(x) ={} &\sum_{i=1}^D\sum_{k=0}^K\left(
		\frac{\1_{k > 0}}{D}\sum_{j=1}^D x_{j,\abs{k-1}}f(x_{j,\abs{k-1}},x_{i,k}) +
		\hht_D(x_{i,k}) + \1_{k=0}
		h_D(0) 
		\right)\frac{\partial \psi}{\partial x_{i,k}}(x)\\
		&+\frac{1}{2}\sum_{i=1}^D\sum_{k=0}^K
		\sigma^2(x_{i,k})\frac{\partial^2\psi}{\partial
		x^2_{i,k}}(x).
	\end{split}
	\label{eq:generatorZ}
\end{equation}
The following lemma establishes uniform bounds on the derivatives of the
semigroups of the loop-free processes.
\begin{lemma}[Uniform $C^2$-bound] \label{l:regularity_semigroup}
	Assume that \textup{Setting~\ref{set:migration_levels}} holds,
	let $K\in\N_0$, and for
	every $D\in\N$ let
	$\{S_t^D \colon t\in[0,\infty)\}$ be as in~\eqref{eq:semigroup}.
	Then there exists $c \in [0,\infty)$ such that it holds for all $D\in\N$, all
	$t\in[0,\infty)$, and all $\psi^D\in
		C^2([0,1]^{[D]\times[K]_0},\R)$ that $S_t^D\psi^D \in
		C^2([0,1]^{[D]\times[K]_0},\R)$ and
	\begin{equation}
		\norm{S_t^D\psi^D}_{C^2} \leq
		e^{ct}\norm{\psi^D}_{C^2}.
		\label{eq:semigroup_estimate}
	\end{equation}
\end{lemma}
\begin{proof}
	For every $D\in\N$ and every $(i,k)\in [D]\times[K]_0$
	let $f_{i,k}\colon [0,1]^{[D]\times[K]_0} \to
	\R$ be the function that satisfies for all $x =
	(x_{j,l})_{(j,l)\in[D]\times[K]_0} \in
	[0,1]^{[D]\times[K]_0}$ that
	\begin{equation}
		f_{i,k}(x) =
		\frac{\1_{k > 0}}{D}\sum_{j=1}^D x_{j,\abs{k-1}}f(x_{j,\abs{k-1}},x_{i,k}) +
		\hht_D(x_{i,k}) + \1_{k=0} h_D(0).
	\end{equation}
	Then it holds for all $D\in\N$ and all $\alpha \in
	\N_0^{[D]\times[K]_0}$ with $\abs{\alpha}=1$ that
	\begin{equation}
		\begin{split}
			\sum_{i=1}^D \sum_{k=0}^K \norm{\partial^\alpha f_{i,k}}_\infty
			\leq \norm{f}_\infty + \norm[\big]{\tfrac{\partial f}{\partial x_1}}_\infty +
			\norm[\big]{\tfrac{\partial f}{\partial x_2}}_\infty  +
			\norm[\big]{\tfrac{d \hht_D}{dx}}_\infty
			\leq 3\norm{f}_{C^1} + \norm{\hht_D}_{C^1}
		\end{split}
	\end{equation}
	and for all $D\in\N$ and all $\alpha \in \N_0^{[D]\times[K]_0}$ with
	$\abs{\alpha}=2$ that
	\begin{equation}
		\begin{split}
			\sum_{i=1}^D \sum_{k=0}^K \norm{\partial^\alpha f_{i,k}}_\infty
			&\leq 2\bigl(\norm[\big]{\tfrac{\partial f}{\partial
				x_1}}_\infty + \norm[\big]{\tfrac{\partial f}{\partial
			x_2}}_\infty\bigr) +
				\norm[\big]{\tfrac{\partial^2 f}{\partial x_1^2}}_\infty +
				2\norm[\big]{\tfrac{\partial^2 f}{\partial x_1 \partial x_2}}_\infty +
				\norm[\big]{\tfrac{\partial^2 f}{\partial x_2^2}}_\infty +
			\norm[\big]{\tfrac{d^2 \hht_D}{dx^2}}_\infty\\
			&\leq 8\norm{f}_{C^2} + \norm{\hht_D}_{C^2}.
		\end{split}
	\end{equation}
	We define
	\begin{equation}
		c \defeq 4\Bigl( 8\norm{f}_{C^2} +
		\sup_{D\in\N}\norm{\hht_D}_{C^2} \Bigr) +
		\tfrac{1}{2}\norm{\sigma^2}_{C^2},
	\end{equation}
	which is finite due to Setting~\ref{set:coefficients}.
	Then Theorem~4.1 in~\textcite{HutzenthalerPieper2018} shows for all $D\in\N$, all
	$t\in[0,\infty)$, and all $\psi^D \in
	C^2([0,1]^{[D]\times[K]_0}, \R)$ that $S^D_t\psi^D \in
	C^2([0,1]^{[D]\times[K]_0},\R)$ and
	that~\eqref{eq:semigroup_estimate} holds.
	This finishes the proof of Lemma~\ref{l:regularity_semigroup}.
\end{proof}

The following lemma follows immediately from Theorem~3.16
in~\textcite{Liggett2010} and Lemma~\ref{l:regularity_semigroup} above.
\begin{lemma}[Kolmogorov backward equation] \label{l:classical_solution}
	Assume that \textup{Setting~\ref{set:migration_levels}} holds,
	let $T\in(0,\infty)$, let $D\in\N$, let
	$K\in\N_0$, let
	$\psi \in C^2([0,1]^{[D]\times[K]_0},\R)$, let
	$\{S_t^D \colon t \in[0,\infty)\}$ be as in~\eqref{eq:semigroup}, let $G^D$ be as
	in~\eqref{eq:generatorZ}, and define the
	function $u\colon [0,T]\times [0,1]^{[D]\times[K]_0} \to \R$
	by
	\begin{equation}
		[0,T] \times [0,1]^{[D]\times[K]_0} \ni (t,x) \mapsto
		u(t,x) \defeq
		(S_{T-t}^D\psi)(x).
	\end{equation}
	Then it holds that
		$u \in C^{1,2}([0,T] \times
		[0,1]^{[D]\times[K]_0},\R)$
	and it holds for all $t \in [0,T]$ and all $x\in
	[0,1]^{[D]\times[K]_0}$ that $u(T,x) = \psi(x)$ and
	\begin{equation}
		\begin{split}
			&\frac{\partial u}{\partial t}(t,x) +
			(G^Du)(t,x) = 0.
		\end{split}
		\label{eq:PDE_JansonTysk}
	\end{equation}
\end{lemma}
%\begin{proof}
%	This follows from Theorem~3.16 in~\textcite{Liggett2010} and
%	Lemma~\ref{l:regularity_semigroup}.
%\end{proof}

The following lemma shows that finitely many levels of the migration level processes
and of the loop-free processes have the same limit as
$D\to\infty$ at every fixed time point.
\begin{lemma}[Asymptotic equality for one-dimensional distributions] \label{l:replace_1di}
	Assume that \textup{Setting~\ref{set:excursion_measure}}
	and \textup{Setting~\ref{set:migration_levels}} hold, that
	\begin{equation}
		\sum_{k\in\N_0}\sup_{D\in\N}\Exp*{\sum_{i=1}^D X_0^{D,k}(i)} < \infty,
		\label{eq:replace_1di_initial2}
	\end{equation}
	that
	\begin{equation}
		\sup_{D\in\N} \Exp*{\left( \sum_{i=1}^D\sum_{k\in\N_0}X_0^{D,k}(i)
		\right)^2} < \infty,
		\label{eq:replace_1di_initial}
	\end{equation}
	that
	\begin{equation}
		\lim_{D\to\infty}\Exp*{\sum_{i=1}^D\sum_{k\in\N_0}
		X_0^{D,k}(i)\sum_{m\in\N_0\setminus\{k\}} X_0^{D,m}(i)} = 0,
		\label{eq:replace_1di_initial3}
	\end{equation}
	and that
	\begin{equation}
		\lim_{\delta\to 0}\sup_{D\in\N}\Exp*{\sum_{i=1}^D\sum_{k\in\N_0}
		\bigl(X_{0}^{D,k}(i)\wedge\delta\bigr)} = 0,
		\label{eq:replace_1di_initial4}
	\end{equation}
	let $T\in(0,\infty)$ and $K\in\N_0$,
	for every $D\in\N$ let $\psi^D\in
	C^2([0,1]^{[D]\times[K]_0},\R)$, and
	suppose that $\sup_{D\in\N}\norm{\psi^D}_{C^2} < \infty$.
	Then it holds that
	\begin{equation}
		\lim_{D\to\infty} 
		\abs*{\Exp*{\psi^D\left(\bigl(X_T^{D,k}(i)
		\bigr)_{(i,k)\in[D]\times[K]_0}\right)}
		-
		\Exp*{\psi^D\left(\bigl(Z_T^{D,k}(i)
		\bigr)_{(i,k)\in[D]\times[K]_0}\right)}}  = 0.
		\label{eq:replace_1di}
	\end{equation}
\end{lemma}
\begin{proof}
	For every $D\in\N$ let $\{S_t^D \colon t\in[0,\infty)\}$ be as
	in~\eqref{eq:semigroup}, let $G^D$ be as in~\eqref{eq:generatorZ}, and
	define the function $u^D\colon [0,T]\times
	[0,1]^{[D]\times[K]_0} \to \R$ by
	\begin{equation}
		[0,T] \times [0,1]^{[D]\times[K]_0} \ni (t,x) \mapsto
		u^D(t,x) \defeq
		(S_{T-t}^D\psi^D)(x).
		\label{eq:uD_definition}
	\end{equation}
	Equations~\eqref{eq:uD_definition} and~\eqref{eq:semigroup} yield for all $D\in\N$ that
	\begin{equation}
		\Exp*{u^{D}\left(T,\bigl(X_T^{D,k}(i)\bigr)_{(i,k)\in[D]\times[K]_0}\right)}
		= \Exp*{\psi^D\left(\bigl(X_T^{D,k}(i)
		\bigr)_{(i,k)\in[D]\times[K]_0}\right)}
		\label{eq:weak_error_T}
	\end{equation}
	and
	\begin{equation}
		\Exp*{u^{D}\left(0,\bigl(X_0^{D,k}(i)\bigr)_{(i,k)\in[D]\times[K]_0}
		\right)}
		= \Exp*{\psi^D\left(\bigl(Z_T^{D,k}(i)
		\bigr)_{(i,k)\in[D]\times[K]_0}\right)}.
		\label{eq:weak_error_0}
	\end{equation}
	This shows that~\eqref{eq:replace_1di} is implied by
	\begin{equation}
		\lim_{D\to\infty}{\abs*{\Exp*{u^{D}\left(T,
		\bigl(X_T^{D,k}(i)\bigr)_{(i,k)\in[D]\times[K]_0} \right)} -
		\Exp*{u^{D}\left(0,
		\bigl(X_0^{D,k}(i)\bigr)_{(i,k)\in[D]\times[K]_0} \right)}}} = 0.
		\label{eq:weak_approach_toshow}
	\end{equation}
	Lemma~\ref{l:classical_solution} implies for all $D\in\N$ that $u^{D} \in
	C^{1,2}([0,T] \times [0,1]^{[D]\times[K]_0},\R)$ and for all
	$D\in\N$, all $t \in [0,T]$, and all $x\in
	[0,1]^{[D]\times[K]_0}$ that
	\begin{equation}
		\frac{\partial u^D}{\partial t}(t,x) + (G^{D} u^{D})(t,x) =
		0.
		\label{eq:PDE_u}
	\end{equation}
	For every $D\in\N$ (a small variation with different orders of
	differentiability of) Whitney's extension theorem (see, e.g., Theorem~2.3.6
	in~\textcite{Hoermander1990}) ensures
	that $u^D$ can be extended to a function in $C^{1,2}([0,\infty) \times
	\R^{[D]\times[K]_0},\R)$. Then It{\^o}'s
	formula, \eqref{eq:XD_k}, \eqref{eq:PDE_u}, \eqref{eq:generatorZ}, and
	Tonelli's theorem yield
	for all $D\in\N$ that
	\begin{equation}
		\begin{split}
			\MoveEqLeft[3]\Exp*{u^{D}\left(T,
			\bigl(X_T^{D,k}(i)\bigr)_{(i,k)\in[D]\times[K]_0} \right)} -
			\Exp*{u^{D}\left(0,
			\bigl(X_0^{D,k}(i)\bigr)_{(i,k)\in[D]\times[K]_0}
			\right)}\\
			={} & \int_0^T \E\Biggl[\sum_{i=1}^D \sum_{k=0}^K
			\tfrac{\partial u^D}{\partial x_{i,k}} \left(s,
			\bigl(X_s^{D,m}(j)\bigr)_{
			(j,m)\in[D]\times[K]_0}
			\right)\\
			&\quad \times\Biggl\{ \tfrac{1}{D}\sum_{j=1}^D
			X_s^{D,k-1}(j)\biggl(f\biggl(\sum_{m\in\N_0}X_s^{D,m}(j),
			\sum_{m\in\N_0}X_s^{D,m}(i)\biggr) - f\bigl(X_s^{D,k-1}(j),
			X_s^{D,k}(i)\bigr)\biggr)\\
			&\quad+
			\tfrac{X_s^{D,k}(i)}{\sum_{m\in\N_0}X_s^{D,m}(i)}\hht_D\biggl(\sum_{m\in\N_0}X_s^{D,m}(i)\biggr)
			- \hht_D\bigl(X_s^{D,k}(i)\bigr)\Biggr\}\\
			&+\tfrac{1}{2}\sum_{i=1}^D\sum_{k=0}^K\tfrac{\partial^2 u^D}{\partial
			x^2_{i,k}} \left(s,
			\bigl(X_s^{D,m}(j)\bigr)_{(j,m)\in[D]\times[K]_0}
			\right)\\
			&\quad\times\Biggl\{\tfrac{X_s^{D,k}(i)}{\sum_{m\in\N_0}X_s^{D,m}(i)}
			\sigma^2\biggl(\sum_{m\in\N_0}X_s^{D,m}(i)\biggr) -
			\sigma^2\bigl(X_s^{D,k}(i)\bigr)\Biggr\}\Biggr] \dd s .
		\end{split}
		\label{eq:weak_error}
	\end{equation}
	Setting~\ref{set:coefficients} implies for all $D\in\N$, all
	$(x,y)\in \{(x_1,x_2) \in [0,1]^2 \colon x_1 + x_2 \leq 1\}$, and all $\delta \in (0,1)$ that
	\begin{equation}
		\begin{split}
			\abs[\big]{\tfrac{x}{x+y}\hht_D(x+y) - \hht_D(x)}
			&\leq \tfrac{x}{x+y}\abs[\big]{\hht_D(x+y) - \hht_D(x)} +
			\tfrac{y}{x+y}\abs[\big]{\hht_D(x)}\\
			&\leq 2L_h\tfrac{xy}{x+y} \\
			&\leq 2L_h(x \wedge y)\\
			&\leq \1_{x \leq \delta}2L_h(x \wedge \delta) + \1_{x >
			\delta}2L_h y\\
			&\leq 2L_h(x \wedge \delta) + \tfrac{2L_h}{\delta} xy.
		\end{split}
		\label{eq:Lipschitz_hD}
	\end{equation}
	Analogously, Setting~\ref{set:coefficients} implies for all $(x,y)\in
	\{(x_1,x_2) \in [0,1]^2 \colon x_1 + x_2 \leq 1\}$ and all $\delta \in (0,1)$ that
	\begin{equation}
		\abs[\big]{\tfrac{x}{x+y}\sigma^2(x+y) - \sigma^2(x)}
		\leq 2L_\sigma(x \wedge \delta) + \tfrac{2L_\sigma}{\delta}xy.
		\label{eq:Lipschitz_sigma2}
	\end{equation}
	Equations~\eqref{eq:Lipschitz_hD} and~\eqref{eq:Lipschitz_sigma2} with $x =
	X_s^{D,k}(i)$ and $y = \sum_{m\in\N_0\setminus\{k\}}X_s^{D,m}(i)$,
	Setting~\ref{set:coefficients}, and~\eqref{eq:weak_error} show for all
	$\delta \in (0,1)$ and all $D\in\N$ that
	\begin{equation}
		\begin{split}
			\MoveEqLeft[3]\abs*{\Exp*{u^{D}\left(T,
			\bigl(X_T^{D,k}(i)\bigr)_{(i,k)\in[D]\times[K]_0} \right)} -
			\Exp*{u^{D}\left(0,
			\bigl(X_0^{D,k}(i)\bigr)_{(i,k)\in[D]\times[K]_0}
			\right)}}\\
			\leq {} &\sup_{t\in[0,T]}\norm{u^{D}(t,\cdot)}_{C^2}T
			\left( L_f \sup_{t\in[0,T]}
			\Exp*{\sum_{j=1}^D\sum_{k\in\N_0}X_t^{D,k-1}(j)
			\sum_{m\in\N_0\setminus\{k-1\}}X_t^{D,m}(j)}\right.\\
			&+
			\frac{L_f}{D}\Exp*{ \sup_{t\in[0,T]} \left(\sum_{i=1}^D\sum_{k\in\N_0}
			X_t^{D,k}(i)\right)^2}
			+
			(2L_h + L_\sigma)\sup_{t\in[0,T]}
			\Exp*{\sum_{i=1}^D\sum_{k\in\N_0}\bigl(X_t^{D,k}(i)\wedge \delta\bigr)}\\
			&\left.+ \frac{2L_h + L_\sigma}{\delta}\sup_{t\in[0,T]}
			\Exp*{\sum_{i=1}^D\sum_{k\in\N_0}X_t^{D,k}(i)\sum_{m\in\N_0\setminus\{k\}}X_t^{D,m}(i)}
			\right).
		\end{split}
		\label{eq:uD_convergence_estimate}
	\end{equation}
	Lemma~\ref{l:regularity_semigroup} and
	$\sup_{D\in\N}\norm{\psi^D}_{C^2} < \infty$ imply that
	$\sup_{D\in\N}\sup_{t\in[0,T]}\norm{u^D(t,\cdot)}_{C^2} < \infty$. The first
	and the fourth summand on the right-hand side
	of~\eqref{eq:uD_convergence_estimate} converge to zero as $D\to\infty$ by
	Lemma~\ref{l:one_generation_per_island} and
	assumptions~\eqref{eq:replace_1di_initial}
	and~\eqref{eq:replace_1di_initial3}. The second summand on the right-hand
	side of~\eqref{eq:uD_convergence_estimate} converges to zero as
	$D\to\infty$ by Lemma~\ref{l:X_Z_second_moment} and
	assumption~\eqref{eq:replace_1di_initial}. The third summand on the
	right-hand side of~\eqref{eq:uD_convergence_estimate} converges to
	zero uniformly in $D\in\N$ as $\delta \to 0$ by Lemma~\ref{l:concentration}
	and assumptions \eqref{eq:replace_1di_initial2},
	\eqref{eq:replace_1di_initial}, and \eqref{eq:replace_1di_initial4}.
	By letting first $D\to\infty$ and then $\delta \to
	0$,~\eqref{eq:weak_approach_toshow} follows.
	This finishes the proof of Lemma~\ref{l:replace_1di}.
\end{proof}
The following lemma shows that in the situation of Setting~\ref{set:initial_dist_level0},
the assumptions of Lemma~\ref{l:one_generation_per_island} and
Lemma~\ref{l:concentration} and some of the assumptions in
Lemma~\ref{l:replace_1di} are satisfied.
\begin{lemma}[Well-behaved initial distribution] \label{l:initial_dist_level0}
	Assume that \textup{Setting~\ref{set:initial_dist_level0}} holds. Then it
	holds that
	\begin{equation}
		\sum_{k\in\N_0}\sup_{D\in\N}\Exp*{\sum_{i=1}^D X^{D,k}_0(i)} < \infty,
		\label{eq:initial_dist_ass1}
	\end{equation}
	that
	\begin{equation}
		\sup_{D\in\N}\Exp*{\left(\sum_{i=1}^D\sum_{k\in\N_0}
		X_0^{D,k}(i)\right)^2} < \infty,
		\label{eq:initial_dist_ass2}
	\end{equation}
	that
	\begin{equation}
		\lim_{D\to\infty} \Exp*{\sum_{i=1}^D\sum_{k\in\N_0}
		X_0^{D,k}(i)\sum_{m\in\N_0\setminus\{k\}} X_0^{D,m}(i)} = 0,
		\label{eq:initial_dist_ass3}
	\end{equation}
	and that
	\begin{equation}
		\lim_{\delta\to 0}\sup_{D\in\N}\Exp*{\sum_{i=1}^D\sum_{k\in\N_0}
		\bigl(X_{0}^{D,k}(i)\wedge\delta\bigr)} = 0.
		\label{eq:initial_dist_ass4}
	\end{equation}
\end{lemma}
\begin{proof}
	Equations \eqref{eq:initial_dist_ass1}, \eqref{eq:initial_dist_ass2}, and
	\eqref{eq:initial_dist_ass3} follow immediately from the structure of the
	initial distribution given in Setting~\ref{set:initial_dist_level0} and
	from~\eqref{eq:initial_second_moment}.
	Moreover, Setting~\ref{set:initial_dist_level0} and the dominated convergence
	theorem show that
	\begin{equation}
		\lim_{\delta\to 0}\sup_{D\in\N}\Exp*{\sum_{i=1}^D\sum_{k\in\N_0}
		\bigl(X_{0}^{D,k}(i)\wedge\delta\bigr)} = \lim_{\delta\to
		0}\Exp*{\sum_{i=1}^\infty \left(X_0(i) \wedge \dl\right)} = 0.
	\end{equation}
	This completes the proof of Lemma~\ref{l:initial_dist_level0}.
\end{proof}
The following lemma uses the Markov property in order to generalize
Lemma~\ref{l:replace_1di} to finitely many time points.
\begin{lemma}[Asymptotic equality for f.d.d.s] \label{l:replace_induction}
	Assume that \textup{Setting~\ref{set:initial_dist_level0}} holds, let
	$K\in\N_0$, let $t_1,t_2,\ldots \in [0,\infty)$ with $t_1 < t_2 <
	\dotsb$, for all $D,j\in\N$ let
	$\psi^D_j\in C^2([0,1]^{[D]\times[K]_0},\R)$, and assume
	for all $j\in\N$ that $\sup_{D\in\N}\norm{\psi^D_j}_{C^2} < \infty$.
	Then it holds for all $n\in\N$ that
	\begin{equation}
		\lim_{D\to\infty} \abs*{
		\Exp*{\prod_{j=1}^n\psi^D_j\left(\bigl(X_{t_j}^{D,k}(i)
		\bigr)_{(i,k)\in[D]\times[K]_0}\right)}
		-
		\Exp*{\prod_{j=1}^n\psi^D_j\left(\bigl(Z_{t_j}^{D,k}(i)
		\bigr)_{(i,k)\in[D]\times[K]_0}\right)}} = 0.
		\label{eq:replace_fidis}
	\end{equation}
\end{lemma}
\begin{proof}
	We prove~\eqref{eq:replace_fidis} by induction on $n\in\N$.
	The base case $n=1$ has been settled in Lemma~\ref{l:replace_1di},
	where the conditions \eqref{eq:replace_1di_initial2},
	\eqref{eq:replace_1di_initial}, \eqref{eq:replace_1di_initial3}, and
	\eqref{eq:replace_1di_initial4} are satisfied due to Lemma~\ref{l:initial_dist_level0}.
	To show the induction step $\N \ni n \to n+1$,
	for every $D\in\N$ we define the function $\psi^D\colon
	[0,1]^{[D]\times[K]_0} \to \R$ by
	\begin{equation}
		[0,1]^{[D]\times[K]_0} \ni x \mapsto \psi^D(x) \defeq
		\psi^D_n(x)\Exp*{\psi^D_{n+1}\left(\bigl(Z_{t_{n+1}-t_n}^{D,k,x}(i)
		\bigr)_{(i,k)\in[D]\times[K]_0}\right)}.
	\end{equation}
	Then Lemma~\ref{l:regularity_semigroup} proves for every $D\in\N$ that $\psi^D\in
	C^2([0,1]^{[D]\times[K]_0},\R)$.
	Moreover, it follows from Lemma~\ref{l:regularity_semigroup} that
	$\sup_{D\in\N}\norm{\psi^D}_{C^2} < \infty$.
	Therefore, the induction hypothesis (applied to
	$\psi^D_1,\dotsc,\psi^D_{n-1}, \psi^D$) yields that
	\begin{equation}
		\begin{split}
			&\lim_{D\to\infty} \left\vert
			\Exp*{\prod_{j=1}^{n-1}\psi^D_j\left(\bigl(X_{t_j}^{D,k}(i)
			\bigr)_{(i,k)\in[D]\times[K]_0}\right)\psi^D\left(\bigl(X_{t_n}^{D,k}(i)
			\bigr)_{(i,k)\in[D]\times[K]_0}\right)} \right.\\
			&\qquad\quad-
			\left.\Exp*{\prod_{j=1}^{n-1}\psi^D_j\left(\bigl(Z_{t_j}^{D,k}(i)
			\bigr)_{(i,k)\in[D]\times[K]_0}\right)\psi^D\left(\bigl(Z_{t_n}^{D,k}(i)
			\bigr)_{(i,k)\in[D]\times[K]_0}\right)} \right\vert
			= 0.
		\end{split}
		\label{eq:fidis_indhypo}
	\end{equation}
	By the Markov property it holds for all $D\in\N$ that
	\begin{equation}
		\begin{split}
			\MoveEqLeft
			\Exp*{\prod_{j=1}^{n-1}\psi^D_j\left(\bigl(Z_{t_j}^{D,k}(i)
			\bigr)_{(i,k)\in[D]\times[K]_0}\right)\psi^D\left(\bigl(Z_{t_n}^{D,k}(i)
			\bigr)_{(i,k)\in[D]\times[K]_0}\right)}\\
			&= \Exp*{\prod_{j=1}^{n+1}\psi^D_j\left(\bigl(Z_{t_j}^{D,k}(i)
			\bigr)_{(i,k)\in[D]\times[K]_0}\right)}.
		\end{split}
	\end{equation}
	Moreover, we observe for all $D\in\N$ that
	\begin{equation}
		\begin{split}
			\MoveEqLeft\Exp*{\prod_{j=1}^{n-1}\psi^D_j\left(\bigl(X_{t_j}^{D,k}(i)
			\bigr)_{(i,k)\in[D]\times[K]_0}\right)\psi^D\left(\bigl(X_{t_n}^{D,k}(i)
			\bigr)_{(i,k)\in[D]\times[K]_0}\right)}\\
			&=
			\E\left[\prod_{j=1}^{n}\psi^D_j\left(\bigl(X_{t_j}^{D,k}(i)
			\bigr)_{(i,k)\in[D]\times[K]_0}\right)\right.\\
			&\quad \left.\phantom{\prod_{j=1}^n}\times\!\left.\Exp*{\psi^D_{n+1}\left(\bigl(Z_{t_{n+1}-t_{n}}^{D,k,x}(i)
			\bigr)_{(i,k)\in[D]\times[K]_0}\right)}
			\right\rvert_{x =
			(X_{t_n}^{D,k}(i))_{(i,k)\in[D]\times\N_0}}\right].
		\end{split}
	\end{equation}
	When the initial distribution is
	given by $(X_{t_n}^{D,k}(i))_{(i,k)\in[D]\times\N_0}$,
	the conditions \eqref{eq:replace_1di_initial2}, \eqref{eq:replace_1di_initial},
	\eqref{eq:replace_1di_initial3}, and \eqref{eq:replace_1di_initial4} are fulfilled due to
	Lemma~\ref{l:initial_dist_level0},
	Lemma~\ref{l:essentially_finitely_many_generations},
	Lemma~\ref{l:X_Z_second_moment},
	Lemma~\ref{l:one_generation_per_island}, and Lemma~\ref{l:concentration}.
	Therefore, Lemma~\ref{l:replace_1di} implies that
	\begin{equation}
		\begin{split}
			&\lim_{D\to\infty} \left\vert
			\E\left[\prod_{j=1}^{n}\psi^D_j\left(\bigl(X_{t_j}^{D,k}(i)
			\bigr)_{(i,k)\in[D]\times[K]_0}\right)\right.\right.\\
			&\qquad \left.\phantom{\prod_{j=1}^n}\times\!\left.\Exp*{\psi^D_{n+1}\left(\bigl(Z_{t_{n+1}-t_{n}}^{D,k,x}(i)
			\bigr)_{(i,k)\in[D]\times[K]_0}\right)}
			\right\rvert_{x =
			(X_{t_n}^{D,k}(i))_{(i,k)\in[D]\times\N_0}}\right]\\
			&\quad-
			\E\left[\prod_{j=1}^{n}\psi^D_j\left(\bigl(X_{t_j}^{D,k}(i)
			\bigr)_{(i,k)\in[D]\times[K]_0}\right)\right.\\
			&\qquad \left.\left.\phantom{\prod_{j=1}^n}\times\!\left.\Exp*{\psi^D_{n+1}\left(\bigl(X_{t_{n+1}-t_{n}}^{D,k,x}(i)
			\bigr)_{(i,k)\in[D]\times[K]_0}\right)}
			\right\rvert_{x =
			(X_{t_n}^{D,k}(i))_{(i,k)\in[D]\times\N_0}}\right]\right\vert
			= 0.
		\end{split}
	\end{equation}
	The Markov property yields for all $D\in\N$ that
	\begin{equation}
		\begin{split}
			\MoveEqLeft\E\left[\prod_{j=1}^{n}\psi^D_j\left(\bigl(X_{t_j}^{D,k}(i)
			\bigr)_{(i,k)\in[D]\times[K]_0}\right)
			\left.\Exp*{\psi^D_{n+1}\left(\bigl(X_{t_{n+1}-t_{n}}^{D,k,x}(i)
			\bigr)_{(i,k)\in[D]\times[K]_0}\right)}
			\right\rvert_{x =
			(X_{t_n}^{D,k}(i))_{(i,k)\in[D]\times\N_0}}\right]\\
			&=
			\Exp*{\prod_{j=1}^{n+1}\psi^D_j\left(\bigl(X_{t_j}^{D,k}(i)
			\bigr)_{(i,k)\in[D]\times[K]_0}\right)}.
		\end{split}
		\label{eq:fidis_markovX}
	\end{equation}
	Combining~\eqref{eq:fidis_indhypo} through~\eqref{eq:fidis_markovX} proves
	the induction step $\N \ni n \to n+1$ and hence finishes the proof of
	Lemma~\ref{l:replace_induction}.
\end{proof}
The following lemma is the main result of
Subsection~\ref{ss:reduction_loop_free} and shows that the migration level processes
and the loop-free processes have the same limit as
$D\to\infty$.
\begin{lemma}[Migration level and loop-free processes have the same limit]
	\label{l:replace_fidi}
	Assume that \textup{Setting~\ref{set:initial_dist_level0}} holds, let
	$n\in\N$, let $\phi_1,\dotsc,\phi_n \in C^2([0,1],\R)$ with the property
	that for all $j\in[n]$ it holds that $\phi_j(0) =
	0$, let $\psi \in C^2_b(\R,\R)$, and let $t_1,\dotsc,t_n \in [0,\infty)$
	with $t_1 < \dotsb < t_n$.  Then it
	holds that
	\begin{equation}
		\lim_{D\to\infty} \abs*{ \Exp*{\prod_{j=1}^n\psi\left(
		\sum_{i=1}^D\sum_{k\in\N_0} \phi_j\bigl(
		X_{t_j}^{D,k}(i) \bigr) \right)}
		-
		\Exp*{\prod_{j=1}^n\psi\left(
		\sum_{i=1}^D\sum_{k\in\N_0} \phi_j\bigl(
		Z_{t_j}^{D,k}(i) \bigr) \right)}} = 0.
		\label{eq:replace_fidi}
	\end{equation}
\end{lemma}
\begin{proof}
	In a first step, we are going to reduce the considerations to
	$k\in[K]_0$ for finite $K\in\N_0$.
	The assumptions on $\phi_1,\dotsc,\phi_n$, and $\psi$ imply the
	existence of constants $L_\phi, L_\psi \in [0,\infty)$ such that it holds for all
	$j\in[n]$ and all $x\in [0,1]$ that
	$\abs{\phi_j(x)} \leq L_\phi x$ and for all $x_1,\dotsc,x_n \in \R$ and all
	$y_1,\dotsc,y_n \in \R$ that $\abs{\prod_{j=1}^n \psi(x_j) -
	\prod_{j=1}^n \psi(y_j)} \leq L_\psi \sum_{j=1}^n\abs{x_j-y_j}$.
	It follows for all $K\in\N_0$ that
	\begin{equation}
		\begin{split}
			\MoveEqLeft\sup_{D\in\N}\,\abs*{\Exp*{\prod_{j=1}^n\psi\left(
			\sum_{i=1}^D\sum_{k\in\N_0} \phi_j\bigl( X_{t_j}^{D,k}(i) \bigr) \right)}
			- \Exp*{\prod_{j=1}^n\psi\left( \sum_{i=1}^D\sum_{k=0}^K \phi_j\bigl(
			X_{t_j}^{D,k}(i) \bigr) \right)} } \\
			&\leq L_\psi L_\phi \sum_{j=1}^n\sum_{k=K+1}^\infty
			\sup_{D\in\N}\Exp*{\sum_{i=1}^D X_{t_j}^{D,k}(i)}.
		\end{split}
		\label{eq:ess_fidi_X}
	\end{equation}
	%and
	%\begin{equation}
	%	\begin{split}
	%		\MoveEqLeft\sup_{D\in\N}\,\abs*{\Exp*{\prod_{j=1}^n\psi_j\left(
	%		\sum_{i=1}^D\sum_{k\in\N_0} \phi\bigl( Z_{t_j}^{D,k}(i) \bigr) \right)}
	%		- \Exp*{\prod_{j=1}^n\psi_j\left( \sum_{i=1}^D\sum_{k=0}^K \phi\bigl(
	%		Z_{t_j}^{D,k}(i) \bigr) \right)} } \\
	%		&\leq L_\psi L_\phi \sum_{j=1}^n\sum_{k=K+1}^\infty
	%		\sup_{D\in\N}\Exp*{\sum_{i=1}^D Z_{t_j}^{D,k}(i)}.
	%	\end{split}
	%	\label{eq:ess_fidi_Z}
	%\end{equation}
	The right-hand side of~\eqref{eq:ess_fidi_X}
	converges to zero as $K\to \infty$ by
	Lemma~\ref{l:initial_dist_level0}
	and Lemma~\ref{l:essentially_finitely_many_generations}. The analogous
	statement holds when $X_{t_j}^{D,k}(i)$ is replaced by
	$Z_{t_j}^{D,k}(i)$ in~\eqref{eq:ess_fidi_X}.
	To prove~\eqref{eq:replace_fidi} it therefore suffices to show for all
	$K\in\N_0$ that
	\begin{equation}
		\lim_{D\to\infty} \abs*{ \Exp*{\prod_{j=1}^n\psi\left( \sum_{i=1}^D\sum_{k=0}^K
		\phi_j\bigl(
		X_{t_j}^{D,k}(i) \bigr) \right)}
		-
		\Exp*{\prod_{j=1}^n\psi\left( \sum_{i=1}^D\sum_{k=0}^K
		\phi_j\bigl(
		Z_{t_j}^{D,k}(i) \bigr) \right)}} = 0.
		\label{eq:replace_fidi_suffices}
	\end{equation}
	We fix $K\in\N_0$ for the rest of the proof. For every $j\in[n]$ and
	$D\in\N$ we define the function
	$\psi^D_j\colon [0,1]^{[D]\times[K]_0} \to \R$ by
	\begin{equation}
		[0,1]^{[D]\times[K]_0} \ni
		(x_{i,k})_{(i,k)\in[D]\times[K]_0} = x \mapsto \psi^D_j(x) \defeq
		\psi\left( \sum_{i=1}^D\sum_{k=0}^K \phi_j(x_{i,k}) \right).
	\end{equation}
	It follows for all $j\in[n]$ that
	$\sup_{D\in\N}\norm{\psi^D_j}_\infty \leq
	\norm{\psi}_\infty$. Since $\phi_1,\dotsc,\phi_n$, and $\psi$ are
	twice continuously
	differentiable, it holds for all $j\in[n]$ and all $D\in\N$
	that $\psi^D_j\in
	C^2([0,1]^{[D]\times[K]_0},\R)$. Furthermore, the chain rule
	and the product rule imply for all $j\in[n]$, all $D\in\N$, all
	$(\tilde i, \tilde k), (\tilde j,\tilde l)\in
	[D]\times[K]_0$, and all $x\in
	[0,1]^{[D]\times[K]_0}$ that
	\begin{equation}
		\abs*{\frac{\partial \psi^D_j}{\partial x_{\tilde i, \tilde k}}(x)} =
		\abs*{\psi^\prime\left( \sum_{i=1}^D \sum_{k=0}^K
		\phi_j(x_{i,k})\right)\phi_j^\prime(x_{\tilde i, \tilde k})}
		\leq
		\norm{\psi^\prime}_\infty\norm{\phi_j^\prime}_\infty
	\end{equation}
	and
	\begin{equation}
		\begin{split}
			\abs*{\frac{\partial^2 \psi_j^D}{\partial x_{\tilde j, \tilde l} \partial x_{\tilde i,
			\tilde k}}(x)} &= \abs*{\psi^{\prime\prime}\left(
					\smallsum\limits_{i=1}^D \smallsum\limits_{k=0}^K
			\phi_j(x_{i,k})\right)\phi_j^\prime(x_{\tilde j, \tilde l})\phi_j^\prime(x_{\tilde
			i, \tilde k}) + \1_{(\tilde i, \tilde k) = (\tilde j, \tilde l)}
			\psi^\prime\left( \smallsum\limits_{i=1}^D \smallsum\limits_{k=0}^K
			\phi_j(x_{i,k})\right)\phi_j^{\prime\prime}(x_{\tilde i, \tilde k})}\\
			&\leq
			\norm{\psi^{\prime\prime}}_\infty\norm{\phi_j^\prime}_\infty^2+
			\norm{\psi^\prime}_\infty\norm{\phi_j^{\prime\prime}}_\infty.
		\end{split}
	\end{equation}
	It follows for all $j\in[n]$ that
	$\sup_{D\in\N}\norm{\psi_j^D}_{C^2} < \infty$. Then
	Lemma~\ref{l:replace_induction} shows~\eqref{eq:replace_fidi_suffices} which in turn
	finishes the proof of Lemma~\ref{l:replace_fidi}.
\end{proof}

\subsection{Proof of Theorem~\ref{thm:convergence}}
\label{ss:convergence_proof}
\begin{proof}[Proof of \textup{Theorem~\ref{thm:convergence}}]
	In a first step, we prove Theorem~\ref{thm:convergence} under the additional
	assumption that
	\begin{equation}
		\Exp*{\left(\sum_{i=1}^\infty X_0(i)\right)^2}<\infty.
		\label{eq:initial_second_moment_proof}
	\end{equation}
	Analogously to the proofs of
	Lemma~\ref{l:vanishing_immigration_weak_process} and
	Lemma~\ref{l:convergence_of_the_loop_free_process}, one shows that
	\begin{equation}
		\left\{ \left( \sum_{i=1}^D X_{t}^{D}(i)\delta_{X_{t}^{D}(i)} \right)_{t
		\in[0,\infty)} \colon D\in\N \right\}
		\label{eq:convergence_relcomp}
	\end{equation}
	is relatively compact.
	In the following, we identify the limit points
	of~\eqref{eq:convergence_relcomp} by proving convergence of
	finite-dimensional distributions.
	For that, fix $n\in\N$, fix $\varphi_1,\dotsc,\varphi_n\in C^2([0,1],\R)$,
	fix $\psi\in C^2_{b}(\R,\R)$, and fix
	$t_1,\dotsc,t_n\in[0,\infty)$ with $t_1<\dotsb< t_n$.
	For every $j\in[n]$ we define
	the function $\phi_j\colon [0,1]\to \R$ by $[0,1]\ni x \mapsto \phi_j(x) \defeq
	x\varphi_j(x)$.
	For every $D\in\N$ let $\{(X_t^{D,k}(i), W_t^k(i))_{t\in[0,\infty)} \colon
	(i,k)\in[D]\times\N_0\}$ be a weak solution of~\eqref{eq:XD_k}
	with initial distribution satisfying for
	all $i\in[D]$
	that $\MCL(X_0^{D,0}(i)) = \MCL(X_0(i))$ and for all
	$(i,k)\in[D]\times \N$ that $\MCL(X_0^{D,k}(i)) = \delta_0$
	and let $\{(Z_t^{D,k}(i))_{t\in[0,\infty)} \colon
	(i,k)\in[D]\times\N_0\}$ be a solution of~\eqref{eq:ZD_k}
	on the same probability space with Brownian motion given by the Brownian
	motion of the weak solution of~\eqref{eq:XD_k} and started in
	$(X_0^{D,k}(i))_{(i,k)\in[D]\times\N_0}$. Due to this and
	assumption~\eqref{eq:initial_second_moment_proof},
	Setting~\ref{set:migration_levels} and Setting~\ref{set:initial_dist_level0}
	are satisfied.
	Firstly, Lemma~\ref{l:decomposition} shows for all $D\in\N$ that
	\begin{equation}
		\Exp*{\prod_{j=1}^n\psi\left( \sum_{i=1}^D \phi_j\bigl(
		X_{t_j}^{D}(i) \bigr) \right)} =
		\Exp*{\prod_{j=1}^n\psi\left( \sum_{i=1}^D\phi_j\left(
		\sum_{k\in\N_0} X_{t_j}^{D,k}(i) \right) \right)}.
		\label{eq:conv_decomp}
	\end{equation}
	The
	calculation in the following two displays is analogous to that in the proof
	of~(4.111) in~\textcite[p.~34]{Hutzenthaler2012}.
	%We observe that for every sequence $(x^k)_{k\in\N_0}\subseteq [0,\infty)$
	%and every $\delta \in (0,\infty)$ it holds that
	%\begin{equation}
	%	\begin{split}
	%		\abs*{1-\sum_{k\in\N_0}\1_{\{x^k \geq \delta\}}} &\leq
	%		\1_{\cap_{m\in\N_0}\{x^m \leq \delta
	%		\}} +
	%		\1_{\cup_{m\in\N_0}\cup_{l\in\N_0\setminus\{m\}}(\{x^m\geq\delta\}\cap\{x^l\geq
	%		\delta\})}\sum_{k\in\N_0}\1_{x^k\geq \delta}\\
	%		&\leq \1_{\cap_{m\in\N_0}\{x^m\leq \delta\}} +
	%		\frac{1}{\delta^2}\sum_{k\in\N_0}x^k\sum_{l\in\N_0\setminus\{k\}}x^l.
	%	\end{split}
	%\end{equation}
	The assumptions on $\varphi_1,\dotsc,\varphi_n$ imply the existence of a
	constant $L_\phi \in[0,\infty)$ such that it holds for all
	$j\in[n]$ and all $x,y\in[0,1]$ that
	\begin{equation}
		\abs{\phi_j(x) - \phi_j(y)}\leq L_\phi \abs{x-y}.
		\label{eq:conv_Lipschitz_phi}
	\end{equation}
	From this we obtain for all
	$D\in\N$, all $j\in[n]$, all $t\in[0,\infty)$, and all $\delta\in(0,\infty)$ that
	\begin{equation}
		\begin{split}
			\MoveEqLeft\Exp*{\abs*{\sum_{i=1}^D \phi_j\left(
			\sum_{m\in\N_0}X_{t}^{D,m}(i)\right)
			\left(1-\sum_{k\in\N_0}\1_{\{X^{D,k}_t(i)\geq \delta\}}\right)}}\\
			&\leq L_\phi \Exp*{\sum_{i=1}^D \sum_{m\in\N_0}\bigl(X^{D,m}_t(i) \wedge
			\delta\bigr)} + \frac{L_\phi}{\delta^2}
			\Exp*{\sum_{i=1}^D\sum_{k\in\N_0}X^{D,k}_t(i)\sum_{l\in\N_0\setminus\{k\}}X^{D,l}_{t}(i)}.
		\end{split}
		\label{eq:conv_decomp_est}
	\end{equation}
	The fact that there exists a
	constant $L_\psi \in [0,\infty)$ such that it holds for all $x_1,\dotsc,x_n \in \R$
	and all $y_1,\dotsc,y_n \in \R$
	that $\abs{\prod_{j=1}^n \psi(x_j) -
	\prod_{j=1}^n \psi(y_j)} \leq L_\psi \sum_{j=1}^n\abs{x_j-y_j}$ together
	with~\eqref{eq:conv_Lipschitz_phi} and~\eqref{eq:conv_decomp_est} proves for all
	$D\in\N$ and all $\delta \in (0,1)$ that
	\begin{equation}
		\begin{split}
			\MoveEqLeft\abs*{\Exp*{\prod_{j=1}^n\psi\left( \sum_{i=1}^D \phi_j\left(
			\sum_{m\in\N_0}X_{t_j}^{D,m}(i) \right) \right)} -
			\Exp*{\prod_{j=1}^n\psi\left( \sum_{i=1}^D\sum_{k\in\N_0} \phi_j\bigl(
			X_{t_j}^{D,k}(i) \bigr) \right)}}\\
			&\leq L_\psi\sum_{j=1}^n
			\Exp*{\sum_{i=1}^D\sum_{k\in\N_0}\1_{\{X_{t_j}^{D,k}(i)\geq \delta\}}
			\abs*{\phi_j\left(\sum_{m\in\N_0}X_{t_j}^{D,m}(i)\right) -
			\phi_j\bigl(X_{t_j}^{D,k}(i)\bigr)}}\\
			&\quad+ L_\psi\sum_{j=1}^n\Exp*{\abs*{\sum_{i=1}^D \phi_j\left(
			\sum_{m\in\N_0}X_{t_j}^{D,m}(i)\right)
			\left(1-\sum_{k\in\N_0}\1_{\{X^{D,k}_t(i)\geq \delta\}}\right)}}\\
			&\quad+ L_\psi\sum_{j=1}^n \Exp*{\sum_{i=1}^D\sum_{k\in\N_0}
			\1_{\{X^{D,k}_{t_j}(i) < \delta\}}
			\abs[\big]{\phi_j\bigl(X^{D,k}_{t_j}(i)\bigr)}}\\
			&\leq \frac{2 L_\psi L_\phi}{\delta^2} \sum_{j=1}^n
			\Exp*{\sum_{i=1}^D\sum_{k\in\N_0}X^{D,k}_{t_j}(i)\sum_{m\in\N_0\setminus\{k\}}
			X^{D,m}_{t_j}(i)} \\
			&\quad + 2L_\psi L_\phi \sum_{j=1}^n \Exp*{\sum_{i=1}^D
			\sum_{k\in\N_0} \bigl(X^{D,k}_{t_j}(i) \wedge \delta \bigr)}.
		\end{split}
		\label{eq:conv_decomp_est2}
	\end{equation}
	%and Lemma~\ref{l:replace_fidi} implies that
	%\begin{equation}
	%	\lim_{D\to\infty}\Exp*{\prod_{j=1}^n\psi_j\left(
	%	\sum_{i=1}^D\sum_{k\in\N_0} \phi\bigl(
	%	X_{t_j}^{D,k}(i) \bigr) \right)} =
	%	\lim_{D\to\infty}\Exp*{\prod_{j=1}^n\psi_j\left(
	%	\sum_{i=1}^D\sum_{k\in\N_0} \phi\bigl(
	%	Z_{t_j}^{D,k}(i) \bigr) \right)}.
	%	\label{eq:conv_replace}
	%\end{equation}
	Lemma~\ref{l:initial_dist_level0} and
	Lemma~\ref{l:one_generation_per_island} ensure that the first summand on
	the right-hand side of~\eqref{eq:conv_decomp_est2} converges to zero as
	$D\to\infty$, while
	Lemma~\ref{l:initial_dist_level0} and Lemma~\ref{l:concentration} show that
	the second summand on the right-hand side of~\eqref{eq:conv_decomp_est2}
	converges to zero uniformly in
	$D\in\N$ as $\delta \to 0$.
	By letting first $D\to\infty$ and then $\delta \to
	0$, we therefore obtain
	from~\eqref{eq:conv_decomp_est2} that
	\begin{equation}
		\lim_{D\to\infty}\abs*{\Exp*{\prod_{j=1}^n\psi\left( \sum_{i=1}^D \phi_j\left(
		\sum_{k\in\N_0}X_{t_j}^{D,k}(i) \right) \right)}
		- \Exp*{\prod_{j=1}^n\psi\left(
		\sum_{i=1}^D\sum_{k\in\N_0} \phi_j\bigl(
		X_{t_j}^{D,k}(i) \bigr) \right)}} = 0.
		\label{eq:conv_decomp_pullout}
	\end{equation}
	Lemma~\ref{l:convergence_of_the_loop_free_process} shows that
	\begin{equation}
		\lim_{D\to\infty}\Exp*{\prod_{j=1}^n\psi\left(
		\sum_{i=1}^D\sum_{k\in\N_0} \phi_j\bigl(
		Z_{t_j}^{D,k}(i) \bigr) \right)} = \Exp*{\prod_{j=1}^n\psi\left( \int
		\phi_j(\eta_{t_j-s})
		\di \MCT(ds \otimes d\eta)\right)}.
		\label{eq:conv_loopfree}
	\end{equation}
	Combining~\eqref{eq:conv_decomp},~\eqref{eq:conv_decomp_pullout},
	Lemma~\ref{l:replace_fidi}, and~\eqref{eq:conv_loopfree} shows that
	\begin{equation}
		\lim_{D\to\infty}\Exp*{\prod_{j=1}^n\psi\left( \sum_{i=1}^D \phi_j\bigl(
		X_{t_j}^{D}(i) \bigr) \right)} = \Exp*{\prod_{j=1}^n\psi\left( \int
		\phi_j(\eta_{t_j-s})
		\di \MCT(ds \otimes d\eta)\right)}.
	\end{equation}
	This implies the convergence of finite-dimensional distributions
	of~\eqref{eq:convergence_relcomp} and proves Theorem~\ref{thm:convergence}
	under the additional assumption~\eqref{eq:initial_second_moment_proof}.

	It remains to prove Theorem~\ref{thm:convergence} in the case
	when~\eqref{eq:initial_second_moment_proof} fails to hold. 
	Fix a bounded
	continuous function $F\colon
	D([0,\infty),\MCM_\textup{f}([0,1])) \to \R$ for the rest of the proof.
	Then Setting~\ref{set:initial_moment} and the previous step imply that a.s.
	\begin{equation}
		\begin{split}
			\MoveEqLeft\lim_{D\to\infty}\Exp*{F\left( \left(\sum_{i=1}^D
			X_t^{D}(i)\delta_{X_t^{D}(i)}
			\right)_{t\in[0,\infty)}\right) \given \bigl(X_0^D(i)\bigr)_{i\in\N}}\\
			&=\Exp*{F\left( \left(\int \eta_{t-s}\delta_{\eta_{t-s}} \di\MCT(ds \otimes
			d\eta)\right)_{t\in[0,\infty)} \right) \given (X_0(i))_{i\in\N}}.
		\end{split}
	\end{equation}
	Then it follows from taking expectations and from the dominated convergence
	theorem that
	\begin{equation}
		\lim_{D\to\infty}\Exp*{F\left( \left(\sum_{i=1}^D
		X_t^{D}(i)\delta_{X_t^{D}(i)} \right)_{t\in[0,\infty)}\right)} =
		\Exp*{F\left( \left(\int \eta_{t-s}\delta_{\eta_{t-s}} \di\MCT(ds \otimes
		d\eta)\right)_{t\in[0,\infty)} \right)}.
	\end{equation}
	This finishes the proof of Theorem~\ref{thm:convergence}.
\end{proof}

\subsubsection*{Acknowledgments}
We thank an anonymous referee for many helpful comments.
This paper has been partially supported by the DFG Priority Program
``Probabilistic Structures in Evolution'' (SPP 1590), grant HU 1889/3-2.

\printbibliography
%\printbibliography[heading=bibintoc]
%\bibliographystyle{plainnat}
%\bibliography{bibliography}
\end{document}